\documentclass[12pt,reqno]{amsart}

\usepackage{amssymb,amsmath,graphicx}
\usepackage{booktabs}
\usepackage{pgf}
\usepackage{tikz}
\usepackage{longtable}
\usepackage{enumerate}
\usetikzlibrary{arrows,automata}

\newtheorem{theorem}{Theorem}[section]
\newtheorem{lemma}[theorem]{Lemma}
\newtheorem{corollary}[theorem]{Corollary}
\newtheorem{proposition}[theorem]{Proposition}

\theoremstyle{definition}
\newtheorem{definition}[theorem]{Definition}
\newtheorem{example}[theorem]{Example}
\newtheorem{remark}[theorem]{Remark}
\newtheorem{remarks}[theorem]{Remarks}
\newtheorem{notation}[theorem]{Notation}

\newcommand{\BB}{\mathcal{B}}
\newcommand{\PP}{\mathbb{P}}

\newcommand{\C}{\mathbb{C}}
\newcommand{\infnear}{\succ}
\newcommand{\prox}{\dasharrow}

\newcommand{\satel}{\odot}
\newcommand{\notsatel}{\hbox{$\mspace{5mu}\not\mspace{-5mu}\odot\mspace{6mu}$}}
\newcommand\Z[1]{\mathbb{Z}/#1\mathbb{Z}}

\DeclareMathOperator{\Aut}{Aut}
\DeclareMathOperator{\Cr}{Cr}
\DeclareMathOperator{\Bir}{Bir}
\DeclareMathOperator{\PGL}{PGL}
\DeclareMathOperator{\oq}{oql}
\DeclareMathOperator{\q}{ql}
\DeclareMathOperator{\lgth}{lgth}

\DeclareMathOperator{\id}{id}
\DeclareMathOperator{\mult}{mult}
\DeclareMathOperator{\h}{h}
\DeclareMathOperator{\load}{load}

\title{On the classification of cubic plane Cremona maps}
\author{Alberto Calabri}
\author{Thi Ngoc Giao Nguyen}

\begin{document}

\maketitle


\section{Introduction}

We work over the field $\C$ of complex numbers.

We denote by $\PP^2$ the projective plane and by $\Bir(\PP^2)$ the \emph{plane Cremona group}, that is the group of birational maps $\PP^2\dasharrow\PP^2$.
Recall that the celebrated Noether-Castelnuovo Theorem says that $\Bir(\PP^2)$ is generated by the automorphisms of $\PP^2$ and the elementary quadratic transformation
\begin{equation}\label{sigma_eq}
\sigma\colon \PP^2 \dasharrow \PP^2,
\qquad [x : y : z] \mapsto [yz : xz : xy].
\end{equation}
Note that a presentation of $\Bir(\PP^2)$ involving exactly these generators have been found only very recently by Urech and Zimmermann in \cite{U-Z}.

Let us say that  two plane Cremona maps $\varphi,\varphi'\colon\PP^2\dasharrow\PP^2$  are \emph{equivalent} if there exist two automorphisms $\alpha,\alpha'\in\Aut(\PP^2)$ such that
$
\varphi'=\alpha'\circ\varphi\circ\alpha.
$
The classification of equivalence classes of quadratic plane Cremona maps is very well-known from the beginning of the study of plane Cremona maps more than one hundred years ago.

Nonetheless, a classification of equivalence classes of cubic plane Cremona maps has been described only few years ago by 
Cerveau and D\'eserti in \cite{C-D}: they find 32 types of cubic plane Cremona maps, namely 27 types are a single map whereas 4 types are families of maps depending on one parameter and one type is a family of maps depending on two parameters. Their classification is based on the detailed analysis of those plane curves which are contracted by a cubic plane Cremona map.

However, it turns out that the classification in \cite{C-D} is not complete and it contains some inaccuracies, see Section \ref{sec comparison with CD} for a more detailed account:
\begin{itemize}
\item we found a map (our type 15 in Table \ref{table1}) that does not occur in their list;
\item we found that their type 17, that is a single map, should be replaced by a one-parameter set of maps (our type 28 in Table \ref{table1});
\item we found that their type 19 is equivalent to a particular case of their type 18;
\item we found that their type 31 is equivalent to a particular case of their type 30.
\end{itemize}

Our main result in this paper is a fine classification of equivalence classes of cubic plane Cremona maps.
Before stating our theorem, we need a bit of notation.

Let us set $\C^{**} = \mathbb{C}\setminus \{ 0,1 \}$ and let us define the following maps:
\begin{align*}
g_1, g_2 \colon \C^{**} \times \C^{**} & \rightarrow \C^{**} \times \C^{**},
\qquad g_1(a,b) = (b,a), \qquad g_2(a,b) = \left(\dfrac{1}{a},\dfrac{1}{b}\right).
\end{align*}
Therefore, $g_3:=g_2 \circ g_1 = g_1 \circ g_2$ is the map $(a,b) \mapsto ({1}/{b},{1}/{a})$. Clearly,
\begin{equation*} \label{set G}
G = \{ \id, g_1,g_2, g_3 \}
\end{equation*}
is a group, under the composition, which is isomorphic to $((\Z{2})^2,+)$.

For $a \neq b$ and $a,b \in \mathbb{C}^{**}$, let us denote by $S'$ the following set
\begin{equation*} \label{set S'}
\begin{aligned}
S' &= \left\{ (a,b),\left(\dfrac{a}{a-1}, \dfrac{a-b}{a-1}\right),\left(\dfrac{b}{b-1},\dfrac{b-a}{b-1}\right),\right. \\
&\phantom{ = \{ } \left.\left(\dfrac{a-b}{b(a-1)},\dfrac{1}{1-a}\right),
\left(\dfrac{b-a}{a(b-1)},\dfrac{1}{1-b}\right),
\left( \dfrac{a-1}{b-1},\dfrac{b(a-1)}{a(b-1)}\right)\right\}
\end{aligned}
\end{equation*}
and let us define 
\begin{equation} \label{set S}
S = \{ g(s) \mid g \in G \text{ and } s \in S' \}.
\end{equation}

\begin{theorem}\label{thm main1}
Any cubic plane Cremona map is equivalent to one of the maps in Table \ref{table1} at page \pageref{table1}, where the first 25 types are single maps, types 26-30 depend on one parameter $\gamma \neq 0,1$ and type 31 depends on two parameters $a,b$, where $a,b \neq 0,1$ and $a\neq b$. 

Two cubic plane Cremona maps of two different types are not equivalent.

Concerning the types depending on parameters:
\begin{itemize}
\item $\varphi_{26,\gamma}$, that is type 26 in Table \ref{table1} with parameter $\gamma\ne0,1$, is equivalent to  $\varphi_{26,\gamma'}$ if and only if either $\gamma' = \gamma$ or $\gamma' = {\gamma}/{(\gamma -1)}$;

\item  $\varphi_{27,\gamma}$, that is type 27 in Table \ref{table1} with parameter $\gamma \neq 0,1$, is equivalent to $\varphi_{27,\gamma'}$ if and only if either $\gamma' = \gamma$ or $\gamma' ={1}/{\gamma}$;

\item for $n \in \{28,29,30 \}$, the map $\varphi_{n,\gamma}$, that is type $n$ in Table \ref{table1} with parameter $\gamma \neq 0,1$, is equivalent to $\varphi_{n,\gamma'}$ if and only if 
$$\gamma' \in \left\{ \gamma,\dfrac{1}{\gamma},1-\gamma,\dfrac{1}{1-\gamma},\dfrac{\gamma}{\gamma-1},\dfrac{\gamma-1}{\gamma} \right\}.$$

\item $\varphi_{31,a,b}$, that is type 31 in Table \ref{table1} with two parameters $a, b\neq 0,1$, $a \ne b$, is equivalent to $\varphi_{31,a',b'}$ if and only if $(a',b') \in S$, where $S$ is defined in \eqref{set S}.

\end{itemize}
\end{theorem}

In Table \ref{table1} at page \pageref{table1}, the first column lists our type, the second column lists the formula of the maps, the third column lists the corresponding types in \cite{C-D}, cf.\ Section \ref{sec comparison with CD}, and finally the fourth column lists the types of the inverse maps.

\begin{table}[!htbp]
\centering
\begin{tabular}{|c|c|c|c|}
\hline
$\sharp$           & Map                    & \cite{C-D}  & Inv \\  \hline
\rule{0ex}{2.5ex}1 & $[xz^2+y^3:yz^2:z^3]$  & 1 & 1\\[0.25ex] \hline
\rule{0ex}{2.5ex}2 & $[x(x^2+yz): y^3: y(x^2+yz)]$& 20 & 8\\[0.25ex] \hline
\rule{0ex}{2.5ex}3 & $[xz^2: x^3+xyz: z^3]$ & 3 & 5\\[0.25ex] \hline
\rule{0ex}{2.5ex}4 & $[x^2z: x^3+z^3+xyz: xz^2]$ & 4 & 4\\[0.25ex] \hline
\rule{0ex}{2.5ex}5 & $[x^2z: x^2y+z^3: xz^2]$& 5 & 3\\[0.25ex] \hline
\rule{0ex}{2.5ex}6 & $[x^2(x-y): xy(x-y): xyz+y^3]$& 12 & 6\\[0.25ex] \hline
\rule{0ex}{2.5ex}7 & $[x(x^2+yz): y(x^2+yz): xy^2]$& 24 & 17\\[0.25ex] \hline
\rule{0ex}{2.5ex}8 & $[xyz: yz^2:z^3-x^2y]$& 6 & 2\\[0.25ex] \hline
\rule{0ex}{2.5ex}9 & $[y^2z: x(xz+y^2): y(xz+y^2)]$& 21 & 9\\[0.25ex] \hline
\rule{0ex}{2.5ex}10 & $[x^3: y^2z: xyz]$& 7 & 10\\[0.25ex] \hline
\rule{0ex}{2.5ex}11 & $[x(y^2+xz): y(y^2+xz): xyz]$& 22 & 18\\[0.25ex] \hline
\rule{0ex}{2.5ex}12  & $[xz^2: x^2y: z^3]$ & 2 & 12\\[0.25ex] \hline
\rule{0ex}{2.5ex}13 & $[x(y^2+xz): y(y^2+xz): xy^2]$& 23 & 20\\[0.25ex] \hline
\rule{0ex}{2.5ex}14 & $[x^3: x^2y: (x-y)yz]$& 11 & 15\\[0.25ex] \hline
\rule{0ex}{2.5ex}15 & $[x^2y: xy^2: (x-y)^2 z]$& $(\star)$ & 14\\[0.25ex] \hline
\rule{0ex}{2.5ex}16 & $[x(x^2+yz): y(x^2+yz): xy(x-y)]$& 28 & 24\\[0.25ex] \hline
\rule{0ex}{2.5ex}17 & $[xyz: y^2z: x(y^2-xz)]$& 10 & 7\\[0.25ex] \hline
\rule{0ex}{2.5ex}18 & $[x^2(y-z): xy(y-z): y^2z]$& 8 & 11\\[0.25ex] \hline
\rule{0ex}{2.5ex}19 & $[x(x^2+yz+xz): y(x^2+yz+xz): xyz]$& 26 & 19\\[0.25ex] \hline
\rule{0ex}{2.5ex}20 & $[x^2z: xyz: y^2(x-z)]$& 9 & 13\\[0.25ex] \hline
\rule{0ex}{2.5ex}21 & $[x(xy+xz+yz): y(xy+xz+yz): xyz]$& 25 & 21\\[0.25ex] \hline
\rule{0ex}{2.5ex}22 & $[xz(x+y): yz(x+y): xy^2]$& 13 & 22\\[0.25ex] \hline
\rule{0ex}{2.5ex}23 & $[x(x^2+xy+yz): y(x^2+xy+yz): xyz]$& 27 & 25\\[0.25ex] \hline
\rule{0ex}{2.5ex}24 & $[xyz : (y - x)yz : x(x - y)(y - z)]$& 15 & 16\\[0.25ex] \hline
\rule{0ex}{2.5ex}25 & $[x(x+y)(y+z):y(x+y)(y+z):xyz]$& 14 & 23\\[0.25ex] \hline
\rule{0ex}{2.5ex}26 & $[x(\gamma xz-\gamma y^2-xy+y^2): \gamma xy(z-y): \gamma y^2(z-x)]$& 29 & 26\\[0.25ex] \hline
\rule{0ex}{2.5ex}27 & $[\gamma x^2y:\gamma xy^2 : (x+y)(x+\gamma y)z]$& 16 & 27\\[0.25ex] \hline
\rule{0ex}{2.5ex}28 & $[xy(x-y):xz(y-\gamma x):z(y+\gamma x)(y-\gamma x)]$& $17^\dagger$ & 28\\[0.25ex] \hline
\rule{0ex}{2.5ex}29 & $[xy(x-y):x(xy-\gamma xy+\gamma xz-yz):{x}^{2}y-{\gamma}^{2}{x}^{2}y+{\gamma}^{2}{x}^{2}z-{y}^{2}z]$& 30 & 30\\[0.25ex] \hline
\rule{0ex}{2.5ex}30 & $[x(xy+\gamma xz-xz-\gamma {y}^{2}):\gamma xz(x-y):\gamma z(x-y)(x+y)]$& 18 & 29\\[0.25ex] \hline
& $[ax(-abxz+ab{y}^{2}-{b}^{2}xy+{b}^{2}xz+axy-a{y}^{2}):ax(-abxz+abyz+axy$&&\\
\raisebox{1ex}{31}& $-ayz-bxy+bxz): -{a}^{2}b{x}^{2}z+{a}^{2}b{y}^{2}z+{a}^{2}{x}^{2}y-{a}^{2}{y}^{2}z-{b}^{2}{x}^{2}y+{b}^{2}{x}^{2}z]$& \raisebox{1ex}{32} & \raisebox{1ex}{31}\\[0.25ex] \hline
\end{tabular}
\bigskip
\caption{Types of cubic plane Cremona maps.}\label{table1}
\end{table}

We then introduce the notion of quadratic length (ordinary quadratic length, resp.) of a plane Cremona map $\varphi$, that is the minimum number of quadratic maps (ordinary quadratic maps, resp.) needed to decompose $\varphi$.

Recall that Blanc and Furter in \cite{B-F} defined the \emph{length} of a plane Cremona map $\varphi$, that is the minimum number of \emph{de Jonqui\`eres} maps needed to decompose $\varphi$.

Using the above classification theorem, it is easy to compute the ordinary quadratic length and the quadratic length of all cubic plane Cremona maps:

\begin{theorem}\label{thm main2}
Plane Cremona maps equivalent to type 1 in Table \ref{table1} have quadratic length 3, while all other cubic plane Cremona maps have quadratic length 2.

A plane Cremona map equivalent to type $n$, $1\leq n \leq 31$,  in Table \ref{table1} has the respective ordinary quadratic length listed in the third column in Table \ref{table2} at page \pageref{table2}.
\end{theorem}

\begin{table}[!htbp]
\centering
\begin{tabular}{cc}
\begin{tabular}{|c|c|c|}
\hline

\rule{0ex}{2.5ex}$\sharp$ & Enriched weighted prox.\ graph & oq \\  \hline

1  &
\scalebox{0.5}{%
\begin{tikzpicture}[->,>=stealth',shorten >=2pt,auto,node distance=2.5cm,ultra thick,baseline=-1.25ex]
  \tikzstyle{every state}=[fill=white,draw=black,text=black]

  \node[state,draw=red,text=red] (A)                    {\huge 2};
  \node[state] (B) [right of=A] {\huge 1};
  \node[state] (C) [right of=B] {\huge 1};
  \node[state] (D) [right of=C] {\huge 1};
  \node[state] (E) [right of=D] {\huge 1};

\path (E) edge (D);
\path (D) edge (C);
\path (C) edge (B);
\path (B) edge (A);
\path (C) edge [bend right] (A);
\end{tikzpicture}%
} & 6
\\[0.5ex]
\hline
\rule{0ex}{3.0ex}2  &
\scalebox{0.5}{%
\begin{tikzpicture}[->,>=stealth',shorten >=2pt,auto,node distance=2.5cm,ultra thick,baseline=-1.25ex]
  \tikzstyle{every state}=[fill=white,draw=black,text=black]

  \node[state,draw=red,text=red] (A)                    {\huge 2};
  \node[state] (B) [right of=A] {\huge 1};
  \node[state] (C) [right of=B] {\huge 1};
  \node[state] (D) [right of=C] {\huge 1};
  \node[state] (E) [right of=D] {\huge 1};

\path (E) edge (D);
\path (D) edge (C);
\path (C) edge (B);
\path (B) edge (A);
\end{tikzpicture}%
}&5
\\[0.5ex]
\hline
3   &
\scalebox{0.5}{%
\begin{tikzpicture}[->,>=stealth',shorten >=2pt,auto,node distance=2.5cm,ultra thick,baseline=-1.25ex]
  \tikzstyle{every state}=[fill=white,draw=black,text=black]

  \node[state,draw=red,text=red] (A)                    {\huge 2};
  \node[state] (B) [right of=A] {\huge 1};
  \node[state] (C) [right of=B] {\huge 1};
  \node[state] (D) [right of=C] {\huge 1};
  \node[state] (E) [right of=D] {\huge 1};

\path (E) edge (D);
\path (D) edge (C);
\path (B) edge (A);
\path (C) edge [bend right] (A);
\end{tikzpicture}%
}&5
\\[0.5ex]
\hline
4   &
\scalebox{0.5}{%
\begin{tikzpicture}[->,>=stealth',shorten >=2pt,auto,node distance=2.5cm,ultra thick,baseline=-1.25ex]
  \tikzstyle{every state}=[fill=white,draw=black,text=black]

  \node[state,draw=red,text=red] (A)                    {\huge 2};
  \node[state] (B) [right of=A] {\huge 1};
  \node[state] (C) [right of=B] {\huge 1};
  \node[state] (D) [right of=C] {\huge 1};
  \node[state] (E) [right of=D] {\huge 1};

\path (E) edge (D);
\path (D) edge [bend right] (A);
\path (C) edge (B);
\path (B) edge (A);
\end{tikzpicture}%
}&4
\\[0.5ex]
\hline
5  &
\scalebox{0.5}{%
\begin{tikzpicture}[->,>=stealth',shorten >=2pt,auto,node distance=2.5cm,ultra thick,baseline=-1.25ex]
  \tikzstyle{every state}=[fill=white,draw=black,text=black]

  \node[state,draw=red,text=red] (A)                    {\huge 2};
  \node[state] (B) [right of=A] {\huge 1};
  \node[state] (C) [right of=B] {\huge 1};
  \node[state] (D) [right of=C] {\huge 1};
  \node[state,draw=red,text=red] (E) [right of=D] {\huge 1};

\path (D) edge (C);
\path (C) edge (B);
\path (B) edge (A);
\path (C) edge [bend right] (A);
\end{tikzpicture}%
}&5
\\[0.5ex]
\hline
6   &
\scalebox{0.5}{%
\begin{tikzpicture}[->,>=stealth',shorten >=2pt,auto,node distance=2.5cm,ultra thick,baseline=-1.25ex]
  \tikzstyle{every state}=[fill=white,draw=black,text=black]

  \node[state,draw=red,text=red] (A)                    {\huge 2};
  \node[state] (B) [right of=A] {\huge 1};
  \node[state] (C) [right of=B] {\huge 1};
  \node[state] (D) [right of=C] {\huge 1};
  \node[state,draw=red,text=red] (E) [right of=D] {\huge 1};

\path (D) edge (C);
\path (B) edge (A);
\path (C) edge [bend right] (A);
\end{tikzpicture}%
}&4
\\[0.5ex]
\hline
\rule{0ex}{3.0ex}7  &
\scalebox{0.5}{%
\begin{tikzpicture}[->,>=stealth',shorten >=2pt,auto,node distance=2.5cm,ultra thick,baseline=-1.25ex]
  \tikzstyle{every state}=[fill=white,draw=black,text=black]

  \node[state,draw=red,text=red] (A)                    {\huge 2};
  \node[state] (B) [right of=A] {\huge 1};
  \node[state] (C) [right of=B] {\huge 1};
  \node[state] (D) [right of=C] {\huge 1};
  \node[state,draw=red,text=red] (E) [right of=D] {\huge 1};

\path (D) edge (C);
\path (C) edge (B);
\path (B) edge (A);
\end{tikzpicture}%
}&4
\\[0.5ex]
\hline
\rule{0ex}{3.0ex}8  &
\scalebox{0.5}{%
\begin{tikzpicture}[->,>=stealth',shorten >=2pt,auto,node distance=2.5cm,ultra thick,baseline=-1.25ex]
  \tikzstyle{every state}=[fill=white,draw=black,text=black]

  \node[state,draw=red,text=red] (A)                    {\huge 2};
  \node[state,draw=red,text=red] (B) [right of=A] {\huge 1};
  \node[state] (C) [right of=B] {\huge 1};
  \node[state] (D) [right of=C] {\huge 1};
  \node[state] (E) [right of=D] {\huge 1};

\path (E) edge (D);
\path (D) edge (C);
\path (C) edge (B);
\path (D) edge [-,dashed,draw=blue,bend right] (C);
\path (C) edge [-,dashed,draw=blue,bend right] (B);
\end{tikzpicture}%
}&5
\\[0.5ex]
\hline
\rule{0ex}{3.0ex}9  &
\scalebox{0.5}{%
\begin{tikzpicture}[->,>=stealth',shorten >=2pt,auto,node distance=2.5cm,ultra thick,baseline=-1.25ex]
  \tikzstyle{every state}=[fill=white,draw=black,text=black]

  \node[state,draw=red,text=red] (A)                    {\huge 2};
  \node[state,draw=red,text=red] (B) [right of=A] {\huge 1};
  \node[state] (C) [right of=B] {\huge 1};
  \node[state] (D) [right of=C] {\huge 1};
  \node[state] (E) [right of=D] {\huge 1};

\path (E) edge (D);
\path (D) edge (C);
\path (C) edge (B);
\end{tikzpicture}%
}&4
\\[0.5ex]
\hline
\rule{0ex}{3.0ex}10  &
\scalebox{0.5}{%
\begin{tikzpicture}[->,>=stealth',shorten >=2pt,auto,node distance=2.5cm,ultra thick,baseline=-1.25ex]
  \tikzstyle{every state}=[fill=white,draw=black,text=black]

  \node[state,draw=red,text=red] (A)                    {\huge 2};
  \node[state] (B) [right of=A] {\huge 1};
  \node[state,draw=red,text=red] (C) [right of=B] {\huge 1};
  \node[state] (D) [right of=C] {\huge 1};
  \node[state] (E) [right of=D] {\huge 1};

\path (E) edge (D);
\path (D) edge (C);
\path (B) edge (A);
\path (D) edge [-,dashed,draw=blue,bend right] (C);
\path (E) edge [-,dashed,draw=blue,bend right] (D);
\end{tikzpicture}%
}&3
\\[0.5ex]
\hline
\rule{0ex}{3.0ex}11  &
\scalebox{0.5}{%
\begin{tikzpicture}[->,>=stealth',shorten >=2pt,auto,node distance=2.5cm,ultra thick,baseline=-1.25ex]
  \tikzstyle{every state}=[fill=white,draw=black,text=black]

  \node[state,draw=red,text=red] (A)                    {\huge 2};
  \node[state] (B) [right of=A] {\huge 1};
  \node[state,draw=red,text=red] (C) [right of=B] {\huge 1};
  \node[state] (D) [right of=C] {\huge 1};
  \node[state] (E) [right of=D] {\huge 1};

\path (E) edge (D);
\path (D) edge (C);
\path (B) edge (A);
\end{tikzpicture}%
}&3
\\[0.5ex]
\hline
12  &
\scalebox{0.5}{%
\begin{tikzpicture}[->,>=stealth',shorten >=2pt,auto,node distance=2.5cm,ultra thick,baseline=-1.25ex]
  \tikzstyle{every state}=[fill=white,draw=black,text=black]

  \node[state,draw=red,text=red] (A)                    {\huge 2};
  \node[state] (B) [right of=A] {\huge 1};
  \node[state] (C) [right of=B] {\huge 1};
  \node[state,draw=red,text=red] (D) [right of=C] {\huge 1};
  \node[state] (E) [right of=D] {\huge 1};

\path (E) edge (D);
\path (C) edge (B);
\path (B) edge (A);
\path (C) edge [bend right] (A);
\end{tikzpicture}%
}&3
\\[0.5ex]
\hline
\rule{0ex}{3.0ex}13  &
\scalebox{0.5}{%
\begin{tikzpicture}[->,>=stealth',shorten >=2pt,auto,node distance=2.5cm,ultra thick,baseline=-1.25ex]
  \tikzstyle{every state}=[fill=white,draw=black,text=black]

  \node[state,draw=red,text=red] (A)                    {\huge 2};
  \node[state] (B) [right of=A] {\huge 1};
  \node[state] (C) [right of=B] {\huge 1};
  \node[state,draw=red,text=red] (D) [right of=C] {\huge 1};
  \node[state] (E) [right of=D] {\huge 1};

\path (E) edge (D);
\path (C) edge (B);
\path (B) edge (A);
\end{tikzpicture}%
}&3
\\[0.5ex]
\hline
14  &
\scalebox{0.5}{%
\begin{tikzpicture}[->,>=stealth',shorten >=2pt,auto,node distance=2.5cm,ultra thick,baseline=-1.25ex]
  \tikzstyle{every state}=[fill=white,draw=black,text=black]

  \node[state,draw=red,text=red] (A)                    {\huge 2};
  \node[state] (B) [right of=A] {\huge 1};
  \node[state] (C) [right of=B] {\huge 1};
  \node[state,draw=red,text=red] (D) [right of=C] {\huge 1};
  \node[state] (E) [right of=D] {\huge 1};

\path (E) edge (D);
\path (B) edge (A);
\path (C) edge [bend right] (A);
\end{tikzpicture}%
}&3
\\[0.5ex]
\hline
15  &
\scalebox{0.5}{%
\begin{tikzpicture}[->,>=stealth',shorten >=2pt,auto,node distance=2.5cm,ultra thick,baseline=-1.25ex]
  \tikzstyle{every state}=[fill=white,draw=black,text=black]

  \node[state,draw=red,text=red] (A)                    {\huge 2};
  \node[state] (B) [right of=A] {\huge 1};
  \node[state] (C) [right of=B] {\huge 1};
  \node[state,draw=red,text=red] (D) [right of=C] {\huge 1};
  \node[state,draw=red,text=red] (E) [right of=D] {\huge 1};

\path (C) edge (B);
\path (B) edge (A);
\path (C) edge [bend right] (A);
\end{tikzpicture}%
}&3
\\[0.5ex]
\hline
\end{tabular}
&
\raisebox{0.0ex}{\begin{tabular}{|c|c|c|}
\hline

\rule{0ex}{2.5ex}$\sharp$ & Enriched weighted prox.\ graph & oq \\  \hline

\rule{0ex}{3.0ex}16  &
\scalebox{0.5}{%
\begin{tikzpicture}[->,>=stealth',shorten >=2pt,auto,node distance=2.5cm,ultra thick,baseline=-1.25ex]
  \tikzstyle{every state}=[fill=white,draw=black,text=black]

  \node[state,draw=red,text=red] (A)                    {\huge 2};
  \node[state] (B) [right of=A] {\huge 1};
  \node[state] (C) [right of=B] {\huge 1};
  \node[state,draw=red,text=red] (D) [right of=C] {\huge 1};
  \node[state,draw=red,text=red] (E) [right of=D] {\huge 1};

\path (C) edge (B);
\path (B) edge (A);
\end{tikzpicture}%
}&3
\\[0.5ex]
\hline
\rule{0ex}{3.0ex}17   &
\scalebox{0.5}{%
\begin{tikzpicture}[->,>=stealth',shorten >=2pt,auto,node distance=2.5cm,ultra thick,baseline=-1.25ex]
  \tikzstyle{every state}=[fill=white,draw=black,text=black]

  \node[state,draw=red,text=red] (A)                    {\huge 2};
  \node[state,draw=red,text=red] (B) [right of=A] {\huge 1};
  \node[state,draw=red,text=red] (C) [right of=B] {\huge 1};
  \node[state] (D) [right of=C] {\huge 1};
  \node[state] (E) [right of=D] {\huge 1};

\path (E) edge (D);
\path (D) edge (C);
\path (D) edge [-,dashed,draw=blue,bend right] (C);
\path (C) edge  [-,dashed,draw=blue,bend right] (B);
\end{tikzpicture}%
}&4
\\[0.5ex]
\hline
\rule{0ex}{3.0ex}18  &
\scalebox{0.5}{%
\begin{tikzpicture}[->,>=stealth',shorten >=2pt,auto,node distance=2.5cm,ultra thick,baseline=-1.25ex]
  \tikzstyle{every state}=[fill=white,draw=black,text=black]

  \node[state,draw=red,text=red] (A)                    {\huge 2};
  \node[state,draw=red,text=red] (B) [right of=A] {\huge 1};
  \node[state,draw=red,text=red] (C) [right of=B] {\huge 1};
  \node[state] (D) [right of=C] {\huge 1};
  \node[state] (E) [right of=D] {\huge 1};

\path (E) edge (D);
\path (D) edge (C);
\path (D) edge [-,dashed,draw=blue,bend right] (C);
\path (E) edge [-,dashed,draw=blue,bend right] (D);
\end{tikzpicture}%
}&3
\\[0.5ex]
\hline
\rule{0ex}{3.0ex}19  &
\scalebox{0.5}{%
\begin{tikzpicture}[->,>=stealth',shorten >=2pt,auto,node distance=2.5cm,ultra thick,baseline=-1.25ex]
  \tikzstyle{every state}=[fill=white,draw=black,text=black]

  \node[state,draw=red,text=red] (A)                    {\huge 2};
  \node[state,draw=red,text=red] (B) [right of=A] {\huge 1};
  \node[state,draw=red,text=red] (C) [right of=B] {\huge 1};
  \node[state] (D) [right of=C] {\huge 1};
  \node[state] (E) [right of=D] {\huge 1};

\path (E) edge (D);
\path (D) edge (C);
\end{tikzpicture}%
}&3
\\[0.5ex]
\hline
\rule{0ex}{3.0ex}20  &
\scalebox{0.5}{%
\begin{tikzpicture}[->,>=stealth',shorten >=2pt,auto,node distance=2.5cm,ultra thick,baseline=-1.25ex]
  \tikzstyle{every state}=[fill=white,draw=black,text=black]

  \node[state,draw=red,text=red] (A)                    {\huge 2};
  \node[state,draw=red,text=red] (B) [right of=A] {\huge 1};
  \node[state] (C) [right of=B] {\huge 1};
  \node[state,draw=red,text=red] (D) [right of=C] {\huge 1};
  \node[state] (E) [right of=D] {\huge 1};

\path (E) edge (D);
\path (C) edge (B);
\path (D) edge [-,dashed,draw=blue,bend right] (C);
\path (C) edge  [-,dashed,draw=blue,bend right] (B);
\end{tikzpicture}%
}&3
\\[0.5ex]
\hline
\rule{0ex}{3.0ex}21  &
\scalebox{0.5}{%
\begin{tikzpicture}[->,>=stealth',shorten >=2pt,auto,node distance=2.5cm,ultra thick,baseline=-1.25ex]
  \tikzstyle{every state}=[fill=white,draw=black,text=black]

  \node[state,draw=red,text=red] (A)                    {\huge 2};
  \node[state,draw=red,text=red] (B) [right of=A] {\huge 1};
  \node[state] (C) [right of=B] {\huge 1};
  \node[state,draw=red,text=red] (D) [right of=C] {\huge 1};
  \node[state] (E) [right of=D] {\huge 1};

\path (E) edge (D);
\path (C) edge (B);
\end{tikzpicture}%
}&2
\\[0.5ex]
\hline
\rule{0ex}{3.0ex}22  &
\scalebox{0.5}{%
\begin{tikzpicture}[->,>=stealth',shorten >=2pt,auto,node distance=2.5cm,ultra thick,baseline=-1.25ex]
  \tikzstyle{every state}=[fill=white,draw=black,text=black]

  \node[state,draw=red,text=red] (A)                    {\huge 2};
  \node[state] (B) [right of=A] {\huge 1};
  \node[state,draw=red,text=red] (C) [right of=B] {\huge 1};
  \node[state,draw=red,text=red] (D) [right of=C] {\huge 1};
  \node[state] (E) [right of=D] {\huge 1};

\path (E) edge (D);
\path (B) edge (A);
\path (E) edge [-,dashed,draw=blue,bend right] (D);
\path (D) edge  [-,dashed,draw=blue,bend right] (C);
\end{tikzpicture}%
}&3
\\[0.5ex]
\hline
\rule{0ex}{3.0ex}23  &
\scalebox{0.5}{%
\begin{tikzpicture}[->,>=stealth',shorten >=2pt,auto,node distance=2.5cm,ultra thick,baseline=-1.25ex]
  \tikzstyle{every state}=[fill=white,draw=black,text=black]

  \node[state,draw=red,text=red] (A)                    {\huge 2};
  \node[state] (B) [right of=A] {\huge 1};
  \node[state,draw=red,text=red] (C) [right of=B] {\huge 1};
  \node[state,draw=red,text=red] (D) [right of=C] {\huge 1};
  \node[state] (E) [right of=D] {\huge 1};

\path (E) edge (D);
\path (B) edge (A);
\end{tikzpicture}%
}&2
\\[0.5ex]
\hline
\rule{0ex}{3.0ex}24  &
\scalebox{0.5}{%
\begin{tikzpicture}[->,>=stealth',shorten >=2pt,auto,node distance=2.5cm,ultra thick,baseline=-1.25ex]
  \tikzstyle{every state}=[fill=white,draw=black,text=black]

  \node[state,draw=red,text=red] (A)                    {\huge 2};
  \node[state,draw=red,text=red] (B) [right of=A] {\huge 1};
  \node[state,draw=red,text=red] (C) [right of=B] {\huge 1};
  \node[state,draw=red,text=red] (D) [right of=C] {\huge 1};
  \node[state] (E) [right of=D] {\huge 1};

\path (E) edge (D);
\path (D) edge [-,dashed,draw=blue,bend right] (C);
\path (C) edge  [-,dashed,draw=blue,bend right] (B);
\end{tikzpicture}%
}&3
\\[0.5ex]
\hline
\rule{0ex}{3.0ex}25  &
\scalebox{0.5}{%
\begin{tikzpicture}[->,>=stealth',shorten >=2pt,auto,node distance=2.5cm,ultra thick,baseline=-1.25ex]
  \tikzstyle{every state}=[fill=white,draw=black,text=black]

  \node[state,draw=red,text=red] (A)                    {\huge 2};
  \node[state,draw=red,text=red] (B) [right of=A] {\huge 1};
  \node[state,draw=red,text=red] (C) [right of=B] {\huge 1};
  \node[state,draw=red,text=red] (D) [right of=C] {\huge 1};
  \node[state] (E) [right of=D] {\huge 1};

\path (E) edge (D);
\path (D) edge [-,dashed,draw=blue,bend right] (C);
\path (E) edge [-,dashed,draw=blue,bend right] (D);
\end{tikzpicture}%
}&2
\\[0.5ex]
\hline
\rule{0ex}{3.0ex}26  &
\scalebox{0.5}{%
\begin{tikzpicture}[->,>=stealth',shorten >=2pt,auto,node distance=2.5cm,ultra thick,baseline=-1.25ex]
  \tikzstyle{every state}=[fill=white,draw=black,text=black]

  \node[state,draw=red,text=red] (A)                    {\huge 2};
  \node[state,draw=red,text=red] (B) [right of=A] {\huge 1};
  \node[state,draw=red,text=red] (C) [right of=B] {\huge 1};
  \node[state,draw=red,text=red] (D) [right of=C] {\huge 1};
  \node[state] (E) [right of=D] {\huge 1};
  \path (E) edge (D);
\end{tikzpicture}%
}&2
\\[0.5ex]
\hline
\rule{0ex}{3.0ex}27  &
\scalebox{0.5}{%
\begin{tikzpicture}[->,>=stealth',shorten >=2pt,auto,node distance=2.5cm,ultra thick,baseline=-1.25ex]
  \tikzstyle{every state}=[fill=white,draw=black,text=black]

  \node[state,draw=red,text=red] (A)                    {\huge 2};
  \node[state] (B) [right of=A] {\huge 1};
  \node[state] (C) [right of=B] {\huge 1};
  \node[state,draw=red,text=red] (D) [right of=C] {\huge 1};
  \node[state,draw=red,text=red] (E) [right of=D] {\huge 1};

\path (B) edge (A);
\path (C) edge [bend right] (A);
\end{tikzpicture}%
}&2
\\[0.5ex]
\hline
\rule{0ex}{3.0ex}28  & 
\scalebox{0.5}{%
\begin{tikzpicture}[->,>=stealth',shorten >=2pt,auto,node distance=2.5cm,ultra thick,baseline=-1.25ex]
  \tikzstyle{every state}=[fill=white,draw=black,text=black]

  \node[state,draw=red,text=red] (A)                    {\huge 2};
  \node[state] (B) [right of=A] {\huge 1};
  \node[state,draw=red,text=red] (C) [right of=B] {\huge 1};
  \node[state,draw=red,text=red] (D) [right of=C] {\huge 1};
  \node[state,draw=red,text=red] (E) [right of=D] {\huge 1};

\path (B) edge (A);
\path (E) edge [-,dashed,draw=blue,bend right] (D);
\path (D) edge  [-,dashed,draw=blue,bend right] (C);
\end{tikzpicture}%
}&3
\\[0.5ex]
\hline
\rule{0ex}{3.0ex}29  &
\scalebox{0.5}{%
\begin{tikzpicture}[->,>=stealth',shorten >=2pt,auto,node distance=2.5cm,ultra thick,baseline=-1.25ex]
  \tikzstyle{every state}=[fill=white,draw=black,text=black]

  \node[state,draw=red,text=red] (A)                    {\huge 2};
  \node[state] (B) [right of=A] {\huge 1};
  \node[state,draw=red,text=red] (C) [right of=B] {\huge 1};
  \node[state,draw=red,text=red] (D) [right of=C] {\huge 1};
  \node[state,draw=red,text=red] (E) [right of=D] {\huge 1};
  \path (B) edge (A);
\end{tikzpicture}%
}&2
\\[0.5ex]
\hline
\rule{0ex}{3.0ex}30  &
\scalebox{0.5}{%
\begin{tikzpicture}[->,>=stealth',shorten >=2pt,auto,node distance=2.5cm,ultra thick,baseline=-1.25ex]
  \tikzstyle{every state}=[fill=white,draw=black,text=black]

  \node[state,draw=red,text=red] (A)                    {\huge 2};
  \node[state,draw=red,text=red] (B) [right of=A] {\huge 1};
  \node[state,draw=red,text=red] (C) [right of=B] {\huge 1};
  \node[state,draw=red,text=red] (D) [right of=C] {\huge 1};
  \node[state,draw=red,text=red] (E) [right of=D] {\huge 1};

\path (E) edge [-,dashed,draw=blue,bend right] (D);
\path (D) edge  [-,dashed,draw=blue,bend right] (C);
\end{tikzpicture}%
}&2
\\[0.5ex]
\hline
\rule{0ex}{3.0ex}31  &
\scalebox{0.5}{%
\begin{tikzpicture}[->,>=stealth',shorten >=2pt,auto,node distance=2.5cm,ultra thick,baseline=-1.25ex]
  \tikzstyle{every state}=[fill=white,draw=black,text=black]

  \node[state,draw=red,text=red] (A)                    {\huge 2};
  \node[state,draw=red,text=red] (B) [right of=A] {\huge 1};
  \node[state,draw=red,text=red] (C) [right of=B] {\huge 1};
  \node[state,draw=red,text=red] (D) [right of=C] {\huge 1};
  \node[state,draw=red,text=red] (E) [right of=D] {\huge 1};
\end{tikzpicture}%
}&2
\\[0.5ex]
\hline
\end{tabular}}
\end{tabular}
\bigskip
\caption{Enriched weighted proximity graphs and ordinary quadratic lengths of cubic plane Cremona maps.}\label{table2}
\end{table}

Our classification theorem is mainly based on the study of \emph{enriched weighted proximity graphs} of the base points of the homaloidal net defining a plane Cremona map, that we introduce in Section \ref{weighted}.

Let us briefly explain the contents of this paper.

In Section 2 we fix notation and recall the definition of infinitely near points.

In Section 3 we introduce special coordinates in order to deal with infinitely near points that we call \emph{standard coordinates}.

In Section 4 we consider plane conics passing through five proper or infinitely near points.

In Section 5 we recall the definition of length of a plane Cremona map and we introduce the notion of quadratic length and ordinary quadratic length of a plane Cremona map.

In Section 6 we introduce the key notion that we use to classify cubic plane Cremona maps, that is the enriched weighted proximity graph of the base points of the homaloidal net that defines a plane Cremona map. This graph takes into account the proximity relations among the base points, their multiplicities and their relative position.

In Section 7 we define the height of a base point of a plane Cremona map, that allows us to give a lower bound on the ordinary quadratic length of a plane Cremona map.

In Section 8 we compare our classification with the previous one by Cerveau and D\'eserti in \cite{C-D}.

In Section 9 we prove our main classification theorem.

In Section 10 we compute the quadratic length and ordinary quadratic length of all cubic plane Cremona maps.

The results contained in this paper are part of the second author's Ph.D.\ Thesis \cite{thesis}.

The authors warmly thank J\'er\'emy Blanc and Ciro Ciliberto for useful discussions.

\section{Notation and infinitely near points}

\begin{notation}\label{not:1}
Any non-zero complex number $z$ can be written uniquely as follows
$$z=re^{i\theta}=r\left(\cos(\theta)+i\sin(\theta)\right), \quad \textrm{ with } r > 0, \: \textrm{ and } \theta \in [0,2\pi).$$
The angle $\theta$ is called the \emph{argument} of $z$ and the real number $r$ is the \emph{norm} of $z$.\\
Any non-zero complex number $z = r\left(\cos(\theta)+i\sin(\theta)\right)$ has two square roots, namely
$$z_0 = \sqrt{r}\left[\cos\left(\dfrac{\theta}{2}\right)+i\sin\left(\dfrac{\theta}{2}\right)\right], \qquad z_1=-z_0.$$
From now on, we denote $z_0$ by $\sqrt{z}$ and $z_1$ by $-\sqrt{z}$.

For any $t \in \mathbb{C}$ such that $t^2 \neq 4$, set $t^{\bullet} = \sqrt{t^2-4}$, $t_{+} = (t+ t^{\bullet})/{2}$ and $t_{-} = (t - t^{\bullet})/{2}$, that is, $t_{\pm}$ are the roots of the equation $x^2-tx +1 =0$. Note that, if $t^2 \neq 4$ then $t_+ \neq t_-$ and $t_+,t_- \neq 0$.
\end{notation}

Next, we recall some very well known facts about the projective plane and plane curves.

\begin{lemma} \label{lem: 4 points lemma} Let  $p_1,p_2,p_3,p_4 \in \mathbb{P}^2$ be four points such that no three of them are collinear. Then, there exists a unique automorphism of $\mathbb{P}^2$, that maps $e_1 = [1:0:0], e_2 = [0:1:0], e_3 = [0:0:1], e_4 = [1:1:1]$ to $p_1,p_2,p_3, p_4$, respectively.
\end{lemma}
\begin{proof}
See, e.g., $\S11.2$ in \cite{CGG}.
\end{proof}

\begin{lemma}
Let $p_1,p_2,p_3,p_4,p_5 \in \mathbb{P}^2$ be five points such that no three of them are collinear. Then, there exists a unique irreducible conic passing through $p_1,\ldots,p_5$.
\end{lemma}

\begin{proof}
See, e.g., $\S5.2$ in \cite{K-L-P-W}.
\end{proof}

\begin{lemma}[cf.\ Lemma 1.2.3 in \cite{MN}]\label{lm two_conic_equi}
Let $C_1$ and $C_2$ be two irreducible conics. Then, there exists an automorphism $\alpha$ of $\PP^2$ such that $\alpha(C_1)=C_2$. \end{lemma}
\begin{proof}
It suffices to show that there exists an automorphism of $\PP^2$ that maps $C_1$ to the conic $C_0\colon xz-y^2 =0$.
Choose three distinct points $p_1, p_2, p_3$ of $C_1$. Let $p_4$ be the intersection point of the tangent line to $C_1$ at $p_1$ with the tangent line to $C_1$ at $p_2$.
Clearly, no three among $p_1,p_2,p_3, p_4$ are collinear. Therefore, by Lemma \ref{lem: 4 points lemma}, there exists an automorphism $\beta$ of $\PP^2$ that maps $p_1, p_2, p_3, p_4$ to $e_1, e_3, e_4, e_2$, respectively. Hence, $\beta(C_1)=C_0$.
\end{proof}

The proof of the previous lemma also shows the following:

\begin{lemma}\label{lm two_conic_equi2}
Let $n\in\{1,2,3\}$.
Let $C_1$ and $C_2$ be two irreducible conics.
Let $p_1, \ldots, p_n$ be distinct points of $C_1$ and let $q_1, \ldots, q_n$ be distinct points of $C_2$.
Then, there exists an automorphism $\alpha$ of $\PP^2$ such that $\alpha(C_1)=C_2$ and $\alpha(p_i)=q_i$, $i=1,\ldots,n$.
\end{lemma}

Let us recall the definition of the bubble space of $\PP^2$, useful to define infinitely near points and related properties.

\begin{definition}[{cf.\ \cite[\S7.3.2]{Dolgachev}}]
Let us denote by $\BB(\PP^2)$ the so-called \emph{bubble space} of $\PP^2$, which is defined as follows.
Consider all surfaces $X$ \emph{above} $\PP^2$, i.e.\ all surfaces $X$ such that there exists a birational morphism $X\to\PP^2$.
If $X_1,X_2$ are two surfaces above $\PP^2$, say $\pi_1\colon X_1\to\PP^2$ and $\pi_2\colon X_2\to\PP^2$ are birational morphisms,
one identifies $p_1\in X_1$ with $p_2\in X_2$ if the birational map $(\pi_2)^{-1}\circ \pi_1\colon X_1\dasharrow X_2$ is a local isomorphism at $p_1$, that sends $p_1$ to $p_2$.
The bubble space $\BB(\PP^2)$ is the union of all points of all surfaces above $\PP^2$ modulo the equivalence relation generated by these identifications.

For any birational morphism $X\to\PP^2$, there is an injective map $X\to\BB(\PP^2)$, therefore we will identify points of $X$ with their images in $\BB(\PP^2)$.

One says that $p_1\in\BB(\PP^2)$ is \emph{infinitely near} $p_2\in\BB(\PP^2)$, say $p_1\in X_1$ and $p_2\in X_2$, with birational morphisms $\pi_1\colon X_1\to\PP^2$ and $\pi_2\colon X_2\to\PP^2$, if the birational map $(\pi_2)^{-1}\pi_1\colon X_1\dasharrow X_2$ is defined at $p_1$, sends $p_1$ to $p_2$, but is not a local isomorphism at $p_1$. In such a case we write that $p_1 \infnear p_2$.

One moreover says that $p_1$ is \emph{in the first neighbourhood} of $p_2$, or that $p_1$ is \emph{infinitely near $p_2$ of the first order}, if $(\pi_2)^{-1}\pi_1$ corresponds locally to the blow-up of $p_2$.
In such a case we write that $p_1 \infnear_1 p_2$.

If $p_1 \infnear p_2$ then one can define the \emph{infinitesimal order} of $p_1$ with respect to $p_2$ by induction, namely if $p_1 \infnear_1 p_3$ and $p_3 \infnear_k p_2$ for some $k$, then $p_1$ is \emph{infinitely near $p_2$ of order $k+1$}. 

If $p_1\infnear p_2$ and $p_1\in X_1$, then there is a unique irreducible curve $E_2\subset X_1$ which corresponds to the exceptional curve of the blowing-up of $p_2\in X_2$. One says that $p_1$ is \emph{proximate} to $p_2$ if $p_1\in E_2$. In such a case we write that $p_1\prox p_2$.

If $p_1\prox p_2$ and $p_1\infnear_k p_2$ with $k >1$, then we say that $p_1$ is \emph{satellite} to $p_2$ and we write $p_1 \satel p_2$. Otherwise, if $p_1$ is not satellite to $p_2$, then we write that $p_1 \notsatel p_2$.

One says that a point $p\in\PP^2\subset\BB(\PP^2)$ is a \emph{proper point of $\PP^2$}.

If $p_1\infnear p_2 \in \PP^2$ where $p_1 \in X_1$ and $\pi_1 : X_1 \rightarrow \PP^2$ is a birational morphism, we say that a plane curve $C$ passes through $p_1$ if $C$ passes through $p_2$ and the strict transform of $C$ on $X_1$ via $\pi_1$ passes through $p_1$.
\end{definition}

\begin{proposition}[Proximity inequality]\label{pro Proximity_inequalities} Let $\varphi : S \rightarrow \PP^2$ be a birational morphism, that is the composition of the blowing-ups $\pi_1, \ldots, \pi_r$ at single points.
Let $C$ be a plane curve and let $C_i$ be the strict transform of $C$ in $S_i$ for $i=1, \ldots,r$. Setting $C_0 =C$ and $m_i = \mult_{p_i}(C_{i-1})$ for $i=1, \ldots,r$, one has, for each $j = 1, \ldots, r$, 
$$m_j \geqslant \sum_{p_k\prox p_j} m_k.$$
\end{proposition}
\begin{proof}
See $\S2.2$ in \cite{MAC} or Theorem $3.5.3$, Corollary $3.5.4$ in \cite{ECA}.
\end{proof}

\section{Standard coordinates of infinitely near points}

In this section, we want to give a way to describe infinitely near points that we call \emph{standard coordinates}.

Let $p_1=[a:b:c] \in \PP^2$. Let us consider three cases:
\begin{itemize}
\item[$(i)$] if $c \neq 0$, then $p_1 =\bigg[\dfrac{a}{c}:\dfrac{b}{c}:1\bigg] = [\overline{a}:\overline{b}:1]$;
\item[$(ii)$] if $c = 0$ and $b \neq 0$, then   $p_1 =\bigg[\dfrac{a}{b}:1:0\bigg] = [\overline{a}:1:0]$;
\item[$(iii)$] if $c=b= 0$, then $p_1 = [1:0:0]$.
\end{itemize}

In case $(i)$, we work on the affine chart $U_2 \simeq \C^2_{\overline{x},\overline{y}}$, so that $p_1$ corresponds to the point $\overline{p}_1 = (\overline{a},\overline{b})$, and we define the isomorphism $\alpha_1 \colon \C^2_{\overline{x},\overline{y}} \rightarrow \C^2_{x_0,y_0}$ by
$$\alpha_1(\overline{x},\overline{y}) = (\overline{x}-\overline{a},\overline{y}-\overline{b}).$$

In case $(ii)$, we work on the affine chart $U_1 \simeq \C^2_{\overline{x},\overline{z}}$, so that $p_1$ corresponds to the point $\overline{p}_1 = (\overline{a},0)$, and we define the isomorphism $\alpha_1 \colon \C^2_{\overline{x},\overline{z}} \rightarrow \C^2_{x_0,y_0}$ by
$$\alpha_1(\overline{x},\overline{z}) = (\overline{x}-\overline{a},\overline{z}).$$

In case $(iii)$, we work on the affine chart $U_0 \simeq \C^2_{\overline{y},\overline{z}}$, so that $p_1$ corresponds to the point $\bar p_1=(0,0)$, and we define the isomorphism $\alpha_1 \colon \C^2_{\overline{y},\overline{z}} \rightarrow \C^2_{x_0,y_0}$ by
$$\alpha_1(\overline{y},\overline{z}) = (\overline{y},\overline{z}).$$

In all three cases, we defined $\alpha_1$ in such a way that $\alpha_1(\overline{p}_1) = (0,0)\in \C^2_{x_0,y_0}$.

We blow-up $\C^2_{x_0,y_0}$ at $(0,0)$ and we consider the first chart $\C^2_{x_1,y_1}$ where the blowing-up map is given in coordinates by $x_0=x_1, y_0=x_1y_1$.

In this chart, the exceptional curve $E_1$ has local equation $x_1=0$, hence a point $p_2 \infnear_1 p_1$ corresponds either to the point $(0,t_2) \in E_1$ with $t_2 \in \C$ or to the point which is the origin of the second chart.
In the former case, let us say that $p_2$ has standard coordinates $p_2=(p_1,t_2)$, while in the latter case let us say that $p_2$ has standard coordinates $p_2=(p_1,\infty)$.
Setting $\PP^1=\C\cup\{\infty\}$, let us denote the standard coordinates of $p_2$  by $p_2=(p_1,t_2)$ with $t_2\in\PP^1$.

\begin{remark}
Recall that a point $p_2\infnear_1 p_1$ corresponds to the direction of a line passing through $p_1$. More precisely, one can see that the point $p_2=(p_1,t_2)$, with $p_1=[a:b:c]$, corresponds to the line defined by the following equation
\[
\begin{cases}
cy-bz=t_2(cx-az) & \text{when } c\ne0 \text{ and } t_2\in\C,\\
cx-az=0 & \text{when } c\ne0 \text{ and } t_2=\infty,\\
bz=t_2(bx-ay) & \text{when } c=0, b \neq 0 \text{ and } t_2\in\C,\\
bx=ay & \text{when } c=0, b \neq 0 \text{ and } t_2=\infty,\\
z=t_2 y & \text{when } b=c=0 \text{ and } t_2\in\C,\\
y=0 & \text{when } b=c=0 \text{ and } t_2=\infty.
\end{cases}
\]
In other words, the above equations define the unique line passing through $p_1$ and $p_2$.
\end{remark}

We want to go on by blowing-up at $p_2=(p_1,t_2)$, with $t_2\in\PP^1=\C\cup\{\infty\}$.
Either $t_2\in\C$ or $t_2=\infty$. In the former case, with notation as above, let $\alpha_2 \colon \C^2_{x_1,y_1} \rightarrow \C^2_{\bar x_1,\bar y_1}$ be the isomorphism defined by
$$\alpha_2(x_1,y_1) = (x_1, y_1-t_2).$$
In the latter case, $p_2$ corresponds to the origin of the second chart of the blowing-up of $\C^2_{x_0,y_0}$ at $(0,0)$ that we write  $\C^2_{x'_1,y'_1}$, where the blowing-up map is given by $x_0=x'_1 y'_1, y_0=y'_1$. 
Let $\alpha_2 \colon \C^2_{x'_1,y'_1} \rightarrow \C^2_{\bar x_1,\bar y_1}$ be the isomorphism
$$\alpha_2(x'_1,y'_1) = (y'_1, x'_1).$$
In this way, in both cases, in $\C^2_{\bar x_1,\bar y_1}$ the exceptional curve $E_1$ has local equation $\bar x_1=0$ and the point $p_2$ corresponds to the origin $(0,0)$.

We blow-up $\C^2_{\bar x_1,\bar y_1}$ at $(0,0)$ and we consider the first chart $\C^2_{x_2,y_2}$ where the blowing-up map is given in coordinates by $\bar x_1=x_2, \bar y_1=x_2y_2$.
In this chart, the exceptional curve $E_2$ has local equation $x_2=0$, hence a point $p_3 \infnear_1 p_2$ corresponds either to the point $(0,t_3) \in E_2$ with $t_3 \in \C$ or to the point which is the origin of the second chart.

Let us say that $p_3$ has standard coordinates $p_3=(p_1,t_2,t_3)$, where either $t_3\in\C$ in the former case or $t_3=\infty$ in the latter case.

Note that the strict transform of $E_1$ can be seen only in the second chart and it meets $E_2$ at the origin of the second chart.
In other words, the point with standard coordinates $(p_1,t_2,\infty)$ is satellite to $p_1$.

More generally, let us proceed by induction of the infinitesimal order.
Suppose that we have blown-up the point $p_{r-1}$ with standard coordinates $p_{r-1}=(p_1,t_2,\ldots,t_{r-1})$, with $t_i\in\PP^1=\C\cup\{\infty\}$, $i=2,\ldots,r-1$.
Following the procedure described above, we may assume that $p_{r-1}$ is the origin of a chart $\C^2_{\bar x_{r-1},\bar y_{r-1}}$ in such a way that the exceptional curve $E_{r-1}$ has local equation $\bar x_{r-1}=0$.

In the first chart of the blowing up of $\C^2_{\bar x_{r-1},\bar y_{r-1}}$ at $(0,0)$, given in coordinates by $\bar x_{r-1}=x_r, \bar y_{r-1}=x_ry_r$, the exceptional curve $E_r$ has local equation $x_r=0$, hence a point $p_r \infnear_1 p_{r-1}$ corresponds either to the point $(0, t_r) \in E_r$ with $t_r\in\C$ or to the point which is the origin of the second chart, given in coordinates by $\bar x_{r-1}=x_ry_r, \bar y_{r-1}=y_r$.

Let us say that $p_r$ has standard coordinates $p_r=(p_1,t_2,\ldots,t_r)$, where $t_r\in\C$ in the former case and $t_r=\infty$ in the latter case.

The above discussion proves the following:

\begin{lemma}
Let $p_1 \in \PP^2$. Then, there is a one-to-one correspondence between points infinitely near $p_1$ of order $r$ and ${(\PP^1) }^r=\underbrace{\PP^1 \times \ldots \times \PP^1}_{r\text{-times}}$.
\end{lemma}

\begin{corollary}
There is a one-to-one correspondence between points infinitely near a proper point of order $r$ and $W = \PP^2 \times{(\PP^1) }^r$.
\end{corollary}

\begin{definition}
We call \emph{standard coordinates} of an infinitely near point the point of $W$ obtained with the above construction.
\end{definition}

\section{Conics and infinitely near points} \label{sec conics_infinitely_near_points}

\begin{remark}\label{rm no_conic1}\label{rek no_conic}
If $p_1 \in \mathbb{P}^2, p_3 \infnear_1 p_2 \infnear_1 p_1$ and $p_3 \odot p_1$, i.e. $p_3 \dashrightarrow p_1$, then there is no smooth plane curve passing through $p_1,p_2,p_3$, because of the proximity inequality at $p_1$.

For the same reason, if $p_1 \in \mathbb{P}^2$, $p_2 \infnear_1 p_1$, $p_3 \infnear_1 p_1$ and $p_2 \neq p_3$, then there is no smooth plane curve passing through $p_1,p_2,p_3$.
\end{remark}

\begin{lemma}\label{lm no_conic2}
If $p_1 \in \mathbb{P}^2, p_3 \infnear_1 p_2 \infnear_1 p_1$ and $p_1,p_2,p_3$ are collinear, namely $p_3$ lies on the strict transform of the line passing through $p_1$ and $p_2$, then there is no irreducible conic passing through $p_1,p_2,p_3$.
\end{lemma}
\begin{proof}
Up to automorphisms of $\PP^2$, we may assume that $p_1 = [1:0:0]$ and $p_2$ has standard coordinates $p_2= (p_1,0)$, so $p_3$ is uniquely determined by $p_1,p_2$, namely $p_3$ has standard coordinates $p_3=(p_1,0,0)$.

Suppose that $C$ is an irreducible conic passing through $p_1,p_2$. Then, $C$ has equation
\begin{equation*}
 a_2y^2+a_3xz+a_4yz+a_5z^2 = 0
\end{equation*}
where $a_2,a_3,a_4,a_5 \in \C$ and $a_2,a_3 \neq 0$ because $C$ is irreducible.

We work in the affine chart $U_0 \simeq \C^2_{\overline{y},\overline{z}}$
and we consider the isomorphism $\alpha_1\colon\C^2_{\bar y,\bar z}\to\C^2_{x_0,y_0}$ defined by $\alpha_1(\bar y,\bar z)=(\bar y, \bar z)$, where the conic $C$ has local equation
\begin{equation*}
a_2 x_0^2+a_3 y_0+a_4 x_0 y_0+a_5 y_0^2 = 0.
\end{equation*}
In the first chart of the blowing-up of $\C^2_{x_0,y_0}$ at the origin $(0,0)$, where $x_0=x_1,y_0=x_1y_1$,
the strict transform of $C$ has local equation
\begin{equation*}\label{eq2 no_conic1}
a_2x_1+a_3y_1+a_4x_1y_1+a_5x_1y_1^2 = 0.
\end{equation*}
Note that $p_2$ is just the origin of $\C^2_{x_1,y_1}$.\\
Then, the strict transform of $C$ via the blowing-up of $\C^2_{x_1,y_1}$ at the origin $(0,0)$ has local equation in the first chart, where $x_1=x_2,y_1=x_2y_2$,
$$a_2+a_3y_2+a_4x_2y_2+a_5x_2y_2^2 =0.$$
Note that $p_3$ is just the origin of $\C^2_{y_2,z_2}$ but the strict transform of $C$ does not pass through $(0,0)$ because $a_2 \neq 0$.
\end{proof}

\begin{remark}
It is easy to check that if $p_1 \in \mathbb{P}^2, p_3 \infnear_1 p_2 \infnear_1 p_1$, $p_3 \notsatel p_1$ and $p_1,p_2,p_3$ are not collinear, then there are irreducible conics passing through $p_1,p_2,p_3$.
\end{remark}

\begin{lemma}\label{lm conic1}
Let $p_1,p_2,p_3,p_4 \in \PP^2$ and $p_5 \infnear_1 p_1$ such that no three among $p_1,\ldots,p_5$ are collinear. Then, there exists a unique irreducible conic passing through $p_1,\ldots, p_5$.
\end{lemma}
\begin{proof} 
Up to automorphisms of $\PP^2$, we may assume that $p_1 = [1:0:0], p_2 = [0:1:0],p_3=[0:0:1],p_4=[1:1:1]$. Then, $p_5$ has standard coordinates $p_5 = (p_1,t_5)$,  namely $p_5$ is infinitely near $p_1$ of the first order in the direction of the line $z-t_5y=0$, where $t_5 \in \C \setminus \{ 0,1 \}$: indeed, if $t_5 =0$, then $p_5,p_2,p_1$ would be collinear; if $t_5 =1$, then $p_5,p_4,p_1$ would be collinear and finally, if $t_5 = \infty$, then $p_5,p_3,p_1$ would be collinear.
Then, one can check that the conic
$$xz-t_5xy+(t_5-1)yz=0$$
is the unique irreducible conic passing through $p_1,\ldots,p_5$.
\end{proof}

\begin{lemma}\label{lm conic2}
Let $p_1,p_2,p_3 \in \PP^2$ and $p_5 \infnear_1 p_4 \infnear_1 p_1$ such that $p_5 \notsatel p_1$ and no three among $p_1, \ldots, p_5$ are collinear. Then, there exists a unique irreducible conic passing through $p_1,\ldots, p_5$.
\end{lemma}
\begin{proof} 
Up to automorphisms of $\PP^2$, we may assume that $p_1 = [1:0:0], p_2 = [0:1:0],p_3=[0:0:1]$ and that $p_4$ has standard coordinates $p_4 = (p_1,1)$, namely $p_4$ is infinitely near $p_1$ of the first order in the direction of the line $y=z$.
Then, $p_5$ has standard coordinates $p_5 = (p_1,1,t_5)$, where $t_5 \in \C^*$: indeed, if $t_5=0$ then $p_5,p_4,p_1$ would be collinear and if $t_5=\infty$, then $p_5 \satel p_1$. Then, one can check that the conic
$$xz-xy-t_5yz =0$$
is the unique irreducible conic passing through $p_1,\ldots,p_5$.
\end{proof}

\begin{lemma}\label{lm conic3}
Let $p_1,p_2,p_3 \in \PP^2$ and $p_4 \infnear_1 p_1, p_5 \infnear_1 p_2$ such that no three among $p_1,\ldots,p_5$ are collinear. Then, there exists a unique irreducible conic passing through $p_1,\ldots, p_5$.
\end{lemma}
\begin{proof} 
Up to automorphisms of $\PP^2$, we may assume that $p_1 = [1:0:0], p_2 = [0:1:0],p_3=[0:0:1]$ and that the two lines, one through $p_1,p_4$ and the other one through $p_2,p_5$, meet at $[1:1:1]$, namely $p_4$ is infinitely near $p_1$ of the first order in the direction of the line $y=z$ and $p_5$ is infinitely near $p_2$ of the first order in the direction of the line $x=z$.
In other words, $p_4$ has standard coordinates $p_4 = (p_1,1)$ and $p_5$ has standard coordinates $p_5=(p_2,1)$.
Then, it is clear that the conic
$$xy-yz-xz=0$$
is the unique irreducible conic passing through $p_1,\ldots,p_5$.
\end{proof}

\begin{lemma}\label{lm conic3b}
Let $p_1,p_2 \in \PP^2$ and $p_5 \infnear_1 p_3 \infnear_1 p_1, p_4 \infnear_1 p_2$ such that $p_5 \notsatel p_1$ and no three among $p_1,\ldots,p_5$ are collinear. Then, there exists a unique irreducible conic passing through $p_1,\ldots, p_5$.
\end{lemma}
\begin{proof}
Up to automorphisms of $\PP^2$, we may assume that $p_1 = [1:0:0],p_2 = [0:1:0]$, and that the two lines, one through $p_1,p_3$ and the other one through $p_2,p_4$, meet at $[0:0:1]$, namely $p_3$ is infinitely near $p_1$ of the first order in the direction of the line $y=0$ and $p_4$ is infinitely near $p_2$ of the first order in the direction of the line $x=0$.
In other words, $p_3$ has standard coordinates $p_3=(p_1,\infty)$ and $p_4$ has standard coordinates $p_4=(p_2,\infty)$.
Then, $p_5$ has standard coordinates $p_5=(p_1,\infty,t_5)$ where $t_5 \in \C^*$: indeed, if $t_5=0$ then $p_5,p_3,p_1$ would be collinear and if $t_5=\infty$, then $p_5 \satel p_1$.
One can check that the conic
$$t_5xy - z^2 =0$$
is the unique irreducible conic passing through $p_1,\ldots,p_5$.
\end{proof}

\begin{remark}
The previous lemmas better explain Remark V.4.2.1 in \cite{RH}. 
\end{remark}

\begin{lemma}\label{lm conic4}
Let $p_1,p_2 \in \PP^2$ and $p_5 \infnear_1 p_4 \infnear_1 p_3 \infnear_1 p_1$ such that $p_4 \notsatel p_1, p_5 \notsatel p_3$ and no three among $p_1,\ldots,p_4$ are collinear. Then, there exists a unique irreducible conic passing through $p_1,\ldots,p_5$.
\end{lemma}
\begin{proof}
Up to automorphisms of $\PP^2$, we may assume that $p_1 = [1:0:0]$, $p_2 = [0:1:0]$ and $p_3, p_4$ have standard coordinates respectively $p_3 =(p_1,\infty)$ and $p_4=(p_1,\infty,1)$, according to the proof of the previous lemma.
Then, $p_5$ has standard coordinates $p_5= (p_1,\infty,1,t_5)$ where $t_5 \in \C$: indeed, if $p_5 = \infty$, then we would have $p_5 \satel p_3$, contradicting the hypothesis.
One can check that the conic
$$xy +t_5 yz - z^2 =0$$
is the unique irreducible conic passing through $p_1,\ldots,p_5$.
\end{proof}

\begin{lemma}\label{lm conic5}
Let $p_5 \infnear_1 p_4 \infnear_1 p_3 \infnear_1 p_2 \infnear_1 p_1 \in \PP^2$  such that $p_3 \notsatel p_1, p_4 \notsatel p_2, p_5 \notsatel p_3$ and $p_1,p_2,p_3$ are not collinear. Then, there exists a unique irreducible conic passing through $p_1,\ldots,p_5$.
\end{lemma}
\begin{proof}
Up to automorphisms of $\PP^2$, we may assume that $p_1 = [1:0:0]$ and $p_2, p_3, p_4$ have standard coordinates respectively $p_2 =(p_1,\infty)$, $p_3=(p_1,\infty,1)$, $p_4=(p_1,\infty,1,0)$, according to the proof of the previous lemma. Then, $p_5$ has standard coordinates $p_5 = (p_1,\infty,1,0,t_5)$ where $t_5 \in \C$: indeed, if $t_5 = \infty$, then we would have $p_5 \satel p_3$, contradicting the hypothesis.
One can check that the conic
$$xy-z^2+t_5 y^2 =0$$
is the unique irreducible conic passing through $p_1,\ldots,p_5$.
\end{proof}

\section{Lengths of plane Cremona maps}  \label{sec plane_Cremona_maps}

A plane Cremona map $\varphi\colon \PP^2 \dasharrow \PP^2$ can be written as
\begin{equation*}
\varphi( [x : y : z] ) = [f_0(x, y, z) : f_1(x, y, z) : f_2(x, y, z)]
\end{equation*}
where $f_i\in\C[x,y,z]$, $i=0,1,2$, are homogeneous polynomials of the same degree, say $d$, that is called the \emph{degree} of $\varphi$ if $f_0, f_1, f_2$ have no common factor.

Plane Cremona maps of degree 1 are automorphisms of $\PP^2$, i.e.\ elements of $\Aut(\PP^2)=\PGL_3$.
Plane Cremona maps of degree 2 (3, resp.) are called \emph{quadratic} (\emph{cubic}, resp.).

\begin{definition}
Let us say that  two plane Cremona maps $\varphi,\varphi'\colon\PP^2\dasharrow\PP^2$  are \emph{equivalent} if there exist two automorphisms $\alpha,\alpha'\in\Aut(\PP^2)$ such that
\[
\varphi'=\alpha'\circ\varphi\circ\alpha.
\]
\end{definition}

\begin{remark}
It is very well-known that quadratic plane Cremona maps are equivalent to one and only one of the following three maps:
the elementary quadratic transformation $\sigma$;
\begin{equation}\label{rhotau}
\rho([x:y:z])=[xy : z^2 : yz];
\qquad
\tau([x:y:z])=[x^2 : xy : y^2-xz].
\end{equation}
Recall that a quadratic plane Cremona map is called \emph{ordinary} (or \emph{of the first type})  if it has three proper base points, i.e., if it is equivalent to $\sigma$. Furthermore, a quadratic plane Cremona map is called \emph{of the second type} if it has two proper base points, i.e., if it is equivalent to $\rho$, and, finally, a quadratic plane Cremona map is called \emph{of the third type} if it has only one proper base point, i.e., if it is equivalent to $\tau$.
\end{remark}

Recall that the linear system defining a plane Cremona map $\varphi$ of degree $d$ has base points $p_1,\ldots,p_r\in\BB(\PP^2)$ of respective multiplicity $m_1,\ldots,m_r$
that satisfy the following equations:
\begin{equation}\label{eq:Cremona}
d^2-1=\sum_{i=1}^r m_i^2,
\qquad\qquad
3(d-1)=\sum_{i=1}^r m_i.
\end{equation}

\begin{definition}
A plane Cremona map is called \emph{de Jonqui\`eres} if it has degree $d$ and a base point of multiplicity $d-1$.
\end{definition}

Equations \eqref{eq:Cremona} imply that plane Cremona maps of degree $2$ and $3$ are de Jonqui\`eres.

According to the Noether-Castelnuovo Theorem, any plane Cremona map $\varphi\colon\PP^2\dasharrow\PP^2$ can be written
\begin{equation}\label{eq:decomp}
\varphi = \alpha_n \circ \sigma \circ \alpha_{n-1} \circ \sigma \circ \cdots \circ \alpha_1 \circ \sigma \circ \alpha_0
\end{equation}
where $\alpha_i\in\Aut(\PP^2)$ for any $i=0,\ldots,n$, for some integer $n$.

\begin{definition}
Let us call \eqref{eq:decomp} a \emph{decomposition} of $\varphi$.
Let us say that a decomposition \eqref{eq:decomp} is \emph{minimal} if $n$ is minimal among all decompositions of $\varphi$.
Let us call such $n$ the \emph{ordinary quadratic length} of $\varphi$ and let us denote it by $\oq(\varphi)$.
\end{definition}

Therefore, the ordinary quadratic length of a plane Cremona map $\varphi$ of degree $\geq2$ is the minimun $n$ such that there exist ordinary quadratic maps $\psi_1,\psi_2,\ldots,\psi_n$ with
\begin{equation}\label{eq:decomp2}
\varphi=\psi_n\circ\psi_{n-1}\circ \cdots \circ \psi_2\circ \psi_1.
\end{equation}

\begin{definition}
Let us call the \emph{quadratic length} of plane Cremona map $\varphi$ the minimum $n$ such that there exists a decomposition \eqref{eq:decomp2} where $\psi_i$ is a (not necessarily ordinary) quadratic map, for each $i=1,\ldots,n$, and denote it by $\q(\varphi)$.
\end{definition}

Recall that Blanc and Furter in \cite{B-F} defined the \emph{length} of a plane Cremona map $\varphi$ as the minimum $n$ such that there exists a decomposition \eqref{eq:decomp2} where $\psi_i$ is a \emph{de Jonqui\`eres} map, for each $i=1,\ldots,n$, and denoted it by $\lgth(\varphi)$.
Clearly, one has that
\begin{equation*}
\lgth(\varphi) \leq \q(\varphi) \leq \oq(\varphi).
\end{equation*}

\begin{definition}
A plane Cremona map $\varphi$ is called \emph{involutory}, or \emph{an involution}, if $\varphi^{-1}=\varphi$.
\end{definition}

\begin{remark}
In order to compute the ordinary quadratic length of plane Cremona maps,
it suffices to work with \emph{involutory} ordinary quadratic maps.
Indeed, any decomposition \eqref{eq:decomp} can be written as the composition of an automorphism and involutory quadratic maps:
\[
\varphi=\alpha'_n \circ \cdots \circ ((\alpha_1\circ\alpha_0)^{-1}\circ\sigma\circ(\alpha_1\circ\alpha_0)) \circ (\alpha_0^{-1}\circ\sigma\circ\alpha_0)
\]
where $\alpha'_n=\alpha_n\circ\alpha_{n-1}\circ\cdots\circ\alpha_1\circ\alpha_0$.
\end{remark}

\begin{remark}
Two equivalent plane Cremona maps clearly have the same length, quadratic length and ordinary quadratic length.
\end{remark}

The following lemma is a straightforward application of the definitions.

\begin{lemma}\label{lm q_0_1}\label{lm oq_0_1}
Let $\varphi: \PP^2 \dashrightarrow \PP^2$ be a plane Cremona map.
Then,
\[
\oq(\varphi)=0 \text{ if and only if } \varphi\in\Aut(\PP^2).
\]
Moreover, one has
\begin{itemize}
\item $\oq(\varphi)=1$ if and only if $\varphi$ is an ordinary quadratic map;
\item $\q(\varphi)=1$ if and only if $\varphi$ is a quadratic map;
\item $\lgth(\varphi)=1$ if and only if $\varphi$ is a de Jonqui\`eres map.
\end{itemize}
\end{lemma}

\begin{corollary}\label{cor>1}
Let $\varphi\colon \PP^2 \dashrightarrow \PP^2$ be a plane Cremona map of degree $d\geq3$.
Then,
\[
\oq(\varphi)\geq \q(\varphi) \geq2.
\]
\end{corollary}

\begin{example}
Let $\rho$ be the quadratic map defined in \eqref{rhotau}.
It is classically very well-known that $\oq(\rho)=2$.
A minimal decomposition of $\rho$ is:
\begin{equation*}\label{rhodecomp}
\rho=[x: z-y: z] \circ\sigma\circ [x : y+z : z] \circ \sigma \circ [x: y-z: z].
\end{equation*}
\end{example}

\begin{example}
Let $\tau$ be the quadratic map defined in \eqref{rhotau}.
It is classically well-known that $\tau$ is the composition of two quadratic maps of the second type and therefore the composition of four ordinary quadratic maps.
A decomposition of $\tau$, given in \cite{C-D}, is:
\begin{equation}\label{taudecomp}
\begin{aligned}
\tau &= [y-x : 2y-x : x-y+z] \circ \sigma \circ [x+z : x : y] \circ \sigma \circ [-y : x-3y+z : x] \, \circ \\
&\phantom{=} \, \circ \sigma \circ [x+z : x : y] \circ \sigma \circ [y-x : -2x+z : 2x-y].
\end{aligned}
\end{equation}
However, we found no reference with a proof that $\oq(\tau)=4$, hence that the above decomposition is minimal, even if we believe that it was classically known.
This fact can be shown as a consequence of our Theorem \ref{thm main2}, see Corollary \ref{oq(tau)} later.
\end{example}

\begin{corollary}\label{cor>2}
Let $\varphi\colon \PP^2 \dashrightarrow \PP^2$ be a plane Cremona map of degree $d\geq5$.
Then,
\[
\oq(\varphi)\geq \q(\varphi) \geq3.
\]
\end{corollary}

\begin{proof}
We claim that, if $\q(\varphi)\leq2$, then $\deg(\varphi)\leq 4$.
This is trivial if $\q(\varphi)\leq1$. Suppose that $\q(\varphi)=2$, namely $\varphi=\rho_2\circ\rho_1$, where $\rho_1,\rho_2$ are quadratic maps.
Let $p_1,p_2,p_3$ be the base points of $\rho_2$. If $m_1, m_2, m_3$ are the multiplicities of $\rho_1^{-1}$ at $p_1,p_2,p_3$, respectively, then
\[
\deg(\varphi) = \deg(\rho_2\circ\rho_1) = 4 - m_1-m_2-m_3\leq 4,
\]
that is our claim.
\end{proof}

\begin{lemma}\label{cordeJonq>d-1}
Let $\varphi\colon \PP^2 \dashrightarrow \PP^2$ be a plane de Jonqui\`eres map of degree $d\leq5$.
Then,
\[
\oq(\varphi)\geq \q(\varphi) \geq d-1.
\]
\end{lemma}

\begin{proof}
It is trivial if $d\leq 3$.
Let us first consider the case $d=4$.

By contradiction, suppose that $\q(\varphi)\leq2$. Clearly, $\q(\varphi)$ cannot be less than 2, so we can write $\varphi = \varrho_2 \circ \varrho_1$,
where $\varrho_1,\varrho_2$ are two quadratic plane Cremona maps. In other words, one has that $\varphi \circ \varrho_1^{-1}$ is the quadratic map $\varrho_2$. We claim that the composition $\varphi \circ \varrho_1^{-1}$ has always degree $\geq 3$, that is a contradiction.

We now prove our claim. Suppose that $p_0, p_1, \ldots , p_6$ are the base points of $\varphi$, where $p_0$ is the triple base point and $p_1, \ldots, p_6$ are simple base points.

We distinguish four possibilities:
\begin{itemize}
\item if $\varrho_1$ has base points $p_0,p_i,p_j$ with $0 < i < j \leq 6$, then $\varphi \circ \varrho_1^{-1}$ has degree $3$;
\item if $\varrho_1$ has base points $p_0,p_i$ with $0 < i \leq 6$ and $p_j$ is not a base point of $\varrho_1$ for any $j$ such that $0 \leq j \leq 6$ and $j \neq 0,i$, then $\varphi \circ \varrho_1^{-1}$ has degree $4$;
\item if $\varrho_1$ has base point $p_0$ and $p_1, \ldots,p_6$ are not base points of $\varrho_1$, then $\varphi \circ \varrho_1^{-1}$ has degree $5$;
\item if $p_0$ is not a base point of $\varrho_1$, then $5 \leq \deg(\varphi \circ \varrho_1^{-1}) \leq 8$.
\end{itemize}
Our claim is proved.

We are left with the case $d=5$.

By contradiction, suppose that $\q(\varphi)\leq3$, hence, $\q(\varphi)=3$ by Corollary \ref{cor>2} and we can write $\varphi = \varrho_3 \circ \varrho_2 \circ \varrho_1$, where $\varrho_1,\varrho_2,\varrho_3$ are quadratic plane Cremona maps. In other words, one has that $\varphi \circ \varrho_1^{-1}=\varrho_3 \circ \varrho_2$ has quadratic length 2.

Let $p_0$ be the base point of multiplicity 4 of $\varphi$.
There are two cases: either $p_0$ is a base point of $\varrho_1$ or $p_0$ is not a base point of $\varrho_1$.

In the former case, the map $\varphi \circ \varrho_1^{-1}=\varrho_3 \circ \varrho_2$ is a de Jonqui\`eres map of degree $d'$ with $4 \leq d' \leq 6$.
If $d'=5,6$, then Corollary \ref{cor>2} gives a contradiction. Otherwise $d'=4$, that is another contradiction with the first part of this proof.

In the latter case, the map $\varphi \circ \varrho_1^{-1}=\varrho_3 \circ \varrho_2$ has degree $d''$ with $7\leq d''\leq 10$ and we get again a contradiction with Corollary \ref{cor>2}.
\end{proof}

\section{Enriched weighted proximity graphs} \label{weighted}

\begin{definition}\label{def weighted_proximity_graph}
Let $\varphi : \PP^2 \dashrightarrow \PP^2$ be a plane Cremona map. Let us associate to $\varphi$ a \emph{weighted digraph} $G_\varphi$, called the \emph{weighted proximity graph of (the base points of) $\varphi$}, defined as follows:
\begin{itemize}
\item the vertices of $G_\varphi$ are the base points $p_1,\ldots,p_r \in \BB(\PP^2)$ of $\varphi$;
\item there is an arrow $p_i \to p_j$ if and only if $p_i$ is proximate to $p_j$;
\item each vertex $p_i$ is weighted with the multiplicity $\mult_{p_i}(\varphi)$ of $\varphi$ at $p_i$.
\end{itemize}
\end{definition}

It is easy to check that (weighted) proximity graphs have the so-called property of being \emph{admissible}, according to the following definition.

\begin{definition}
Let us say that a digraph is admissible if it is acyclic and satisfies the following three properties:
\begin{enumerate}[(i)]
\item each vertex has the external degree at most two;
\item if a vertex $u$ has outdegree 2, say $u \to v$ and $u \to w$, then either $v \to w$ or $w \to v$;
\item fixing two vertices $v$ and $w$, then there exists at most one vertex $u$ such that $u \to v$  and $u \to w$.
\end{enumerate}
\end{definition}

\begin{example}
A de Jonqui\`eres map of degree $d$ has weighted proximity graph with $2d-1$ vertices, one with weight $d-1$ and the other $2d-2$ vertices with weight 1.
\end{example}

\begin{remark} 
Note that the number of connected components of $G_\varphi$ equals the number of proper base points in $\PP^2$ among the base points $p_1,\ldots,p_r \in \BB(\PP^2)$ of $\varphi$.
\end{remark}

\begin{remark}
Clearly, two equivalent plane Cremona maps have the same weighted proximity graph. The converse is true for quadratic maps but it is not true in general. 
\end{remark}

\begin{example}
We will see later that the two cubic plane Cremona maps
$$\varphi_{10} = [x^3: y^2z:xyz], \qquad \textrm{ and } \quad \varphi_{11} =[x(y^2+xz):y(y^2+xz):xyz]$$
have the same weighted proximity graph
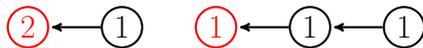
\begin{center}
\scalebox{0.5}{%
\begin{tikzpicture}[->,>=stealth',shorten >=2pt,auto,node distance=2.5cm,ultra thick,baseline=-1.25ex]
  \tikzstyle{every state}=[fill=white,draw=black,text=black]

  \node[state,draw=red,text=red] (A)                    {\huge 2};
  \node[state] (B) [right of=A] {\huge 1};
  \node[state,draw=red,text=red] (C) [right of=B] {\huge 1};
  \node[state] (D) [right of=C] {\huge 1};
  \node[state] (E) [right of=D] {\huge 1};

\path (E) edge (D);
\path (D) edge (C);
\path (B) edge (A);
\end{tikzpicture}%
}
\end{center}
but they are not equivalent.
\end{example}

\begin{notation}
When we draw the weighted proximity graph of a plane Cremona map $\varphi$, for readers' convenience we write \emph{proper base points} in \textcolor{red}{red} and \emph{infinitely near points} in \textcolor{black}{black}.
\end{notation}

\begin{example}
Let $\sigma$, $\rho$ and $\tau$ be the quadratic maps defined in \eqref{sigma_eq} and \eqref{rhotau}. Their respective proximity graphs $G_\sigma$, $G_\rho$ and $G_\tau$ are:
\begin{center}
$G_\sigma$ =
\scalebox{0.5}{%
\begin{tikzpicture}[->,>=stealth',shorten >=2pt,auto,node distance=2.5cm,ultra thick,baseline=-1.25ex]
  \tikzstyle{every state}=[fill=white,draw=black,text=black]

  \node[state,draw=red,text=red] (A)                    {\huge 1};
  \node[state,draw=red,text=red] (B) [right of=A] {\huge 1};
  \node[state,draw=red,text=red] (C) [right of=B] {\huge 1};
  \end{tikzpicture}%
} \qquad
$G_\rho$ =
\scalebox{0.5}{%
\begin{tikzpicture}[->,>=stealth',shorten >=2pt,auto,node distance=2.5cm,ultra thick,baseline=-1.25ex]
  \tikzstyle{every state}=[fill=white,draw=black,text=black]

  \node[state,draw=red,text=red] (A)                    {\huge 1};
  \node[state] (B) [right of=A] {\huge 1};
  \node[state,draw=red,text=red] (C) [right of=B] {\huge 1};
  \path (B) edge (A);
\end{tikzpicture}%
} \qquad
$G_\tau$ =
\scalebox{0.5}{%
\begin{tikzpicture}[->,>=stealth',shorten >=2pt,auto,node distance=2.5cm,ultra thick,baseline=-1.25ex]
  \tikzstyle{every state}=[fill=white,draw=black,text=black]

  \node[state,draw=red,text=red] (A)                    {\huge 1};
  \node[state] (B) [right of=A] {\huge 1};
  \node[state] (C) [right of=B] {\huge 1};

\path (C) edge (B);
\path (B) edge (A);
\end{tikzpicture}%
}
\end{center}
\end{example}

\begin{remark}\label{lm no_pro_graph}
Let 
\begin{center}
$G$ =
\scalebox{0.5}{%
\begin{tikzpicture}[->,>=stealth',shorten >=2pt,auto,node distance=2.5cm,ultra thick,baseline=-1.25ex]
  \tikzstyle{every state}=[fill=white,draw=black,text=black]

  \node[state,draw=red,text=red] (A)                    {\huge 1};
  \node[state] (B) [right of=A] {\huge 1};
  \node[state] (C) [right of=B] {\huge 1};
  \path (B) edge (A);
  \path (C) edge[bend right] (A);
\end{tikzpicture}%
} \qquad\qquad
\textrm{ and } \qquad\qquad $G' $ =
\scalebox{0.5}{%
\begin{tikzpicture}[->,>=stealth',shorten >=2pt,auto,node distance=2.5cm,ultra thick,baseline=-1.25ex]
  \tikzstyle{every state}=[fill=white,draw=black,text=black]

  \node[state,draw=red,text=red] (A)                    {\huge 1};
  \node[state] (B) [right of=A] {\huge 1};
  \node[state] (C) [right of=B] {\huge 1};
\path (C) edge[bend right] (A);
\path (C) edge (B);
\path (B) edge (A);
\end{tikzpicture}%
}.
\end{center}
Then, there is no plane Cremona map with $G$ or $G'$ as weighted proximity graph, cf. the proximity inequality \ref{pro Proximity_inequalities} and Remark \ref{rm no_conic1}. 
\end{remark}

We now classify weighted proximity graphs of cubic plane Cremona maps.

\begin{theorem}\label{thm:graphs}
There are exactly 21 weighted proximity graphs of cubic plane Cremona maps, up to isomorphism, that are listed in Table \ref{table0} at page \pageref{table0}.
\end{theorem}

\begin{table}[!htbp]
\centering
\begin{tabular}{cc}
\begin{tabular}{|c|c|}
\hline

\rule{0ex}{2.5ex}n$^o$ & Weighted proximity graph\\  \hline

1  &
\scalebox{0.5}{%
\begin{tikzpicture}[->,>=stealth',shorten >=2pt,auto,node distance=2.5cm,ultra thick,baseline=-1.25ex]
  \tikzstyle{every state}=[fill=white,draw=black,text=black]

  \node[state,draw=red,text=red] (A)                    {\huge 2};
  \node[state] (B) [right of=A] {\huge 1};
  \node[state] (C) [right of=B] {\huge 1};
  \node[state] (D) [right of=C] {\huge 1};
  \node[state] (E) [right of=D] {\huge 1};

\path (E) edge (D);
\path (D) edge (C);
\path (C) edge (B);
\path (B) edge (A);
\path (C) edge [bend right] (A);
\end{tikzpicture}%
}
\\[0.5ex]
\hline
2  &
\scalebox{0.5}{%
\begin{tikzpicture}[->,>=stealth',shorten >=2pt,auto,node distance=2.5cm,ultra thick,baseline=-1.25ex]
  \tikzstyle{every state}=[fill=white,draw=black,text=black]

  \node[state,draw=red,text=red] (A)                    {\huge 2};
  \node[state] (B) [right of=A] {\huge 1};
  \node[state] (C) [right of=B] {\huge 1};
  \node[state,draw=red,text=red] (D) [right of=C] {\huge 1};
  \node[state] (E) [right of=D] {\huge 1};

\path (E) edge (D);
\path (C) edge (B);
\path (B) edge (A);
\path (C) edge [bend right] (A);
\end{tikzpicture}%
}
\\[0.5ex]
\hline
3  &
\scalebox{0.5}{%
\begin{tikzpicture}[->,>=stealth',shorten >=2pt,auto,node distance=2.5cm,ultra thick,baseline=-1.25ex]
  \tikzstyle{every state}=[fill=white,draw=black,text=black]

  \node[state,draw=red,text=red] (A)                    {\huge 2};
  \node[state] (B) [right of=A] {\huge 1};
  \node[state] (C) [right of=B] {\huge 1};
  \node[state] (D) [right of=C] {\huge 1};
  \node[state,draw=red,text=red] (E) [right of=D] {\huge 1};

\path (D) edge (C);
\path (C) edge (B);
\path (B) edge (A);
\path (C) edge [bend right] (A);
\end{tikzpicture}%
}
\\[0.5ex]
\hline
4   &
\scalebox{0.5}{%
\begin{tikzpicture}[->,>=stealth',shorten >=2pt,auto,node distance=2.5cm,ultra thick,baseline=-1.25ex]
  \tikzstyle{every state}=[fill=white,draw=black,text=black]

  \node[state,draw=red,text=red] (A)                    {\huge 2};
  \node[state] (B) [right of=A] {\huge 1};
  \node[state] (C) [right of=B] {\huge 1};
  \node[state] (D) [right of=C] {\huge 1};
  \node[state] (E) [right of=D] {\huge 1};

\path (E) edge (D);
\path (D) edge [bend right] (A);
\path (C) edge (B);
\path (B) edge (A);
\end{tikzpicture}%
}
\\[0.5ex]
\hline
5   &
\scalebox{0.5}{%
\begin{tikzpicture}[->,>=stealth',shorten >=2pt,auto,node distance=2.5cm,ultra thick,baseline=-1.25ex]
  \tikzstyle{every state}=[fill=white,draw=black,text=black]

  \node[state,draw=red,text=red] (A)                    {\huge 2};
  \node[state] (B) [right of=A] {\huge 1};
  \node[state] (C) [right of=B] {\huge 1};
  \node[state] (D) [right of=C] {\huge 1};
  \node[state] (E) [right of=D] {\huge 1};

\path (E) edge (D);
\path (D) edge (C);
\path (B) edge (A);
\path (C) edge [bend right] (A);
\end{tikzpicture}%
}
\\[0.5ex]
\hline
\rule{0ex}{3.0ex}6  &
\scalebox{0.5}{%
\begin{tikzpicture}[->,>=stealth',shorten >=2pt,auto,node distance=2.5cm,ultra thick,baseline=-1.25ex]
  \tikzstyle{every state}=[fill=white,draw=black,text=black]

  \node[state,draw=red,text=red] (A)                    {\huge 2};
  \node[state] (B) [right of=A] {\huge 1};
  \node[state] (C) [right of=B] {\huge 1};
  \node[state] (D) [right of=C] {\huge 1};
  \node[state] (E) [right of=D] {\huge 1};

\path (E) edge (D);
\path (D) edge (C);
\path (C) edge (B);
\path (B) edge (A);
\end{tikzpicture}%
}
\\[0.5ex]
\hline
7  &
\scalebox{0.5}{%
\begin{tikzpicture}[->,>=stealth',shorten >=2pt,auto,node distance=2.5cm,ultra thick,baseline=-1.25ex]
  \tikzstyle{every state}=[fill=white,draw=black,text=black]

  \node[state,draw=red,text=red] (A)                    {\huge 2};
  \node[state] (B) [right of=A] {\huge 1};
  \node[state] (C) [right of=B] {\huge 1};
  \node[state,draw=red,text=red] (D) [right of=C] {\huge 1};
  \node[state,draw=red,text=red] (E) [right of=D] {\huge 1};

\path (C) edge (B);
\path (B) edge (A);
\path (C) edge [bend right] (A);
\end{tikzpicture}%
}
\\[0.5ex]
\hline
8   &
\scalebox{0.5}{%
\begin{tikzpicture}[->,>=stealth',shorten >=2pt,auto,node distance=2.5cm,ultra thick,baseline=-1.25ex]
  \tikzstyle{every state}=[fill=white,draw=black,text=black]

  \node[state,draw=red,text=red] (A)                    {\huge 2};
  \node[state] (B) [right of=A] {\huge 1};
  \node[state] (C) [right of=B] {\huge 1};
  \node[state] (D) [right of=C] {\huge 1};
  \node[state,draw=red,text=red] (E) [right of=D] {\huge 1};

\path (D) edge (C);
\path (B) edge (A);
\path (C) edge [bend right] (A);
\end{tikzpicture}%
}
\\[0.5ex]
\hline
9  &
\scalebox{0.5}{%
\begin{tikzpicture}[->,>=stealth',shorten >=2pt,auto,node distance=2.5cm,ultra thick,baseline=-1.25ex]
  \tikzstyle{every state}=[fill=white,draw=black,text=black]

  \node[state,draw=red,text=red] (A)                    {\huge 2};
  \node[state] (B) [right of=A] {\huge 1};
  \node[state] (C) [right of=B] {\huge 1};
  \node[state,draw=red,text=red] (D) [right of=C] {\huge 1};
  \node[state] (E) [right of=D] {\huge 1};

\path (E) edge (D);
\path (B) edge (A);
\path (C) edge [bend right] (A);
\end{tikzpicture}%
}
\\[0.5ex]
\hline
\rule{0ex}{3.0ex}10  &
\scalebox{0.5}{%
\begin{tikzpicture}[->,>=stealth',shorten >=2pt,auto,node distance=2.5cm,ultra thick,baseline=-1.25ex]
  \tikzstyle{every state}=[fill=white,draw=black,text=black]

  \node[state,draw=red,text=red] (A)                    {\huge 2};
  \node[state] (B) [right of=A] {\huge 1};
  \node[state] (C) [right of=B] {\huge 1};
  \node[state] (D) [right of=C] {\huge 1};
  \node[state,draw=red,text=red] (E) [right of=D] {\huge 1};

\path (D) edge (C);
\path (C) edge (B);
\path (B) edge (A);
\end{tikzpicture}%
}
\\[0.5ex]
\hline
\end{tabular}
&
\raisebox{0.0ex}{\begin{tabular}{|c|c|}
\hline

\rule{0ex}{2.5ex}n$^o$ & Weighted proximity graph\\  \hline

\rule{0ex}{3.0ex} 11  &
\scalebox{0.5}{%
\begin{tikzpicture}[->,>=stealth',shorten >=2pt,auto,node distance=2.5cm,ultra thick,baseline=-1.25ex]
  \tikzstyle{every state}=[fill=white,draw=black,text=black]

  \node[state,draw=red,text=red] (A)                    {\huge 2};
  \node[state] (B) [right of=A] {\huge 1};
  \node[state] (C) [right of=B] {\huge 1};
  \node[state,draw=red,text=red] (D) [right of=C] {\huge 1};
  \node[state] (E) [right of=D] {\huge 1};

\path (E) edge (D);
\path (C) edge (B);
\path (B) edge (A);
\end{tikzpicture}%
}
\\[0.5ex]
\hline
\rule{0ex}{3.0ex} 12  &
\scalebox{0.5}{%
\begin{tikzpicture}[->,>=stealth',shorten >=2pt,auto,node distance=2.5cm,ultra thick,baseline=-1.25ex]
  \tikzstyle{every state}=[fill=white,draw=black,text=black]

  \node[state,draw=red,text=red] (A)                    {\huge 2};
  \node[state] (B) [right of=A] {\huge 1};
  \node[state,draw=red,text=red] (C) [right of=B] {\huge 1};
  \node[state] (D) [right of=C] {\huge 1};
  \node[state] (E) [right of=D] {\huge 1};

\path (E) edge (D);
\path (D) edge (C);
\path (B) edge (A);
\end{tikzpicture}%
}
\\[0.5ex]
\hline
\rule{0ex}{3.0ex} 13  &
\scalebox{0.5}{%
\begin{tikzpicture}[->,>=stealth',shorten >=2pt,auto,node distance=2.5cm,ultra thick,baseline=-1.25ex]
  \tikzstyle{every state}=[fill=white,draw=black,text=black]

  \node[state,draw=red,text=red] (A)                    {\huge 2};
  \node[state,draw=red,text=red] (B) [right of=A] {\huge 1};
  \node[state] (C) [right of=B] {\huge 1};
  \node[state] (D) [right of=C] {\huge 1};
  \node[state] (E) [right of=D] {\huge 1};

\path (E) edge (D);
\path (D) edge (C);
\path (C) edge (B);
\end{tikzpicture}%
}
\\[0.5ex]
\hline
\rule{0ex}{3.0ex} 14  &
\scalebox{0.5}{%
\begin{tikzpicture}[->,>=stealth',shorten >=2pt,auto,node distance=2.5cm,ultra thick,baseline=-1.25ex]
  \tikzstyle{every state}=[fill=white,draw=black,text=black]

  \node[state,draw=red,text=red] (A)                    {\huge 2};
  \node[state] (B) [right of=A] {\huge 1};
  \node[state] (C) [right of=B] {\huge 1};
  \node[state,draw=red,text=red] (D) [right of=C] {\huge 1};
  \node[state,draw=red,text=red] (E) [right of=D] {\huge 1};

\path (B) edge (A);
\path (C) edge [bend right] (A);
\end{tikzpicture}%
}
\\[0.5ex]
\hline
\rule{0ex}{3.0ex} 15  &
\scalebox{0.5}{%
\begin{tikzpicture}[->,>=stealth',shorten >=2pt,auto,node distance=2.5cm,ultra thick,baseline=-1.25ex]
  \tikzstyle{every state}=[fill=white,draw=black,text=black]

  \node[state,draw=red,text=red] (A)                    {\huge 2};
  \node[state] (B) [right of=A] {\huge 1};
  \node[state] (C) [right of=B] {\huge 1};
  \node[state,draw=red,text=red] (D) [right of=C] {\huge 1};
  \node[state,draw=red,text=red] (E) [right of=D] {\huge 1};

\path (C) edge (B);
\path (B) edge (A);
\end{tikzpicture}%
}
\\[0.5ex]
\hline
\rule{0ex}{3.0ex} 16  &
\scalebox{0.5}{%
\begin{tikzpicture}[->,>=stealth',shorten >=2pt,auto,node distance=2.5cm,ultra thick,baseline=-1.25ex]
  \tikzstyle{every state}=[fill=white,draw=black,text=black]

  \node[state,draw=red,text=red] (A)                    {\huge 2};
  \node[state] (B) [right of=A] {\huge 1};
  \node[state,draw=red,text=red] (C) [right of=B] {\huge 1};
  \node[state,draw=red,text=red] (D) [right of=C] {\huge 1};
  \node[state] (E) [right of=D] {\huge 1};

\path (E) edge (D);
\path (B) edge (A);
\end{tikzpicture}%
}
\\[0.5ex]
\hline
\rule{0ex}{3.0ex} 17  &
\scalebox{0.5}{%
\begin{tikzpicture}[->,>=stealth',shorten >=2pt,auto,node distance=2.5cm,ultra thick,baseline=-1.25ex]
  \tikzstyle{every state}=[fill=white,draw=black,text=black]

  \node[state,draw=red,text=red] (A)                    {\huge 2};
  \node[state,draw=red,text=red] (B) [right of=A] {\huge 1};
  \node[state] (C) [right of=B] {\huge 1};
  \node[state,draw=red,text=red] (D) [right of=C] {\huge 1};
  \node[state] (E) [right of=D] {\huge 1};

\path (E) edge (D);
\path (C) edge (B);
\end{tikzpicture}%
}
\\[0.5ex]
\hline
\rule{0ex}{3.0ex} 18  &
\scalebox{0.5}{%
\begin{tikzpicture}[->,>=stealth',shorten >=2pt,auto,node distance=2.5cm,ultra thick,baseline=-1.25ex]
  \tikzstyle{every state}=[fill=white,draw=black,text=black]

  \node[state,draw=red,text=red] (A)                    {\huge 2};
  \node[state,draw=red,text=red] (B) [right of=A] {\huge 1};
  \node[state,draw=red,text=red] (C) [right of=B] {\huge 1};
  \node[state] (D) [right of=C] {\huge 1};
  \node[state] (E) [right of=D] {\huge 1};

\path (E) edge (D);
\path (D) edge (C);
\end{tikzpicture}%
}
\\[0.5ex]
\hline
\rule{0ex}{3.0ex} 19  &
\scalebox{0.5}{%
\begin{tikzpicture}[->,>=stealth',shorten >=2pt,auto,node distance=2.5cm,ultra thick,baseline=-1.25ex]
  \tikzstyle{every state}=[fill=white,draw=black,text=black]

  \node[state,draw=red,text=red] (A)                    {\huge 2};
  \node[state] (B) [right of=A] {\huge 1};
  \node[state,draw=red,text=red] (C) [right of=B] {\huge 1};
  \node[state,draw=red,text=red] (D) [right of=C] {\huge 1};
  \node[state,draw=red,text=red] (E) [right of=D] {\huge 1};
  \path (B) edge (A);
\end{tikzpicture}%
}
\\[0.5ex]
\hline
\rule{0ex}{3.0ex} 20  &
\scalebox{0.5}{%
\begin{tikzpicture}[->,>=stealth',shorten >=2pt,auto,node distance=2.5cm,ultra thick,baseline=-1.25ex]
  \tikzstyle{every state}=[fill=white,draw=black,text=black]

  \node[state,draw=red,text=red] (A)                    {\huge 2};
  \node[state,draw=red,text=red] (B) [right of=A] {\huge 1};
  \node[state,draw=red,text=red] (C) [right of=B] {\huge 1};
  \node[state,draw=red,text=red] (D) [right of=C] {\huge 1};
  \node[state] (E) [right of=D] {\huge 1};
  \path (E) edge (D);
\end{tikzpicture}%
}
\\[0.5ex]
\hline
\rule{0ex}{3.0ex} 21  &
\scalebox{0.5}{%
\begin{tikzpicture}[->,>=stealth',shorten >=2pt,auto,node distance=2.5cm,ultra thick,baseline=-1.25ex]
  \tikzstyle{every state}=[fill=white,draw=black,text=black]

  \node[state,draw=red,text=red] (A)                    {\huge 2};
  \node[state,draw=red,text=red] (B) [right of=A] {\huge 1};
  \node[state,draw=red,text=red] (C) [right of=B] {\huge 1};
  \node[state,draw=red,text=red] (D) [right of=C] {\huge 1};
  \node[state,draw=red,text=red] (E) [right of=D] {\huge 1};
\end{tikzpicture}%
}
\\[0.5ex]
\hline
\end{tabular}}
\end{tabular}
\bigskip
\caption{Weighted proximity graphs of cubic plane Cremona maps}\label{table0}
\end{table}

\begin{proof}[Proof of Theorem \ref{thm:graphs}]
Weighted proximity graphs of cubic plane Cremona maps have 5 vertices, one with weight 2 and the other four with weight 1. Moreover, the proximity inequalities implies that only the double point may have satellite points and there can be at most one of them. For the same reason, a simple base point may have at most one proximate point while the double point may have at most two proximate points. We claim that these conditions are enough to find the 21 weighted proximity graphs of cubic plane Cremona maps, that are listed in Table \ref{table0}.

Indeed, we may start from the weighted graph with no arrow, that is number 21 in the list of Table \ref{table0}. We then add one arrow at each time in such a way that the graph is still admissible and the weights satisfy the proximity inequalities for all vertices.
For example, if we add one arrow to graph 21, then we find exactly two non-isomorphic weighted proximity graphs, that are numbers 19 and 20. If we add a second arrow, then we find other 5 graphs, that are numbers 14--18. And so on: in the following step we find the graphs with three arrows, that are numbers 7--13. In the next step, we find number 2--6 with four arrows and finally there is only one graph, number 1, with five arrows.

This procedure has also been implemented in Maple, in order to double check it.
\end{proof}

\begin{definition}
Let us add to the weighted proximity graph $G_\varphi$ of a \emph{cubic} plane Cremona map $\varphi$ the list of lines passing through \emph{three} base points of $\varphi$.
Let us call this object the \emph{enriched weighted proximity graph} of $\varphi$.
\end{definition}

\begin{remark}\label{rem:line}
These lines are \emph{unexpected}, in the sense that three points in general position are not aligned.

A line through three base points of a cubic plane Cremona map $\varphi$ cannot pass through the (proper) base point of multiplicity 2, otherwise the linear system defining the map would be reducible by B\'ezout Theorem. For the same reason, a line cannot pass through all four simple base points of $\varphi$. Furthermore, there cannot be two different such lines, because they should have two points in common.
\end{remark}

\begin{notation}
The line passing through three base points of a cubic plane Cremona map are indicated as \textcolor{blue}{dashed blue} curve in the pictures of weighted proximity graphs.
\end{notation}

\begin{theorem}\label{thm:31graphs}
There are exactly 31 enriched weighted proximity graphs of cubic plane Cremona maps, up to isomorphism, listed in Table \ref{table2} at page \pageref{table2}.
\end{theorem}

\begin{proof}
Recall that a line $\ell$ passes through an infinitely near point $p$ only if $\ell$ passes through the proper point $q$ such that $p\infnear q$ and the strict transform of $\ell$ passes through $p$.
Therefore, the enriched weighted proximity graph cannot include a line passing through a base point infinitely near the base point of multiplicity 2, by Remark \ref{rem:line}.

Hence, there is no line through three base points in the weighted proximity graphs 1--11, 14 and 15 in Table \ref{table0}.

Let us denote by $p_1$ the base point of multiplicity 2 and by $p_2,\ldots,p_5$ the other simple base points going from left to right in the pictures of the weighted proximity graphs in Table \ref{table0}.

The weighted proximity graph 12 in Table \ref{table0} may have a line through the proper simple base point $p_3$ and both of its infinitely near base points, that are $p_4$ and $p_5$. Accordingly, we find the two enriched weighted proximity graphs 10 and 11 in Table \ref{table2}.

Similarly, the weighted proximity graph 13 in Table \ref{table0} may have a line through $p_2$, $p_3$, $p_4$ and we find the two enriched weighted proximity graphs 8 and 9 in Table \ref{table2}.

Then, the weighted proximity graph 16 in Table \ref{table0} may have a line through $p_3$, $p_4$, $p_5$ and we find the two enriched weighted proximity graphs 22 and 23 in Table \ref{table2}.

The weighted proximity graph 17 in Table \ref{table0} may have either a line through $p_2$, $p_4$, $p_5$ or a line through $p_2$, $p_3$, $p_4$, that however give two isomorphic enriched weighted proximity graphs, hence we find the two enriched weighted proximity graphs 20 and 21 in Table \ref{table2}.

The weighted proximity graph 18 in Table \ref{table0} may have either a line through $p_3$, $p_4$, $p_5$ or a line through $p_2$, $p_3$, $p_4$. Accordingly, we find the three enriched weighted proximity graphs 17, 18 and 19 in Table \ref{table2}.

The weighted proximity graph 19 in Table \ref{table0} may have a line through $p_3$, $p_4$, $p_5$ and we find the two enriched weighted proximity graphs 28 and 29 in Table \ref{table2}.

The weighted proximity graph 20 in Table \ref{table0} may have either a line through $p_2$, $p_3$, $p_4$ or a line through $p_2$, $p_4$, $p_5$. (There could be also a line through $p_3$, $p_4$, $p_5$ but the resulting enriched weighted proximity graph would be isomorphic to a previous one.) Accordingly, we find the three enriched weighted proximity graphs 24, 25 and 26 in Table \ref{table2}.

Finally, the weighted proximity graph 21 in Table \ref{table0} may have four different lines that however give four isomorphic enriched weighted proximity graph. Hence we find the two enriched weighted proximity graphs 30 and 31 in Table \ref{table2}.
\end{proof}

\section{Height of points with respect to a plane Cremona map}

\begin{definition}
Let $\varphi$ be a plane Cremona map.
Let us define the \emph{height $h_\varphi(p)$ of a point $p\in\BB(\PP^2)$ with respect to $\varphi$} as follows:
\begin{equation*}\label{eq:height}
\h_\varphi(p)=\begin{cases}
0 & \text{if $p$ is not a base point of $\varphi$,} \\
1 & \text{if $p$ is a proper base point of $\varphi$,} \\
n+1 & \text{if $p$ is a base point of $\varphi$ and $p\infnear_n p'\in\PP^2$.}
\end{cases}
\end{equation*}
\end{definition}

\begin{definition}
Let $\varphi$ be a plane Cremona map. Let us also define the \emph{load} of a proper base point $p$ with respect to $\varphi$ as follows:
\begin{equation*}\label{eq:load}
\load_{\varphi}(p) = \sharp \big\lbrace q \text{ is a base point of } \varphi \: \big\vert \: q \infnear p \big\rbrace +1,
\end{equation*}
that is the number of base points of $\varphi$ which are infinitely near $p$ increased by $1$.
\end{definition}

\begin{remarks}
If $p$ is a simple proper base point of $\varphi$, then the proximity inequality implies that base points that are infinitely near $p$ cannot be satellite; in other words, there is a sequence $p_n \succ_1 p_{n-1} \succ_1 \cdots \succ_1 p_1 \succ_1 p$ where $p_i$ is a base point infinitely near $p$ of order $i$, $i=1,\ldots,n$; therefore, $\load_{\varphi}(p)$ is equal to the maximum height of base points that are infinitely near $p$.

If $\varphi$ is a de Jonqui\`eres map of degree $d$ and it has a unique proper base point $p$, then $\load_{\varphi}(p) = 2d-1$.
\end{remarks}

\begin{notation}\label{not:psi(p)}
Let $\varrho$ be an involutory ordinary quadratic map and let $p_1,p_2,p_3\in\PP^2$ be its base points.
Denote by $\ell_1$ ($\ell_2$, $\ell_3$, resp.) the line passing through $p_2$ and $p_3$ ($p_1$ and $p_3$, $p_1$ and $p_2$, resp.) and denote by $T$ the triangle $\ell_1\cup\ell_2\cup\ell_3$.

Let us define a bijection $\bar\varrho\colon\BB(\PP^2)\to\BB(\PP^2)$ induced by $\varrho$ as follows:
\begin{itemize}
\item $\bar\varrho(p) = p$, if $p=p_i$, $i=1,2,3$;
\item $\bar\varrho(p) = \varrho(p)$, if $p\in\PP^2\setminus T$;
\item $\bar\varrho(p)$ is the point infinitely near $p_i$ of order 1 in the direction of the strict transform of the line passing through $p_i$ and $p$, if $p\in\ell_i\setminus\{p_j,p_k\}$, $\{i,j,k\}=\{1,2,3\}$;
\item $\bar\varrho(p)$ is the point infinitely near $p_j$ of order 1 in the direction of the line $\ell_i$, if $p$ is the point infinitely near $p_i$ of order 1 in the direction of the line $\ell_j$, where $\{i,j\}\subset\{1,2,3\}$;
\item $\bar\varrho(p)$ is the point $q\in \ell_i$ such that the line passing through $p_i$ and $q$ is the strict transform of the line passing through $p_i$ in the direction of the point $p$, if $p$ is infinitely near $p_i$ of order 1 (not lying on $\ell_j$ and $\ell_k$, $\{i,j,k\}=\{1,2,3\}$);
\item $\bar\varrho(p)$ is the point infinitely near $\varrho(p')$ of order $n$ in the direction of the strict transform of a plane curve $C$, if $p$ is infinitely near $p'\in\PP^2$ and $C$ is a curve passing through $p$.
\end{itemize}
Let us say that $\bar\varrho(p)\in\BB(\PP^2)$ is the point corresponding to $p\in\BB(\PP^2)$ via $\varrho$.
\end{notation}

\begin{proposition}\label{PropMAC118a}
Let $p_1, p_2, p_3$ be the base points of an involutory ordinary quadratic plane Cremona map $\varrho\colon \mathbb{P}^2 \dashrightarrow \mathbb{P}^2$.
Let $\varphi\colon \mathbb{P}^2 \dashrightarrow \mathbb{P}^2$ be a plane Cremona map of degree $d>1$ with base points $p_4,\ldots,p_r$ and possibly $p_1,p_2,p_3$. Denote by $m_i$ the multiplicity of $\varphi$ at $p_i$, $i=1,\ldots,r$
(that is $m_i=0$ if $p_i$ is not a base point of $\varphi$, $i=1,2,3$).
Denote by $\bar\varrho(p)$ the (possibly infinitely near) point corresponding to $p$ via $\varrho$ as in Notation \ref{not:psi(p)}.

Then, the composite map $\varphi\circ \varrho^{-1}=\varphi\circ \varrho$ has degree $d-\varepsilon$, where
\[
\varepsilon = m_1+m_2+m_3-d,
\]
and it has $\bar\varrho(p_i)$, $i=4,\ldots,r$, as base point of multiplicity $m_i$. Furthermore, it has multiplicity $m_i-\varepsilon\geq0$ at $p_i$, $i=1,2,3$ (that is, $p_i$ is not a base point of $\varphi\circ \varrho$ when $\varepsilon=m_i$).
\end{proposition}

\begin{proof}
Cf.\ Proposition 4.2.5 in \cite{MAC}.
\end{proof}

\begin{lemma}\label{lem:oq+-1}
Let $\varphi$ be a plane Cremona map and $\varrho$ an involutory ordinary quadratic map.
If $p\in\BB(\PP^2)$ and $\bar p=\bar\varrho(p)\in\BB(\PP^2)$ as in Notation \ref{not:psi(p)}, then
\[
-1 \leq \h_\varphi(p)- \h_{\varphi\circ\varrho}(\bar p)\leq 1.
\]
\end{lemma}

\begin{proof}
Set $\varphi'=\varphi\circ\varrho$.
Let us see the possible cases:
\begin{itemize}
\item if $p$ is not a base point of $\varphi$, that is $\h_\varphi(p)=0$, then either $\bar p$ is not a base point of $\varphi'$ or $\bar p$ is a proper base point of $\varphi'$ by Proposition \ref{PropMAC118a} and Notation \ref{not:psi(p)}. In the former case, one has $\h_{\varphi'}(\bar p)=0$, whereas in the latter case one has $\h_{\varphi'}(\bar p)=1$, and the assertion follows;
\item if $p$ is a proper base point of $\varphi$, that is $\h_\varphi(p)=1$, then Proposition \ref{PropMAC118a} and Notation \ref{not:psi(p)} imply that three cases may occur:
\begin{enumerate}
\item $\bar p$ is not a base point of $\varphi'$,
\item $\bar p$ is still a proper base point of $\varphi'$,
\item $\bar p$ is a base point of $\varphi'$ which is infinitely near (of order 1) a proper base point,
\end{enumerate}
accordingly, one has $\h_{\varphi'}(\bar p)=0$, $\h_{\varphi'}(\bar p)=1$, $\h_{\varphi'}(\bar p)=2$, and the assertion follows;
\item if $p$ is a base point of $\varphi$ and $p$ is infinitely near $p'$ of order $n$, where $p'$ is a proper base point of $\varphi$, that is $\h_\varphi(p)=n+1$ and $\h_\varphi(p')=1$, then the previous analysis shows that $0\leq\h_{\varphi'}(\bar p')\leq 2$ and accordingly $n\leq \h_{\varphi'}(\bar p) \leq n+2$, that is the assertion.
\end{itemize}
We conclude that the assertion holds in any case.
\end{proof}

\begin{proposition}\label{oq>height}
Let $\varphi$ be a plane Cremona map. Then
\[
\oq(\varphi) \geq \max\{ \h_\varphi(p) \mid p \in \BB(\PP^2) \}.
\]
\end{proposition}

\begin{proof}
Let us set $n=\oq(\varphi)$ and let
\[
\varphi=\alpha\circ\varrho_n\circ\varrho_{n-1}\circ\cdots\circ\varrho_2\circ\varrho_1
\]
be a minimal decomposition of $\varphi$, where $\varrho_i$, $i=1,\ldots,n$, is an involutory ordinary quadratic map and $\alpha$ is an automorphism of $\PP^2$.
We proceed by induction on $n$.
Let us set
\[
m(\varphi)=\max\{ \h_\varphi(p) \mid p \in \BB(\PP^2) \} .
\]
The assertion is clearly true for $n=0,1$ because an automorphism has no base point and an ordinary quadratic map has exactly three points of height 1.

We then suppose that $n\geq2$ and we denote $\varphi\circ\varrho_1$ by $\varphi'$, so that $\oq(\varphi')=n-1$ and by induction hypothesis $n-1\geq m(\varphi')$. 
Now Lemma \ref{lem:oq+-1} implies that
\[
\h_{\varphi'}(\bar\varrho_1(p)) \geq \h_{\varphi}(p)-1,
\]
for any $p\in\BB(\PP^2)$, hence $m(\varphi') \geq m(\varphi)-1$. Therefore, we conclude that
\[
n=\oq(\varphi) = (n-1) + 1 \geq m(\varphi')+1 \geq m(\varphi),
\]
that is the assertion.
\end{proof}

\section{Comparison with the classification in \cite{C-D}}\label{sec comparison with CD}

In this section we compare our classification with the one in \cite{C-D}. 

We will freely use Notation \ref{not:1} at page \pageref{not:1}.

\begin{remark}
The classification in \cite{C-D} is not complete. Our type $15$ does not occur in their list, even if it is equivalent to the inverse of their type $11$.
\end{remark}

\begin{remark}
The type 17 in \cite{C-D}, that we denote by $\psi_{17}$, is equivalent to our type 28 in Table \ref{table1} with $\gamma_0=-1$, that we denote by $\varphi_{28,\gamma_0}$, because
$$[y: y+z: x]\circ \varphi_{28,\gamma_0} = \psi_{17}.$$
 However, our type 28 with $\gamma \neq -1$ does not occur in the list in \cite{C-D}. This explains why we added $\dagger$ at type $17$ in the third column of Table \ref{table1}.
\end{remark}

\begin{remarks} \begin{itemize}
\item[$\clubsuit$] In \cite{C-D}, their type $19$, that we denote by $\psi_{19}$, is equivalent to their type $18$ with parameter $\gamma_0 = -3/\sqrt{2}$, that we denote by $\psi_{18,\gamma_0}$. Indeed, one has
$$\psi_{19} = [2x-z:z: \sqrt{2}y-2x] \circ \psi_{18,\gamma_0} \circ [x-y:\sqrt{2}x: \sqrt{2}(z-y)].$$
\item[$\clubsuit$] Similarly, in \cite{C-D}, their type $31$, that we denote by $\psi_{31}$, is equivalent to their type $30$ with parameter $\gamma_0 = 3/\sqrt{2}$, that we denote by $\psi_{30,\gamma_0}$. Indeed, one has
$$\psi_{31} = [y+\sqrt{2}x: -y: 2(z-y)] \circ \psi_{30,\gamma_0} \circ [x+y: -\sqrt{2}y: x+2z].$$
\end{itemize}
This explains why types $19$ and $31$ in \cite{C-D} do not appear in the third column of Table \ref{table1} at page \pageref{table1}.
\end{remarks}

\begin{remarks}\label{rek CD}
\begin{itemize}
\item[$\clubsuit$] Let $\varphi_{24}$ be the map $24$ in Table \ref{table1}. Then, $\varphi_{24}$ is equivalent to type $15$ in \cite{C-D}, that we denote by $\psi_{15}$. Indeed, one has
$$\varphi_{24} \circ [x: x+y: x+y+z] = \psi_{15}.$$
\item[$\clubsuit$] Let $\varphi_{26,\gamma}$ be the map $26$ with parameter $\gamma \neq 0,1$ in Table \ref{table1}. Let $\psi_{29,t}$ be the map of type $29$ with parameter $t \neq 0,1$ in \cite{C-D}.
Then, one has that
$$[ty-x: t(y-tz): ty]\circ\varphi_{26,\gamma_0} \circ [-tx: y: y+z] = \psi_{29,t},$$
where $\gamma_0 = {1}/({1-t})$, that shows that $\varphi_{26,\gamma_0}$ is equivalent to $\psi_{29,t}$.
\item[$\clubsuit$] Let $\varphi_{27,\gamma}$ be the map $27$ with parameter $\gamma \neq 0,1$ in Table \ref{table1}. Let $\psi_{16,t}$ be the map of type $16$ with parameter $t$ such that $t^2 \neq 4$ in \cite{C-D}.
Then, one has that $\psi_{16,t}$ and 
$$\varphi_{27,\gamma_0} \circ [-(t_-x+y):t_+x+y:z],$$
where $\gamma_0 = {t_-}/{t_+}$, are defined by the same homaloidal net, therefore $\varphi_{27,\gamma_0}$  is equivalent to $\psi_{16,t}$.
\item[$\clubsuit$] Let $\varphi_{29,\gamma}$ be the map $29$ with parameter $\gamma \neq 0,1$ in Table \ref{table1}. Let $\psi_{30,t}$ be the map of type $30$ with parameter $t$ such that $t^2 \neq 4$ in \cite{C-D}.
Then, one has that $\psi_{30,t}$ and
$$\varphi_{29,\gamma_0} \circ [t^{\bullet}y:t_{+}(y+t_{-}x):t_{+}y+x+z)],$$
where $\gamma_0 = {1}/{2}+{t}/({2t^{\bullet}})$, are defined by the same homaloidal net, therefore $\varphi_{29,\gamma_0}$  is equivalent to $\psi_{30,t}$.
\item[$\clubsuit$] Let $\varphi_{30,\gamma}$ be the map $30$ with parameter $\gamma \neq 0,1$ in Table \ref{table1}. Let $\psi_{18,t}$ be the map of type $18$ with parameter $t$ such that $t^2 \neq 4$ in \cite{C-D}.
Then, one has that $\varphi_{30,\gamma_0}$, where $\gamma_0 = tt_{+}-1$, and
$$\psi_{18,t} \circ [x:-t_{+}y:t_{+}y-t_{-}x-t^{\bullet}z]$$
are defined by the same homaloidal net, therefore $\varphi_{30,\gamma_0}$ is equivalent to $\psi_{18,t}$.
\item[$\clubsuit$] Let  $\varphi_{31,a,b}$ be the map $31$ with two parameters $a,b$ such that $a \neq b$ and $a,b \neq 0,1$ in Table \ref{table1}. Let $\psi_{32,t,h}$ be the map of type $32$ with two parameters $t,h$ such that $t^2 \neq 4$ and $h \neq t_{\pm}$ in \cite{C-D}.
Then, one has that $\psi_{32,t,h}$ and
$$\varphi_{31,a_0,b_0} \circ [t^{\bullet}x:-t_{-}x-y:-t_{-}x-y-z], \qquad (a_0,b_0) = \bigg(\dfrac{(2-tt_{+})h}{h-t_{+}}, \dfrac{h}{h-t_{+}}\bigg),$$
are defined by the same homaloidal net, therefore $\varphi_{31,a_0,b_0}$ is equivalent to $\psi_{32,t,h}$.
\end{itemize}
\end{remarks}

\begin{remarks}
\begin{itemize}
\item[$\clubsuit$] Type $19$ in \cite{C-D}, that we denote by $\psi_{19}$, is equivalent to $\varphi_{30,-1}$, that is type $30$ in Table \ref{table1} with parameter $\gamma = -1$, because
$$[y-x+z:x-y:y-z]\circ \varphi_{30,-1}\circ [-x:y:z] = \psi_{19}.$$
\item[$\clubsuit$] Type $31$ in \cite{C-D}, that we denote by $\psi_{31}$, is equivalent to $\varphi_{29,-1}$, that is type $29$ in Table \ref{table1} with parameter $\gamma = -1$, because
$$[-y:2x-y-z:2x]\circ \varphi_{29,-1} \circ [x:y:2x+2z] = \psi_{31}.$$
\end{itemize}
\end{remarks}

\begin{remark}\label{CDdecomp}
In Section 6.4, Th\'eor\`eme 6.39 in \cite{C-D}, there is a list of decompositions in quadratic maps of their 32 types of cubic plane Cremona maps.
Note that the decompositions of types 25 and 26 are exchanged and the decomposition of type 24 is incorrect. A correct decomposition in quadratic maps of their type 24 (our type $7$) is:
\[
[x(x^2+yz): y(x^2+yz): xy^2] = [x : z : y] \circ \rho \circ [y : x+y : z] \circ \rho \circ [z : y : x].
\]
\end{remark}

\begin{longtable}[c]{|c|c|}
\hline
\rule{0ex}{2.5ex}$\sharp$ & A decomposition $\alpha_n\circ\textcolor{red}{\sigma}\circ\alpha_{n-1}\circ\textcolor{red}{\sigma}\circ\cdots\circ\alpha_1\circ\textcolor{red}{\sigma}\circ\alpha_0$ \\  \hline

   &  $\alpha_6=[27y+225z: 12y: 8x-8y], \alpha_5 = [2x+5y:5y-x:15x+15z],$ \\
1    & $\alpha_4=[2x+2z: 5x: 3x+10y-2z], \alpha_3 =[x-y: z+2y-x: 2y],$ \\
    &  $\alpha_2=[z: z-2x:2x+2y-z], \alpha_1=[x-y: z-x+y: 2x-y], \alpha_0=[y: y+z: x]$ \\
\hline
   & $\alpha_5=[8y-8x: x+z: 4x], \alpha_4=[x+y: y: z-x]$, \\
 2  &$\alpha_3=[2x: -y-2x: y+2x-2z], \alpha_2=[y-x: x: x+z-y],$ \\
  & $\alpha_1=[x: z-x: y], \alpha_0=[x: z: x+y]$ \\
\hline
  & $\alpha_5=[4y: 4y+3x: 4y+4z], \alpha_4=[3x-z: z-y: y],$ \\
3  & $\alpha_3=[9z+3x: y: 3z-y], \alpha_2=[3y+4z-x: x-z: 3x]$, \\
  & $\alpha_1=[y+z: x-y+z:y-z], \alpha_0=[2y: x+z: x-z]$, \\
\hline
  & $\alpha_4=[y+z: x+2z: z-y], \alpha_3=[2x: y-z: y+z],$\\[-0.75ex]
\raisebox{1ex}{4}  & $\alpha_2=[y-4x-4z: x: z], \alpha_1=[y+z: x: y-z], \alpha_0=[2y: x+z: x-z]$ \\
\hline
   & $\alpha_5=[4y+4z: 12z+x+9y: 6y+8z], \alpha_4=[2y+z: -2x-z: 2x+2z]$, \\
5 & $\alpha_3=[2y: 2y-z+x: z-y], \alpha_2=[2z+2x-y: 2z-y:y]$, \\
   & $\alpha_1=[y-x-z: 2z+x: x+z], \alpha_0=[x-z: y: z]$\\
\hline
  & $\alpha_4=[-2y-4x:4x:2x+y+2z], \alpha_3=[x-2z:z:y], \alpha_2=[y:z-2y-x:2x],$ \\[-0.75ex]
\raisebox{1ex}{6}  & $\alpha_1=[y-x:x+y:2z], \alpha_0=[x-y:x+y:x+y+2z]$	\\
\hline
  & $\alpha_4=[y-x: y: y-z], \alpha_3=[x: z: z+y], \alpha_2=[z: y-x-z: x],$ \\[-0.75ex]
\raisebox{1ex}{7} & $\alpha_1=[x: y: x+z], \alpha_0=[x: x+z: y-x]$  \\
\hline
   & $\alpha_5=[2y: -4x-4y: 8x+9y+z], \alpha_4=[2x-2y: 2y-x: x+2z]$, \\
8 & $\alpha_3=[x+y: x: z-y-2x], \alpha_2=[x+y: -2x: 2x+y+z]$, \\
   & $\alpha_1=[2x+y: y:2x-y+2z], \alpha_0=[-z: x+z: y+z]$ \\
\hline
 & $\alpha_4=[z-y:x:z],\alpha_3=[x:y+z:z], \alpha_2=[x:x-y:z],$\\[-0.75ex]
\raisebox{1ex}{9} &$\alpha_1=[x:y+z:z],\alpha_0=[x:z-y:y]$ \\ 
\hline
 & $\alpha_3=[x+y: -y-z: y], \alpha_2=[z-x:x+y:-y], $ \\[-0.75ex]
\raisebox{1ex}{10}  & $\alpha_1=[z: x-z:y+z-x], \alpha_0=[x: x+z: x+y]$  \\
\hline
11 & $\alpha_3=[z:x:y],\alpha_2=[x:x+z-y:z],\alpha_1=[x:y+z:z],\alpha_0=[x:z-y:y]$ \\
\hline
 & $\alpha_3=[-x: x-z: x+y], \alpha_2=[y-x: x: y+z],$ \\[-0.75ex]
\raisebox{1ex}{12}  & $\alpha_1=[y: x+y: z-x-y], \alpha_0=[-z: x+z: y-z]$ \\
\hline
13 & $\alpha_3=[x:y:z],\alpha_2=[z-y:z:x+z-y],\alpha_1=[x:y+z:z],\alpha_0=[z:x-y:y]$ \\
\hline
14 & $\alpha_3=[x+z:z:y],\alpha_2=[x:z-y:z-x],\alpha_1=[x:y+z:z],\alpha_0=[y:z-x:x]$ \\
\hline
 & $\alpha_3=[z-y : y+z : 4y-4x], \alpha_2=[x+y : y : z],$ \\[-0.75ex]
\raisebox{1ex}{15} & $\alpha_1=[y-x+2z : x-y : x+y], \alpha_0=[x:y:z]$ \\
\hline
  & $\alpha_3=[-x:y:2y-z],\alpha_2=[y:x:x+z],$\\[-0.75ex]
\raisebox{1ex}{16} & $\alpha_1=[x+z:-z:y-2x-2z],\alpha_0=[x:x+z:y-x]$ \\
\hline
  & $\alpha_4=[-y: x-y: 3y+z], \alpha_3=[x+y: y: z], \alpha_2=[z: x: y-x+z], $ \\[-0.75ex]
\raisebox{1ex}{17}  & $\alpha_1=[x: y-x: z], \alpha_0=[y: y+z: x-y]$  \\
\hline
& $\alpha_3=[x+z:z:z-y],\alpha_2=[x:y+z:z],$ \\
\raisebox{1ex}{18} & $\alpha_1=[y-x:z-y-x:x],\alpha_0=[x:y:z]$ \\
\hline
19 & $\alpha_3=[x:z:-y],\alpha_2=[y:z-y:x],\alpha_1=[x:z:y-x],\alpha_0=[x:x+z:y]$ \\
\hline
  & $\alpha_3=[z-y:z:x+z],\alpha_2=[x:y+z:z],$\\[-0.75ex]
\raisebox{1ex}{20} &$\alpha_1=[z-x-y:x-y:y],\alpha_0=[x:y:z]$ \\
\hline
21 & $\alpha_2=[x:y:z],\alpha_1=[x:y:x+y+z], \alpha_0=[x:y:z]$ \\
\hline
  & $\alpha_3=[y-2z:z:x+z],\alpha_2=[x:y+z:z],$\\[-0.75ex]
\raisebox{1ex}{22} &$\alpha_1=[x+y-z:2x+y:-x-y],\alpha_0=[x:y:z]$ \\
\hline
23 & $\alpha_2=[x:-y:z],\alpha_1=[y+z:z:x+y+z],\alpha_0=[z:x:-x-y]$ \\
\hline
  & $\alpha_3=[x:y:z],\alpha_2=[x+z:z-x:6z-4y],$\\[-0.75ex]
\raisebox{1ex}{24} &$\alpha_1=[x:y+z:z],\alpha_0=[y-2x: 2z-3y: y]$ \\
\hline
25 & $\alpha_2=[-x:z:y],\alpha_1=[z:x+y:y+z],\alpha_0=[z:y:-x-y]$ \\
\hline
& $\alpha_2=[\gamma (\gamma x-2x+y)+x+z:\gamma (\gamma x-x+y):\gamma (\gamma x+y)],$\\[-0.75ex]
\raisebox{1ex}{26} & $\alpha_1=[\gamma ((\gamma -1)x-\gamma y+z):\gamma (y-x):\gamma x],\alpha_0=[x:y:z]$ \\
\hline
& $\alpha_2=[\gamma (\gamma x+y):-\gamma (x+y):(\gamma -1)^2 z],\alpha_1=[\gamma x +y:-x-y:(\gamma -1)z],$\\[-0.75ex]
\raisebox{1ex}{27} & $\alpha_0=[x:y:z]$ \\
\hline
  & $\alpha_3=[z:\gamma^2(x+y):\gamma^2(x+\gamma x+\gamma y)],\alpha_2=[x+\gamma y-y:-\gamma y:z],$\\[-0.75ex]
\raisebox{1ex}{28} &$\alpha_1=[x:x-\gamma y:-\gamma z],\alpha_0=[x:y:z]$ \\
\hline
  & $\alpha_2=[y+z:y-x:\gamma (y-x-z)-x+y+z],$\\[-0.75ex]
\raisebox{1ex}{29} &$\alpha_1=[x-y:x-\gamma y:(1-\gamma)z-x+\gamma y],\alpha_0=[x:y:z]$ \\
\hline
  & $\alpha_2=[\gamma^2 x+(1-\gamma)y-z:\gamma(\gamma x-y):\gamma ((\gamma+1)x-y)],$\\[-0.75ex]
\raisebox{1ex}{30} &$\alpha_1=[\gamma (y+z)-y:y+z:x+z],\alpha_0=[z-x:y-x:x]$ \\
\hline
   & $\alpha_2=[a(a(x+(b-1)^2z)+by):a(ax+y):by-((b-1)z-x)a^2-$\\[-0.75ex]
31 &$-(b((1-b)z-x)-y)a],\alpha_1=[ax-by:y-x:(b-1)ax-b(a-1)y+(a-b)z],$ \\[-0.75ex]
   &$\alpha_0=[x:y:z]$ \\ 
\hline
\addlinespace 
\caption{Decompositions of the $31$ cubic plane Cremona maps listed in Table \ref{table1}.} \label{table3}
\end{longtable}

\begin{remark}
For types $9,11,13,14,18,19,20,21,22,23,24,25,26,27,29,30$ and $31$, the decompositions listed in \cite{C-D} are already minimal.
\end{remark}

\begin{table}[!h]
\centering
\begin{tabular}{|c|c|}
\hline
\rule{0ex}{2.5ex}$\sharp$ & A decomposition into quadratic maps \\ \hline
\rule{0ex}{2.5ex}1  & $[x:z:y]\circ\textcolor{red}{\rho}\circ[z:y:x]\circ\textcolor{red}{\tau}\circ[z:y:-x]\circ\textcolor{red}{\rho}\circ[x:z:y]$ \\[0.25ex] \hline
\rule{0ex}{2.5ex}2  & $[x+z:y:z]\circ\textcolor{red}{\rho}\circ[y-x:z:x]\circ\textcolor{red}{\tau}\circ[y:x:-z]$ \\[0.25ex] \hline
\rule{0ex}{2.5ex}3  & $[z:x:y]\circ\textcolor{red}{\rho}\circ[z:y:x]\circ\textcolor{red}{\tau}\circ[z:x:-y]$ \\[0.25ex] \hline
\rule{0ex}{2.5ex}4  & $[y:z:x]\circ\textcolor{red}{\tau}\circ [y:x:-z]\circ\textcolor{red}{\tau}\circ[x:z:-y]$ \\[0.25ex] \hline
\rule{0ex}{2.5ex}5  & $[x:z:y]\circ\textcolor{red}{\tau}\circ[-z:-y:x+z]\circ\textcolor{red}{\rho}\circ[y-z:x:z]$ \\[0.25ex] \hline
\rule{0ex}{2.5ex}6  & $[-y:z:x]\circ\textcolor{red}{\rho}\circ[x+y+2z:y+z:-z]\circ\textcolor{red}{\rho}\circ [x+z:x:y-x]$ \\[0.25ex] \hline
\rule{0ex}{2.5ex}7  & $[-x-z:z:y]\circ\textcolor{red}{\rho}\circ[-z-y:x+y+z:z]\circ\textcolor{red}{\rho}\circ[z-x:y:x]$ \\[0.25ex] \hline
\rule{0ex}{2.5ex}8  & $[y:x:-z]\circ\textcolor{red}{\tau}\circ[z:x+z:y]\circ\textcolor{red}{\rho}\circ[x-z:y:z]$ \\[0.25ex] \hline
\rule{0ex}{2.5ex}9  & $[y:-x-z:z]\circ\textcolor{red}{\rho}\circ[-2z-x:x+y+z:z]\circ\textcolor{red}{\rho}\circ[x-y: z:y]$ \\[0.25ex] \hline
\rule{0ex}{2.5ex}10 & $[y:x+z:z]\circ\textcolor{red}{\rho}\circ[z-y:x+z:y]\circ\textcolor{red}{\rho}\circ[z-x:y: x]$ \\[0.25ex] \hline
\rule{0ex}{2.5ex}11 & $[x:z:z-y]\circ\textcolor{red}{\rho}\circ [z:x+y+z:y]\circ\textcolor{red}{\rho}\circ[z-y:x : y]$ \\[0.25ex] \hline
\rule{0ex}{2.5ex}12 & $[z:x+z:y]\circ\textcolor{red}{\rho}\circ [x+z-y:z:y]\circ\textcolor{red}{\rho}\circ[y-z:x:z]$ \\[0.25ex] \hline
\rule{0ex}{2.5ex}13 & $[z:y:x]\circ\textcolor{red}{\sigma}\circ [y+x+z:z:y]\circ\textcolor{red}{\rho}\circ [z-y: x:y]$ \\[0.25ex] \hline
\rule{0ex}{2.5ex}14 & $[y:y+z:-x]\circ\textcolor{red}{\rho}\circ[x+z:z-y:y]\circ\textcolor{red}{\rho}\circ[z-x:y:x]$ \\[0.25ex] \hline
\rule{0ex}{2.5ex}15 & $[x:z+x:y]\circ\textcolor{red}{\rho}\circ[y:z:x-y]\circ\textcolor{red}{\sigma}$ \\[0.25ex] \hline
\rule{0ex}{2.5ex}16 & $[x:z:y+z]\circ\textcolor{red}{\rho}\circ[y:x-z-y:y+z]\circ\textcolor{red}{\sigma}\circ[x+z:y-x:x]$ \\[0.25ex] \hline
\rule{0ex}{2.5ex}17 & $[y:x:-z]\circ\textcolor{red}{\tau}\circ\textcolor{red}{\sigma}$ \\[0.25ex] \hline
\rule{0ex}{2.5ex}18 & $[x+z:z:-y]\circ\textcolor{red}{\rho}\circ[y-x:y-z:x]\circ\textcolor{red}{\sigma}$ \\[0.25ex] \hline
\rule{0ex}{2.5ex}19 & $[z:x:y+z]\circ\textcolor{red}{\rho}\circ[z:x-y-z:y]\circ\textcolor{red}{\sigma}\circ[x+z:y:x]$ \\[0.25ex] \hline
\rule{0ex}{2.5ex}20 & $[y:z:x+z]\circ\textcolor{red}{\rho}\circ[z-x-y:x:y]\circ\textcolor{red}{\sigma}$ \\[0.25ex] \hline
\rule{0ex}{2.5ex}22 & $[y-z:z:x+z]\circ\textcolor{red}{\rho}\circ[z-x-y:x:x+y]\circ\textcolor{red}{\sigma}$ \\[0.25ex] \hline
\rule{0ex}{2.5ex}24 & $[y+z:-z:x-z]\circ\textcolor{red}{\rho}\circ[x-y+z:y-x:x]\circ\textcolor{red}{\sigma}$ \\[0.25ex] \hline
\rule{0ex}{2.5ex}28 & $[x:z-y:2\gamma z-(1+\gamma) y]\circ \textcolor{red}{\rho} \circ [(1-\gamma)z: x-y: x-\gamma y]\circ\textcolor{red}{\sigma}$ \\[0.25ex] \hline
\end{tabular}
\bigskip
\caption{Decomposition into quadratic maps of some types.}\label{table4}
\end{table}

\begin{remark}
With those maps (including types $21,23,25,26,27,29,30,31$ listed in Table \ref{table1}), whose have the ordinary quadratic length exactly $2$ (hence their quadratic lengths are also $2$), then their minimal decompositions into quadratic maps can be found in Table \ref{table3}.
\end{remark}

\section{Proof of the classification Theorem \ref{thm main1}}\label{proof main theorem1}

According to Theorem \ref{thm:31graphs}, there are 31 enriched weighted proximity graphs of cubic plane Cremona maps, listed in Table \ref{table2} at page \pageref{table2}.
We will show that a  cubic plane Cremona map with enriched weighted proximity graph of type $n$, $1\leq n \leq 31$, in Table \ref{table2} is equivalent to the map of type $n$ in Table \ref{table1} at page \pageref{table1}.

\begin{lemma}
Let $\varphi_1$ be the map 1 in Table \ref{table1} and let $\psi_1$ be a map with enriched weighted proximity graph 1 in Table  \ref{table2}. Then, $\psi_1$ is equivalent to $\varphi_1$.
\end{lemma}

\begin{proof}
The base points of $\varphi_1$ are $p_0 = [1: 0: 0]$ of multiplicity $2$ and $p_1,\ldots,p_4$ with $p_4\succ_{1}p_3\succ_{1}p_2\succ_{1}p_1\succ_{1}p_0$ and $p_2 \satel p_0$, whose standard coordinates are $p_1 =(p_0,0)$, $p_2=(p_0,0,\infty)$, $p_3=(p_0,0,\infty,-1)$, $p_4=(p_0,0,\infty,-1,0)$.

The base points of $\psi_1$ are $q_0$ of multiplicity 2 and $q_1,\ldots,q_4$ with $q_4\succ_{1}q_3\succ_{1}q_2\succ_{1}q_1\succ_{1}q_0$ and $q_2 \satel p_0$. Clearly, there exists an automorphism $\alpha_1$ of $\PP^2$ such that $\alpha_1(p_0)=q_0$ and $\alpha_1(p_1)=q_1$, so that also $\alpha_1(p_2)=q_2$.

The base points of $\psi_1\circ\alpha_1$ are then $p_0, p_1, p_2, q'_3, q'_4$ where $q'_3$ has standard coordinates $q'_3=(p_0,0,\infty,u_3)$  for some $u_3 \in \C^*$ because, if $u_3$ were 0, then $q'_3$ would be proximate to $p_0$, a contradiction, and, if $u_3$ were $\infty$, then $q'_3$ would be proximate to $p_1$, again a contradiction.

An automorphism $\alpha_2$ of $\PP^2$ that fixes $p_0,p_1,p_2$ and that maps $p_3$ to $q'_3$ is
$$\alpha_2([x:y:z]) = [-x: u_3 y: u_3 z].$$
The base points of $\psi_1\circ\alpha_1\circ\alpha_2$  are then $p_0, p_1, p_2, p_3, q''_4$ where $q''_4$ has standard coordinates $q''_4=(p_0,0,\infty,-1,u_4)$ for some $u_4 \in \C$ because, if $u_4$ were $\infty$, then $q''_4$ would be proximate to $p_2$, a contradiction.

An automorphism $\alpha_3$ of $\PP^2$ that fixes $p_0,p_1,p_2,p_3$ and that maps $p_4$ to $q''_4$ is
$$\alpha_3([x:y:z]) = [3x : 3y+u_{4}z : 3z].$$
Therefore, the maps $\varphi_1$ and $\psi_1\circ\alpha_1\circ\alpha_2\circ\alpha_3$ are defined by the same homaloidal net and, hence, $\varphi_1$ and $\psi_1$ are equivalent.
\end{proof}

\begin{lemma}
Let $\varphi_2$ be the map 2 in Table \ref{table1} and let $\psi_2$ be a map with enriched weighted proximity graph 2 in Table  \ref{table2}. Then, $\psi_2$ is equivalent to $\varphi_2$.
\end{lemma}

\begin{proof}
The base points of $\varphi_2$ are $p_0=[0:0:1]$ of multiplicity 2 and $p_1,\ldots,p_4$ with standard coordinates $p_1=(p_0,0)$, $p_2=(p_0,0,-1)$, $p_3=(p_0,0,-1,0)$, $p_4=(p_0,0,-1,0,0)$. So there is a unique irreducible conic passing through $p_0, \ldots, p_4$, that is $C_1 \colon x^2+yz=0$. Let $q_0$ be the double base point of $\psi_2$ and let $q_1,\ldots,q_4$ be the simple base points of $\psi_2$. According to Lemma \ref{lm conic5}, there is a unique irreducible conic $C_2$ passing through $q_0,\ldots,q_4$. Moreover, Lemma \ref{lm two_conic_equi2} implies that there exists an automorphism $\alpha$ of $\PP^2$ such that $\alpha(q_0)=p_0$ and $\alpha(C_2)=C_1$.
This forces $\alpha(q_i)=p_i$, $i=1,2,3,4$. Therefore, $\psi_2$ is equivalent to $\varphi_2$.
\end{proof}

\begin{lemma}
Let $\varphi_3$ be the map 3 in Table \ref{table1} and let $\psi_3$ be a map with enriched weighted proximity graph 3 in Table  \ref{table2}. Then, $\psi_3$ is equivalent to $\varphi_3$.
\end{lemma}

\begin{proof}
The base points of $\varphi_3$ are $p_0 = [0:1:0]$ of multiplicity 2 and $p_1, p_2, p_3, p_4$ where $p_1\succ_{1}p_0$ and $p_4\succ_{1}p_3\succ_{1}p_2\succ_{1}p_0$ with standard coordinates $p_1 =(p_0,\infty), p_2=(p_0,0),p_3=(p_0,0,-1)$ and $p_4=(p_0,0,-1,0)$.

The base points of $\psi_3$ are $q_0$ of multiplicity 2 and $q_1,\ldots,q_4$ where $q_1\succ_{1}q_0$ and $q_4\succ_{1}q_3\succ_{1}q_2\succ_{1}q_0$. Clearly, there exists an automorphism $\alpha_1$ of $\PP^2$ such that $\alpha_1(q_i)=p_i$ for $i=0,1,2$.

The base points of $\psi_3\circ\alpha_1$ are then $p_0, p_1, p_2, q'_3, q'_4$ where $q'_3$ has standard coordinates $q'_3=(p_0,0,u_3)$  for some $u_3 \in \C^*$ because, if $u_3$ were 0, then $q'_3$ would be aligned with $p_0$ and $p_2$, a contradiction, and, if $u_3$ were $\infty$, then $q'_3$ would be proximate to $p_0$, again a contradiction.

An automorphism $\alpha_2$ of $\PP^2$ that fixes $p_0,p_1,p_2$ and that maps $p_3=(p_0,0,-1)$  to $q'_3=(p_0,0,u_3)$ is 
$$\alpha_2([x:y:z]) = [u_3 x: -y: u_3 z].$$
The base points of $\psi_3\circ\alpha_1\circ\alpha_2$  are then $p_0, p_1, p_2, p_3, q''_4$ where $q''_4$ has standard coordinates $q''_4=(p_0,0,-1,u_4)$ for some $u_4 \in \C$ because, if $u_4$ were $\infty$, then $q''_4$ would be proximate to $p_2$, a contradiction.

An automorphism $\alpha_3$ of $\PP^2$ that fixes $p_0,p_1,p_2,p_3$ and that maps $p_4$ to $q''_4$ is
$$\alpha_3([x:y:z]) = [x : y-u_{4}x : z].$$
Therefore, the maps $\varphi_3$ and $\psi_3\circ\alpha_1\circ\alpha_2\circ\alpha_3$ are defined by the same homaloidal net and, hence, $\varphi_3$ and $\psi_3$ are equivalent.
\end{proof}

\begin{lemma}
Let $\varphi_4$ be the map 4 in Table \ref{table1} and let $\psi_4$ be a map with enriched weighted proximity graph 4 in Table  \ref{table2}. Then, $\psi_4$ is equivalent to $\varphi_4$.
\end{lemma}

\begin{proof}
The base points of $\varphi_{4}$ are $p_0=[0:1:0]$ of multiplicity 2 and $p_1,\ldots,p_4$ where $p_3\succ_1p_1\succ_1p_0$ and $p_4\succ_1p_2\succ_1p_0$, with standard coordinates $p_1=(p_0,\infty)$, $p_3=(p_0,\infty,-1)$, $p_2=(p_0,0)$ and $p_4=(p_0,0,-1)$.

The base points of $\psi_4$ are $q_0$ of multiplicity 2 and $q_1,\ldots,q_4$ where $q_3\succ_1q_1\succ_1q_0$ and $q_4\succ_1q_2\succ_1q_0$.
Clearly, there exists an automorphism $\alpha_1$ of $\PP^2$ such that $\alpha_1(p_i)=q_i$ for $i=0,1,2$.

The base points of $\psi_4\circ\alpha_1$ are then $p_0, p_1, p_2, q'_3, q'_4$ where $q'_3$ has standard coordinates $q'_3=(p_0,\infty,u_3)$  for some $u_3 \in \C^*$ because, if $u_3$ were 0, then $q'_3$ would be aligned with $p_0$ and $p_1$, a contradiction, and, if $u_3$ were $\infty$, then $q'_3$ would be proximate to $p_0$, again a contradiction.

An automorphism $\alpha_2$ of $\PP^2$ that fixes $p_0,p_1,p_2$ and that maps $p_3=(p_0,\infty,-1)$ to $q'_3=(p_0,\infty,u_3)$  is 
$$\alpha_3([x:y:z]) = [-u_3x: y: z].$$
The base points of $\psi_4\circ\alpha_1\circ\alpha_2$  are then $p_0, p_1, p_2, p_3, q''_4$ where $q''_4$ has standard coordinates $q''_4=(p_0,0,u_4)$ for some $u_4 \in \C^*$ because, if $u_4$ were 0, then $q'_4$ would be aligned with $p_0$ and $p_2$, a contradiction, and if $u_4$ were $\infty$, then $q''_4$ would be proximate to $p_0$, a contradiction.

An automorphism $\alpha_3$ of $\PP^2$ that fixes $p_0,p_1,p_2,p_3$ and that maps $p_4=(p_0,0,-1)$ to $q''_4=(p_0,0,u_4)$  is 
$$\alpha_3([x:y:z]) = [(-u_4)^{2/3} x: y: (-u_4)^{1/3}z].$$
Therefore, the maps $\varphi_4$ and $\psi_4\circ\alpha_1\circ\alpha_2\circ\alpha_3$ are defined by the same homaloidal net and, hence, $\varphi_4$ and $\psi_4$ are equivalent.
\end{proof}

\begin{lemma}
Let $\varphi_5$ be the map 5 in Table \ref{table1} and let $\psi_5$ be a map with enriched weighted proximity graph 5 in Table  \ref{table2}. Then, $\psi_5$ is equivalent to $\varphi_5$.
\end{lemma}

\begin{proof}
The base points of $\varphi_{5}$ are $p_0=[0:1:0]$ of multiplicity $2$, $p_1=[1:0:0]$ and $p_2,p_3,p_4$ where $p_4\succ_1p_3\succ_1p_2\succ_1p_0$ and $p_3\satel p_0$, with standard coordinates $p_2=(p_0,\infty)$, $p_3=(p_0,\infty,\infty)$ and $p_4=(p_0,\infty,\infty,-1)$.

The base points of $\psi_5$ are $q_0\in\PP^2$ of multiplicity 2, $q_1\in\PP^2$ and $q_2,q_3,q_4$ where $q_4\succ_1q_3\succ_1q_2\succ_1q_0$ and $q_3\satel q_0$.
Clearly, there exists an automorphism $\alpha_1$ of $\PP^2$ such that $\alpha_1(p_i)=q_i$ for $i=0,1,2$. It follows that also $\alpha_1(p_3)=q_3$.

The base points of $\psi_5\circ\alpha_1$ are then $p_0, p_1, p_2, p_3, q'_4$ where $q'_4$ has standard coordinates $q'_4=(p_0,\infty,\infty,u_4)$  for some $u_4 \in \C^*$ because, if $u_4$ were 0, then $q'_4$ would be proximate to $p_0$, a contradiction, and, if $u_4$ were $\infty$, then $q'_4$ would be proximate to $p_2$, again a contradiction.

An automorphism $\alpha_2$ of $\PP^2$ that fixes $p_0,p_1,p_2,p_3$ and that maps $p_4=(p_0,\infty,\infty,-1)$ to $q'_4=(p_0,\infty,\infty,u_4)$  is 
$$\alpha_2([x:y:z]) = [x: -u_4 y: z].$$
Therefore, the maps $\varphi_5$ and $\psi_5\circ\alpha_1\circ\alpha_2$ are defined by the same homaloidal net and, hence, $\varphi_5$ and $\psi_5$ are equivalent.
\end{proof}

\begin{lemma}
Let $\varphi_6$ be the map 6 in Table \ref{table1} and let $\psi_6$ be a map with enriched weighted proximity graph 6 in Table  \ref{table2}. Then, $\psi_6$ is equivalent to $\varphi_6$.
\end{lemma}

\begin{proof}
The base points of $\varphi_{6}$ are $p_0=[0:0:1]$ of multiplicity $2$, $p_1=[1:1:-1]$ and $p_2,p_3,p_4$ where $p_2\succ_1p_0$ and $p_4\succ_1p_3\succ_1p_0$, with standard coordinates $p_2=(p_0,0)$, $p_3=(p_0,\infty)$ and $p_4=(p_0,\infty,-1)$.

The base points of $\psi_6$ are $q_0\in\PP^2$ of multiplicity 2, $q_1\in\PP^2$ and $q_2,q_3,q_4$ where $q_2\succ_1q_0$ and $q_4\succ_1q_3\succ_1q_0$.
Clearly, there exists an automorphism $\alpha_1$ of $\PP^2$ such that $\alpha_1(p_i)=q_i$ for $i=0,1,2,3$.

The base points of $\psi_6\circ\alpha_1$ are then $p_0, p_1, p_2, p_3, q'_4$ where $q'_4$ has standard coordinates $q'_4=(p_0,\infty,u_4)$  for some $u_4 \in \C^*$ because, if $u_4$ were 0, then $q'_4$ would be aligned with $p_0$ and $p_3$, a contradiction, and, if $u_4$ were $\infty$, then $q'_4$ would be proximate to $p_0$, again a contradiction.

An automorphism $\alpha_2$ of $\PP^2$ that fixes $p_0,p_1,p_2,p_3$ and that maps $p_4=(p_0,\infty,-1)$ to $q'_4=(p_0,\infty,u_4)$  is 
$$\alpha_2([x:y:z]) = [x: y: (-u_4-1) x-u_4 z].$$
Therefore, the maps $\varphi_6$ and $\psi_6\circ\alpha_1\circ\alpha_2$ are defined by the same homaloidal net and, hence, $\varphi_6$ and $\psi_6$ are equivalent.
\end{proof}

\begin{lemma}
Let $\varphi_7$ be the map 7 in Table \ref{table1} and let $\psi_7$ be a map with enriched weighted proximity graph 7 in Table  \ref{table2}. Then, $\psi_7$ is equivalent to $\varphi_7$.
\end{lemma}

\begin{proof}
The base points of $\varphi_7$ are $p_0=[0:0:1]$ of multiplicity 2, $p_1=[0:1:0]$ and $p_2, p_3, p_4$ where $p_4\succ_1 p_3\succ_1 p_2\succ_1 p_0$ with standard coordinates $p_2=(p_0,0)$, $p_3=(p_0,0,-1)$, $p_4=(p_0,0,-1,\infty)$. So there is a unique irreducible conic passing through $p_0, \ldots, p_4$, that is $C_1 \colon x^2+yz=0$. The base points of $\psi_7$ are $q_0$ of multiplicity 2 and $q_1,\ldots,q_4$ where $q_1\in\PP^2$ and $q_4\succ_1 q_3\succ_1 q_2\succ_1 q_0$. According to Lemma \ref{lm conic4}, there is a unique irreducible conic $C_2$ passing through $q_0,\ldots,q_4$. Moreover, Lemma \ref{lm two_conic_equi2} implies that there exists an automorphism $\alpha$ of $\PP^2$ such that $\alpha(C_1)=C_2$ and $\alpha(p_i)=q_i$, $i=0,1$.
This forces $\alpha(p_i)=q_i$, $i=2,3,4$. Therefore, $\psi_7$ is equivalent to $\varphi_7$.
\end{proof}

\begin{lemma}
Let $\varphi_8$ be the map 8 in Table \ref{table1} and let $\psi_8$ be a map with enriched weighted proximity graph 8 in Table  \ref{table2}. Then, $\psi_8$ is equivalent to $\varphi_8$.
\end{lemma}

\begin{proof}
The base points of $\varphi_{8}$ are $p_0=[0:1:0]$ of multiplicity $2$, $p_1=[1:0:0]$ and $p_2,p_3,p_4$ where $p_4\succ_1p_3\succ_1p_2\succ_1p_1$ with standard coordinates $p_2=(p_1,\infty), p_3=(p_1,\infty,0)$ and $p_4=(p_1,\infty,0,1)$.

The base points of $\psi_8$ are $q_0$ of multiplicity 2, $q_1\in\PP^2$ and $q_2,q_3,q_4$ where $q_4\succ_1q_3\succ_1q_2\succ_1q_1$ and $q_3$ is aligned with $q_1$ and $q_2$.
Clearly, there exists an automorphism $\alpha_1$ of $\PP^2$ such that $\alpha_1(p_i)=q_i$ for $i=0,1,2$.
It follows that also $\alpha_1(p_3)=q_3$.

The base points of $\psi_8\circ\alpha_1$ are then $p_0, p_1, p_2, p_3, q'_4$ where $q'_4$ has standard coordinates $q'_4=(p_1,\infty,0,u_4)$  for some $u_4 \in \C^*$ because, if $u_4$ were 0, then $q'_4$ would be aligned with $p_1, p_2, p_3$, a contradiction, and, if $u_4$ were $\infty$, then $q'_4$ would be proximate to $p_2$, again a contradiction.

An automorphism $\alpha_2$ of $\PP^2$ that fixes $p_0,p_1,p_2,p_3$ and that maps $p_4=(p_1,\infty,0,1)$ to $q'_4=(p_0,\infty,0,u_4)$  is 
$$\alpha_2([x:y:z]) = [x: u_4 y: z].$$
Therefore, the maps $\varphi_8$ and $\psi_8\circ\alpha_1\circ\alpha_2$ are defined by the same homaloidal net and, hence, $\varphi_8$ and $\psi_8$ are equivalent.
\end{proof}

\begin{lemma}
Let $\varphi_9$ be the map 9 in Table \ref{table1} and let $\psi_9$ be a map with enriched weighted proximity graph 9 in Table  \ref{table2}. Then, $\psi_9$ is equivalent to $\varphi_9$.
\end{lemma}

\begin{proof}
The base points of $\varphi_9$ are $p_0=[0:0:1]$ of multiplicity 2, $p_1=[1:0:0]$ and $p_2, p_3, p_4$ where $p_4\succ_1 p_3\succ_1 p_2\succ_1 p_1$ with standard coordinates $p_2=(p_1,0)$, $p_3=(p_1,0,-1)$, $p_4=(p_1,0,-1,0)$. So there is a unique irreducible conic passing through $p_0, \ldots, p_4$, that is $C_1 \colon xz+y^2=0$. The base points of $\psi_9$ are $q_0$ of multiplicity 2 and $q_1,\ldots,q_4$ where $q_1\in\PP^2$ and $q_4\succ_1 q_3\succ_1 q_2\succ_1 q_1$. According to Lemma \ref{lm conic4}, there is a unique irreducible conic $C_2$ passing through $q_0,\ldots,q_4$. Moreover, Lemma \ref{lm two_conic_equi2} implies that there exists an automorphism $\alpha$ of $\PP^2$ such that $\alpha(C_1)=C_2$ and $\alpha(p_i)=q_i$, $i=0,1$.
This forces $\alpha(p_i)=q_i$, $i=2,3,4$. Therefore, $\psi_9$ is equivalent to $\varphi_9$.
\end{proof}

\begin{lemma}
Let $\varphi_{10}$ be the map 10 in Table \ref{table1} and let $\psi_{10}$ be a map with enriched weighted proximity graph 10 in Table  \ref{table2}. Then, $\psi_{10}$ is equivalent to $\varphi_{10}$.
\end{lemma}

\begin{proof}
The base points of $\varphi_{10}$ are $p_0=[0:0:1]$ of multiplicity $2$, $p_1=[0:1:0]$ and $p_2,p_3,p_4$ where $p_2\succ_1 p_0$ and $p_4\succ_1p_3\succ_1p_1$ with standard coordinates $p_2=(p_0,0), p_3=(p_1,0)$ and $p_4=(p_1,0,0)$.
The base points of $\psi_{10}$ are $q_0$ of multiplicity 2, $q_1\in\PP^2$ and $q_2,q_3,q_4$ where $q_2\succ_1 q_0$ and $q_4\succ_1q_3\succ_1q_1$ and $q_4$ is aligned with $q_1$ and $q_3$.
Clearly, there exists an automorphism $\alpha_1$ of $\PP^2$ such that $\alpha_1(p_i)=q_i$ for $i=0,1,2,3$.
It follows that also $\alpha_1(p_4)=q_4$, so the maps $\varphi_{10}$ and $\psi_{10}\circ\alpha_1$ are defined by the same homaloidal net, therefore $\varphi_{10}$ and $\psi_{10}$ are equivalent.
\end{proof}

\begin{lemma}
Let $\varphi_{11}$ be the map 11 in Table \ref{table1} and let $\psi_{11}$ be a map with enriched weighted proximity graph 11 in Table  \ref{table2}. Then, $\psi_{11}$ is equivalent to $\varphi_{11}$.
\end{lemma}

\begin{proof}
The base points of $\varphi_{11}$ are $p_0=[0:0:1]$ of multiplicity 2, $p_1=[1:0:0]$ and $p_2, p_3, p_4$ where $p_2\succ_1 p_0$ and $p_4\succ_1 p_3\succ_1 p_1$ with standard coordinates $p_2=(p_0,\infty)$, $p_3=(p_1,0)$, $p_4=(p_1,0,-1)$. So there is a unique irreducible conic passing through $p_0, \ldots, p_4$, that is $C_1 \colon xz+y^2=0$. The base points of $\psi_{11}$ are $q_0$ of multiplicity 2 and $q_1,\ldots,q_4$ where $q_1\in\PP^2$ and $q_2\succ_1 q_0$ and $q_4\succ_1 q_3\succ_1 q_1$. According to Lemma \ref{lm conic3b}, there is a unique irreducible conic $C_2$ passing through $q_0,\ldots,q_4$. Moreover, Lemma \ref{lm two_conic_equi2} implies that there exists an automorphism $\alpha$ of $\PP^2$ such that $\alpha(C_1)=C_2$ and $\alpha(p_i)=q_i$, $i=0,1$.
This forces $\alpha(p_i)=q_i$, $i=2,3,4$. Therefore, $\psi_{11}$ is equivalent to $\varphi_{11}$.
\end{proof}

\begin{lemma}
Let $\varphi_{12}$ be the map 12 in Table \ref{table1} and let $\psi_{12}$ be a map with enriched weighted proximity graph 12 in Table  \ref{table2}. Then, $\psi_{12}$ is equivalent to $\varphi_{12}$.
\end{lemma}

\begin{proof}
The base points of $\varphi_{12}$ are $p_0=[0:1:0]$ of multiplicity $2$, $p_1=[1:0:0]$ and $p_2,p_3,p_4$ where $p_3\succ_1p_1$, $p_4\succ_1p_2\succ_1p_0$ and $p_4\satel p_0$ with standard coordinates $p_2=(p_0,\infty), p_3=(p_1,\infty)$ and $p_4=(p_0,\infty,\infty)$.
The base points of $\psi_{12}$ are $q_0$ of multiplicity 2, $q_1\in\PP^2$ and $q_2,q_3,q_4$ where $q_3\succ_1 q_1$, $q_4\succ_1q_2\succ_1q_0$ and $q_4 \satel q_0$.
Clearly, there exists an automorphism $\alpha_1$ of $\PP^2$ such that $\alpha_1(p_i)=q_i$ for $i=0,1,2,3$.
It follows that also $\alpha_1(p_4)=q_4$, so the maps $\varphi_{12}$ and $\psi_{12}\circ\alpha_1$ are defined by the same homaloidal net, therefore $\varphi_{12}$ and $\psi_{12}$ are equivalent.
\end{proof}

\begin{lemma}
Let $\varphi_{13}$ be the map 13 in Table \ref{table1} and let $\psi_{13}$ be a map with enriched weighted proximity graph 13 in Table  \ref{table2}. Then, $\psi_{13}$ is equivalent to $\varphi_{13}$.
\end{lemma}

\begin{proof}
The base points of $\varphi_{13}$ are $p_0=[0:0:1]$ of multiplicity 2, $p_1=[1:0:0]$ and $p_2, p_3, p_4$ where $p_2\succ_1 p_1$ and $p_4\succ_1 p_3\succ_1 p_0$ with standard coordinates $p_2=(p_1,0)$, $p_3=(p_0,\infty)$, $p_4=(p_0,\infty,-1)$. So there is a unique irreducible conic passing through $p_0, \ldots, p_4$, that is $C_1 \colon xz+y^2=0$. The base points of $\psi_{13}$ are $q_0$ of multiplicity 2 and $q_1,\ldots,q_4$ where $q_1\in\PP^2$ and $q_2\succ_1 q_1$ and $q_4\succ_1 q_3\succ_1 q_0$. According to Lemma \ref{lm conic3b}, there is a unique irreducible conic $C_2$ passing through $q_0,\ldots,q_4$. Moreover, Lemma \ref{lm two_conic_equi2} implies that there exists an automorphism $\alpha$ of $\PP^2$ such that $\alpha(C_1)=C_2$ and $\alpha(p_i)=q_i$, $i=0,1$.
This forces $\alpha(p_i)=q_i$, $i=2,3,4$. Therefore, $\psi_{13}$ is equivalent to $\varphi_{13}$.
\end{proof}

\begin{lemma}
Let $\varphi_{14}$ be the map 14 in Table \ref{table1} and let $\psi_{14}$ be a map with enriched weighted proximity graph 14 in Table  \ref{table2}. Then, $\psi_{14}$ is equivalent to $\varphi_{14}$.
\end{lemma}

\begin{proof}
The base points of $\varphi_{14}$ are $p_0=[0:0:1]$ of multiplicity $2$, $p_1=[0:1:0]$ and $p_2,p_3,p_4$ where $p_2\succ_1p_0$, $p_3\succ_1 p_0$ and $p_4\succ_1 p_1$ with standard coordinates $p_2=(p_0,0), p_3=(p_0,1)$ and $p_4=(p_1,0)$.
The base points of $\psi_{14}$ are $q_0$ of multiplicity 2, $q_1\in\PP^2$ and $q_2,q_3,q_4$ where $q_2\succ_1 q_0$, $q_3\succ_1 q_0$ and $q_4 \succ_1 q_1$.
Clearly, there exists an automorphism $\alpha_1$ of $\PP^2$ such that $\alpha_1(p_i)=q_i$ for $i=0,1,2,4$.

The base points of $\psi_{14}\circ\alpha_1$ are then $p_0, p_1, p_2, q'_3, p_4$ where $q'_3$ has standard coordinates $q'_3=(p_0,u_3)$  for some $u_3 \in \C^*$ because, if $u_3$ were 0, then $q'_3$ would be equal to $p_2$, a contradiction, and, if $u_3$ were $\infty$, then $q'_3$ would be aligned with $p_0$ and $p_1$, again a contradiction.

An automorphism $\alpha_2$ of $\PP^2$ that fixes $p_0,p_1,p_2,p_4$ and that maps $p_3=(p_0,1)$ to $q'_3=(p_0,u_3)$  is 
$$\alpha_2([x:y:z]) = [x: u_3 y: z].$$
Therefore, the maps $\varphi_{14}$ and $\psi_{14}\circ\alpha_1\circ\alpha_2$ are defined by the same homaloidal net and, hence, $\varphi_{14}$ and $\psi_{14}$ are equivalent.
\end{proof}

\begin{lemma}
Let $\varphi_{15}$ be the map 15 in Table \ref{table1} and let $\psi_{15}$ be a map with enriched weighted proximity graph 15 in Table  \ref{table2}. Then, $\psi_{15}$ is equivalent to $\varphi_{15}$.
\end{lemma}

\begin{proof}
The base points of $\varphi_{15}$ are $p_0=[0:0:1]$ of multiplicity $2$, $p_1=[0:1:0]$, $p_2=[1:0:0]$ and $p_3,p_4$ where $p_4\succ_1p_3\succ_1p_0$ and $p_4\satel p_0$ with standard coordinates $p_3=(p_0,1)$ and $p_4=(p_0,1,\infty)$.
The base points of $\psi_{15}$ are $q_0$ of multiplicity 2, $q_1,q_2\in\PP^2$ and $q_3,q_4$ where $q_4\succ_1q_3\succ_1q_0$ and $q_4 \satel q_0$.
Clearly, there exists an automorphism $\alpha_1$ of $\PP^2$ such that $\alpha_1(p_i)=q_i$ for $i=0,1,2,3$.
It follows that also $\alpha_1(p_4)=q_4$, so the maps $\varphi_{15}$ and $\psi_{15}\circ\alpha_1$ are defined by the same homaloidal net, therefore $\varphi_{15}$ and $\psi_{15}$ are equivalent.
\end{proof}

\begin{lemma}
Let $\varphi_{16}$ be the map 16 in Table \ref{table1} and let $\psi_{16}$ be a map with enriched weighted proximity graph 16 in Table  \ref{table2}. Then, $\psi_{16}$ is equivalent to $\varphi_{16}$.
\end{lemma}

\begin{proof}
The base points of $\varphi_{16}$ are $p_0=[0:0:1]$ of multiplicity 2, $p_1=[0:1:0]$, $p_2=[1:1:-1]$ and $p_3, p_4$ where $p_4\succ_1 p_3\succ_1 p_0$ with standard coordinates $p_3=(p_0,0)$, $p_4=(p_0,0,-1)$. So there is a unique irreducible conic passing through $p_0, \ldots, p_4$, that is $C_1 \colon x^2+yz=0$. The base points of $\psi_{16}$ are $q_0$ of multiplicity 2 and $q_1,\ldots,q_4$ where $q_1,q_2\in\PP^2$ and $q_4\succ_1 q_3\succ_1 q_0$. According to Lemma \ref{lm conic2}, there is a unique irreducible conic $C_2$ passing through $q_0,\ldots,q_4$. Moreover, Lemma \ref{lm two_conic_equi2} implies that there exists an automorphism $\alpha$ of $\PP^2$ such that $\alpha(C_1)=C_2$ and $\alpha(p_i)=q_i$, $i=0,1,2$.
This forces $\alpha(p_i)=q_i$, $i=3,4$. Therefore, $\psi_{16}$ is equivalent to $\varphi_{16}$.
\end{proof}

\begin{lemma}
Let $\varphi_{17}$ be the map 17 in Table \ref{table1} and let $\psi_{17}$ be a map with enriched weighted proximity graph 17 in Table  \ref{table2}. Then, $\psi_{17}$ is equivalent to $\varphi_{17}$.
\end{lemma}

\begin{proof}
The base points of $\varphi_{17}$ are $p_0=[0:0:1]$ of multiplicity $2$, $p_1=[1:0:0]$, $p_2=[0:1:0]$ and $p_3,p_4$ where $p_4\succ_1 p_3\succ_1 p_1$ with standard coordinates $p_3=(p_1,0)$ and $p_4=(p_1,0,1)$.
The base points of $\psi_{17}$ are $q_0$ of multiplicity 2, $q_1,q_2\in\PP^2$ and $q_3,q_4$ where $q_4\succ_1 q_3\succ_1 q_1$ and $q_3$ is aligned with $q_1$ and $q_2$.
Clearly, there exists an automorphism $\alpha_1$ of $\PP^2$ such that $\alpha_1(p_i)=q_i$ for $i=0,1,2$.
It follows that also $\alpha_1(p_3)=q_3$.

The base points of $\psi_{17}\circ\alpha_1$ are then $p_0, p_1, p_2, p_3, q'_4$ where $q'_4$ has standard coordinates $q'_4=(p_1,0,u_4)$  for some $u_4 \in \C^*$ because, if $u_4$ were 0, then $q'_4$ would be aligned with $p_1, p_2$ and $p_3$, a contradiction, and, if $u_4$ were $\infty$, then $q'_4$ would be satellite to $p_1$, again a contradiction.

An automorphism $\alpha_2$ of $\PP^2$ that fixes $p_0,p_1,p_2,p_3$ and that maps $p_4=(p_1,0,1)$ to $q'_4=(p_1,0,u_4)$  is 
$$\alpha_2([x:y:z]) = [u_4 x: y: z].$$
Therefore, the maps $\varphi_{17}$ and $\psi_{17}\circ\alpha_1\circ\alpha_2$ are defined by the same homaloidal net and, hence, $\varphi_{17}$ and $\psi_{17}$ are equivalent.
\end{proof}

\begin{lemma}
Let $\varphi_{18}$ be the map 18 in Table \ref{table1} and let $\psi_{18}$ be a map with enriched weighted proximity graph 18 in Table  \ref{table2}. Then, $\psi_{18}$ is equivalent to $\varphi_{18}$.
\end{lemma}

\begin{proof}
The base points of $\varphi_{18}$ are $p_0=[0:0:1]$ of multiplicity $2$, $p_1=[1:0:0]$, $p_2=[0:1:0]$ and $p_3,p_4$ where $p_4\succ_1p_3\succ_1p_1$ with standard coordinates $p_3=(p_1,1)$ and $p_4=(p_1,1,0)$.
The base points of $\psi_{18}$ are $q_0$ of multiplicity 2, $q_1,q_2\in\PP^2$ and $q_3,q_4$ where $q_4\succ_1q_3\succ_1q_1$ and $q_4$ is aligned with $q_1$ and $q_3$.
Clearly, there exists an automorphism $\alpha_1$ of $\PP^2$ such that $\alpha_1(p_i)=q_i$ for $i=0,1,2,3$.
It follows that also $\alpha_1(p_4)=q_4$, so the maps $\varphi_{18}$ and $\psi_{18}\circ\alpha_1$ are defined by the same homaloidal net, therefore $\varphi_{18}$ and $\psi_{18}$ are equivalent.
\end{proof}

\begin{lemma}
Let $\varphi_{19}$ be the map 19 in Table \ref{table1} and let $\psi_{19}$ be a map with enriched weighted proximity graph 19 in Table  \ref{table2}. Then, $\psi_{19}$ is equivalent to $\varphi_{19}$.
\end{lemma}

\begin{proof}
The base points of $\varphi_{19}$ are $p_0=[0:0:1]$ of multiplicity 2, $p_1=[0:1:0]$, $p_2=[1:0:-1]$ and $p_3, p_4$ where $p_4\succ_1 p_3\succ_1 p_1$ with standard coordinates $p_3=(p_1,0)$, $p_4=(p_1,0,-1)$. So there is a unique irreducible conic passing through $p_0, \ldots, p_4$, that is $C_1 \colon x^2+xz+yz=0$. The base points of $\psi_{19}$ are $q_0$ of multiplicity 2 and $q_1,\ldots,q_4$ where $q_1,q_2\in\PP^2$ and $q_4\succ_1 q_3\succ_1 q_1$. According to Lemma \ref{lm conic2}, there is a unique irreducible conic $C_2$ passing through $q_0,\ldots,q_4$. Moreover, Lemma \ref{lm two_conic_equi2} implies that there exists an automorphism $\alpha$ of $\PP^2$ such that $\alpha(C_1)=C_2$ and $\alpha(p_i)=q_i$, $i=0,1,2$.
This forces $\alpha(p_i)=q_i$, $i=3,4$. Therefore, $\psi_{19}$ is equivalent to $\varphi_{19}$.
\end{proof}

\begin{lemma}
Let $\varphi_{20}$ be the map 20 in Table \ref{table1} and let $\psi_{20}$ be a map with enriched weighted proximity graph 20 in Table  \ref{table2}. Then, $\psi_{20}$ is equivalent to $\varphi_{20}$.
\end{lemma}

\begin{proof}
The base points of $\varphi_{20}$ are $p_0=[0:0:1]$ of multiplicity $2$, $p_1=[1:0:0]$, $p_2=[0:1:0]$ and $p_3,p_4$ where $p_3\succ_1 p_1$ and $p_4\succ_1 p_2$ with standard coordinates $p_3=(p_1,0)$ and $p_4=(p_2,1)$.
The base points of $\psi_{20}$ are $q_0$ of multiplicity 2, $q_1,q_2\in\PP^2$ and $q_3,q_4$ where $q_3\succ_1 q_1$ and $q_4\succ_1 q_2$ and $q_3$ is aligned with $q_1$ and $q_2$.
Clearly, there exists an automorphism $\alpha_1$ of $\PP^2$ such that $\alpha_1(p_i)=q_i$ for $i=0,1,2,4$.
It follows that also $\alpha_1(p_3)=q_3$, so the maps $\varphi_{20}$ and $\psi_{20}\circ\alpha_1$ are defined by the same homaloidal net, therefore $\varphi_{20}$ and $\psi_{20}$ are equivalent.
\end{proof}

\begin{lemma}
Let $\varphi_{21}$ be the map 21 in Table \ref{table1} and let $\psi_{21}$ be a map with enriched weighted proximity graph 21 in Table  \ref{table2}. Then, $\psi_{21}$ is equivalent to $\varphi_{21}$.
\end{lemma}

\begin{proof}
The base points of $\varphi_{21}$ are $p_0=[0:0:1]$ of multiplicity 2, $p_1=[1:0:0]$, $p_2=[0:1:0]$ and $p_3, p_4$ where $p_3\succ_1 p_1$ and $p_4\succ_1 p_2$ with standard coordinates $p_3=(p_1,-1)$, $p_4=(p_2,-1)$. So there is a unique irreducible conic passing through $p_0, \ldots, p_4$, that is $C_1 \colon xy+xz+yz=0$. The base points of $\psi_{21}$ are $q_0$ of multiplicity 2 and $q_1,\ldots,q_4$ where $q_1,q_2\in\PP^2$, $q_3\succ_1 q_1$ and $q_4\succ_1 q_2$. According to Lemma \ref{lm conic3}, there is a unique irreducible conic $C_2$ passing through $q_0,\ldots,q_4$. Moreover, Lemma \ref{lm two_conic_equi2} implies that there exists an automorphism $\alpha$ of $\PP^2$ such that $\alpha(C_1)=C_2$ and $\alpha(p_i)=q_i$, $i=0,1,2$.
This forces $\alpha(p_i)=q_i$, $i=3,4$. Therefore, $\psi_{21}$ is equivalent to $\varphi_{21}$.
\end{proof}

\begin{lemma}
Let $\varphi_{22}$ be the map 22 in Table \ref{table1} and let $\psi_{22}$ be a map with enriched weighted proximity graph 22 in Table  \ref{table2}. Then, $\psi_{22}$ is equivalent to $\varphi_{22}$.
\end{lemma}

\begin{proof}
The base points of $\varphi_{22}$ are $p_0=[0:0:1]$ of multiplicity $2$, $p_1=[1:0:0]$, $p_2=[0:1:0]$ and $p_3,p_4$ where $p_3\succ_1 p_0$ and $p_4\succ_1 p_1$ with standard coordinates $p_3=(p_0,-1)$ and $p_4=(p_1,0)$.
The base points of $\psi_{22}$ are $q_0$ of multiplicity 2, $q_1,q_2\in\PP^2$ and $q_3,q_4$ where $q_3\succ_1 q_0$ and $q_4\succ_1 q_1$ and $q_4$ is aligned with $q_1$ and $q_2$.
Clearly, there exists an automorphism $\alpha_1$ of $\PP^2$ such that $\alpha_1(p_i)=q_i$ for $i=0,1,2,3$.
It follows that also $\alpha_1(p_4)=q_4$, so the maps $\varphi_{22}$ and $\psi_{22}\circ\alpha_1$ are defined by the same homaloidal net, therefore $\varphi_{22}$ and $\psi_{22}$ are equivalent.
\end{proof}

\begin{lemma}
Let $\varphi_{23}$ be the map 23 in Table \ref{table1} and let $\psi_{23}$ be a map with enriched weighted proximity graph 23 in Table  \ref{table2}. Then, $\psi_{23}$ is equivalent to $\varphi_{23}$.
\end{lemma}

\begin{proof}
The base points of $\varphi_{23}$ are $p_0=[0:0:1]$ of multiplicity 2, $p_1=[0:1:0]$, $p_2=[1:-1:0]$ and $p_3, p_4$ where $p_3\succ_1 p_0$ and $p_4\succ_1 p_1$ with standard coordinates $p_3=(p_0,0)$, $p_4=(p_1,-1)$. So there is a unique irreducible conic passing through $p_0, \ldots, p_4$, that is $C_1 \colon x^2+xy+yz=0$. The base points of $\psi_{23}$ are $q_0$ of multiplicity 2 and $q_1,\ldots,q_4$ where $q_1,q_2\in\PP^2$, $q_3\succ_1 q_0$ and $q_4\succ_1 q_1$. According to Lemma \ref{lm conic3}, there is a unique irreducible conic $C_2$ passing through $q_0,\ldots,q_4$. Moreover, Lemma \ref{lm two_conic_equi2} implies that there exists an automorphism $\alpha$ of $\PP^2$ such that $\alpha(C_1)=C_2$ and $\alpha(p_i)=q_i$, $i=0,1,2$.
This forces $\alpha(p_i)=q_i$, $i=3,4$. Therefore, $\psi_{23}$ is equivalent to $\varphi_{23}$.
\end{proof}

\begin{lemma}
Let $\varphi_{24}$ be the map 24 in Table \ref{table1} and let $\psi_{24}$ be a map with enriched weighted proximity graph 24 in Table  \ref{table2}. Then, $\psi_{24}$ is equivalent to $\varphi_{24}$.
\end{lemma}

\begin{proof}
The base points of $\varphi_{24}$ are $p_0=[0:0:1]$ of multiplicity $2$, $p_1=[1:0:0]$, $p_2=[0:1:0]$, $p_3=[1:1:0]$ and $p_4$ where $p_4\succ_1 p_1$ with standard coordinates $p_4=(p_1,1)$.
The base points of $\psi_{24}$ are $q_0$ of multiplicity 2, $q_1,q_2,q_3\in\PP^2$ and $q_4$ where $q_4\succ_1 q_1$ and $q_3$ is aligned with $q_1$ and $q_2$.
Clearly, there exists an automorphism $\alpha_1$ of $\PP^2$ such that $\alpha_1(p_i)=q_i$ for $i=0,1,2,3$.

The base points of $\psi_{24}\circ\alpha_1$ are then $p_0, p_1, p_2, p_3, q'_4$ where $q'_4$ has standard coordinates $q'_4=(p_1,u_4)$  for some $u_4 \in \C^*$ because, if $u_4$ were 0, then $q'_4$ would be aligned with $p_1, p_2$ and $p_3$, a contradiction, and, if $u_4$ were $\infty$, then $q'_4$ would be aligned with $p_0$ ad $p_1$, again a contradiction.

An automorphism $\alpha_2$ of $\PP^2$ that fixes $p_0,p_1,p_2,p_3$ and that maps $p_4=(p_1,1)$ to $q'_4=(p_1,u_4)$  is 
$$\alpha_2([x:y:z]) = [x: y: u_4 z].$$
Therefore, the maps $\varphi_{24}$ and $\psi_{24}\circ\alpha_1\circ\alpha_2$ are defined by the same homaloidal net and, hence, $\varphi_{24}$ and $\psi_{24}$ are equivalent.
\end{proof}

\begin{lemma}
Let $\varphi_{25}$ be the map 25 in Table \ref{table1} and let $\psi_{25}$ be a map with enriched weighted proximity graph 25 in Table  \ref{table2}. Then, $\psi_{25}$ is equivalent to $\varphi_{25}$.
\end{lemma}

\begin{proof}
The base points of $\varphi_{25}$ are $p_0=[0:0:1]$ of multiplicity $2$, $p_1=[1:0:0]$, $p_2=[0:1:-1]$, $p_3=[1:-1:0]$ and $p_4$ where $p_4\succ_1 p_1$ with standard coordinates $p_4=(p_1,0)$.
The base points of $\psi_{25}$ are $q_0$ of multiplicity 2, $q_1,q_2,q_3\in\PP^2$ and $q_4$ where $q_4\succ_1 q_1$ and $q_4$ is aligned with $q_1$ and $q_2$.
Clearly, there exists an automorphism $\alpha_1$ of $\PP^2$ such that $\alpha_1(p_i)=q_i$ for $i=0,1,2,3$.
It follows that also $\alpha_1(p_4)=q_4$, so the maps $\varphi_{25}$ and $\psi_{25}\circ\alpha_1$ are defined by the same homaloidal net, therefore $\varphi_{25}$ and $\psi_{25}$ are equivalent.
\end{proof}

\begin{lemma}\label{l:26}
Let $\varphi_{26,\gamma}$ be the map 26 in Table \ref{table1} with parameter $\gamma$ and let $\psi_{26}$ be a map with enriched weighted proximity graph 26 in Table  \ref{table2}. Then, $\psi_{26}$ is equivalent to $\varphi_{26,\gamma}$ for some $\gamma\ne0,1$.
\end{lemma}

\begin{proof}
The base points of $\varphi_{26,\gamma}$ are $p_0=[0:0:1]$ of multiplicity $2$, $p_1=[1:0:0]$, $p_2=[0:1:0]$, $p_3=[1:1:1]$ and $p_4$ where $p_4\succ_1 p_1$ with standard coordinates $p_4=(p_1,1/\gamma)$.
The base points of $\psi_{26}$ are $q_0$ of multiplicity 2, $q_1,q_2,q_3\in\PP^2$ and $q_4$ where $q_4\succ_1 q_1$.
Clearly, there exists an automorphism $\alpha_1$ of $\PP^2$ such that $\alpha_1(p_i)=q_i$ for $i=0,1,2,3$.

The base points of $\psi_{26}\circ\alpha_1$ are then $p_0, p_1, p_2, p_3, q'_4$ where $q'_4$ has standard coordinates $q'_4=(p_1,u_4)$  for some $u_4 \in \C^{**}$ because, if $u_4$ were 0, then $q'_4$ would be aligned with $p_1$ and $p_2$, a contradiction; if $u_4$ were $\infty$, then $q'_4$ would be aligned with $p_0$ ad $p_1$, again a contradiction, and, if $u_4$ were 1, then $q'_4$ would be aligned with $p_1$ and $p_3$, still a contradiction.
Setting $\gamma=1/u_4$, the maps $\varphi_{26,\gamma}$ and $\psi_{26}\circ\alpha_1$ are defined by the same homaloidal net, therefore $\varphi_{26,\gamma}$ and $\psi_{26}$ are equivalent.
\end{proof}

\begin{lemma}\label{l:27}
Let $\varphi_{27,\gamma}$ be the map 27 in Table \ref{table1} with parameter $\gamma$ and let $\psi_{27}$ be a map with enriched weighted proximity graph 27 in Table  \ref{table2}. Then, $\psi_{27}$ is equivalent to $\varphi_{27,\gamma}$ for some $\gamma\ne0,1$.
\end{lemma}

\begin{proof}
The base points of $\varphi_{27,\gamma}$ are $p_0=[0:0:1]$ of multiplicity $2$, $p_1=[0:1:0]$, $p_2=[1:0:0]$ and $p_3,p_4$ where $p_3\succ_1 p_0$ and $p_4\succ_1 p_0$ with standard coordinates $p_3=(p_0,-1)$ and $p_4=(p_0,-1/\gamma)$.
The base points of $\psi_{27}$ are $q_0$ of multiplicity 2, $q_1,q_2\in\PP^2$ and $q_3,q_4$ where $q_3\succ_1 q_0$ and $q_4\succ_1 q_0$.
Clearly, there exists an automorphism $\alpha_1$ of $\PP^2$ such that $\alpha_1(p_i)=q_i$ for $i=0,1,2,3$.

The base points of $\psi_{27}\circ\alpha_1$ are then $p_0, p_1, p_2, p_3, q'_4$ where $q'_4$ has standard coordinates $q'_4=(p_0,u_4)$  for some $u_4 \in \C^{**}$ because, if $u_4$ were 0, then $q'_4$ would be aligned with $p_0$ and $p_2$, a contradiction; if $u_4$ were $\infty$, then $q'_4$ would be aligned with $p_0$ ad $p_1$, again a contradiction, and, if $u_4$ were 1, then $q'_4$ would be equal to $p_3$, still a contradiction.
Setting $\gamma=-1/u_4$, the maps $\varphi_{27,\gamma}$ and $\psi_{27}\circ\alpha_1$ are defined by the same homaloidal net, therefore $\varphi_{27,\gamma}$ and $\psi_{27}$ are equivalent.
\end{proof}

\begin{lemma}\label{l:28}
Let $\varphi_{28,\gamma}$ be the map 28 in Table \ref{table1} with parameter $\gamma$ and let $\psi_{28}$ be a map with enriched weighted proximity graph 28 in Table  \ref{table2}. Then, $\psi_{28}$ is equivalent to $\varphi_{28,\gamma}$ for some $\gamma\ne0,1$.
\end{lemma}

\begin{proof}
The base points of $\varphi_{28,\gamma}$ are $p_0=[0:0:1]$ of multiplicity $2$, $p_1=[0:1:0]$, $p_2=[1:0:0]$, $p_3=[1:1:0]$ and $p_4$ where $p_4\succ_1 p_0$ with standard coordinates $p_4=(p_0,\gamma)$.
The base points of $\psi_{28}$ are $q_0$ of multiplicity 2, $q_1,q_2,q_3\in\PP^2$ and $q_4$ where $q_4\succ_1 q_0$ and $q_1,q_2,q_3$ are collinear.
Clearly, there exists an automorphism $\alpha_1$ of $\PP^2$ such that $\alpha_1(p_i)=q_i$ for $i=0,1,2,3$.

The base points of $\psi_{28}\circ\alpha_1$ are then $p_0, p_1, p_2, p_3, q'_4$ where $q'_4=(p_0,u_4)$ for some $u_4 \in \C^{**}$ because, if $u_4$ were 0, then $q'_4$ would be aligned with $p_0$ and $p_2$, a contradiction; if $u_4$ were $\infty$, then $q'_4$ would be aligned with $p_0$ ad $p_1$, again a contradiction, and, if $u_4$ were 1, then $q'_4$ would be aligned with $p_0$ and $p_3$, still a contradiction.
Setting $\gamma=u_4$, the maps $\varphi_{28,\gamma}$ and $\psi_{28}\circ\alpha_1$ are defined by the same homaloidal net, therefore $\varphi_{28,\gamma}$ and $\psi_{28}$ are equivalent.
\end{proof}

\begin{lemma}\label{l:29}
Let $\varphi_{29,\gamma}$ be the map 29 in Table \ref{table1} with parameter $\gamma$ and let $\psi_{29}$ be a map with enriched weighted proximity graph 29 in Table  \ref{table2}. Then, $\psi_{29}$ is equivalent to $\varphi_{29,\gamma}$ for some $\gamma\ne0,1$.
\end{lemma}

\begin{proof}
The base points of $\varphi_{29,\gamma}$ are $p_0=[0:0:1]$ of multiplicity $2$, $p_1=[0:1:0]$, $p_2=[1:0:0]$, $p_3=[1:1:1]$ and $p_4$ where $p_4\succ_1 p_0$ with standard coordinates $p_4=(p_0,\gamma)$.
The base points of $\psi_{29}$ are $q_0$ of multiplicity 2, $q_1,q_2,q_3\in\PP^2$ and $q_4$ where $q_4\succ_1 q_0$.
Clearly, there exists an automorphism $\alpha_1$ of $\PP^2$ such that $\alpha_1(p_i)=q_i$ for $i=0,1,2,3$.

The base points of $\psi_{29}\circ\alpha_1$ are then $p_0, p_1, p_2, p_3, q'_4$ where $q'_4=(p_0,u_4)$ for some $u_4 \in \C^{**}$ because, if $u_4$ were 0, then $q'_4$ would be aligned with $p_0$ and $p_2$, a contradiction; if $u_4$ were $\infty$, then $q'_4$ would be aligned with $p_0$ ad $p_1$, again a contradiction, and, if $u_4$ were 1, then $q'_4$ would be aligned with $p_0$ and $p_3$, still a contradiction.
Setting $\gamma=u_4$, the maps $\varphi_{29,\gamma}$ and $\psi_{29}\circ\alpha_1$ are defined by the same homaloidal net, therefore $\varphi_{29,\gamma}$ and $\psi_{29}$ are equivalent.
\end{proof}

\begin{lemma}\label{l:30}
Let $\varphi_{30,\gamma}$ be the map 30 in Table \ref{table1} with parameter $\gamma$ and let $\psi_{30}$ be a map with enriched weighted proximity graph 30 in Table  \ref{table2}. Then, $\psi_{30}$ is equivalent to $\varphi_{30,\gamma}$ for some $\gamma\ne0,1$.
\end{lemma}

\begin{proof}
The base points of $\varphi_{30,\gamma}$ are $p_0=[0:0:1]$ of multiplicity $2$, $p_1=[0:1:0]$, $p_2=[1:0:0]$, $p_3=[\gamma:1:0]$ and $p_4=[1:1:1]$.
The base points of $\psi_{30}$ are $q_0$ of multiplicity 2, $q_1,q_2,q_3,q_4\in\PP^2$ where $q_1,q_2,q_3$ are collinear.
Clearly, there exists an automorphism $\alpha_1$ of $\PP^2$ such that $\alpha_1(p_i)=q_i$ for $i=0,1,2,4$.

The base points of $\psi_{30}\circ\alpha_1$ are then $p_0, p_1, p_2, q'_3, p_4$ where $q'_3=[u_3:1:0]$ for some $u_3 \in \C^{**}$ because, if $u_3$ were 0, then $q'_3$ would be equal to $p_1$, a contradiction, and, if $u_3$ were 1, then $q'_3$ would be aligned with $p_0$ and $p_4$, again a contradiction.
Setting $\gamma=u_3$, the maps $\varphi_{30,\gamma}$ and $\psi_{30}\circ\alpha_1$ are defined by the same homaloidal net, therefore $\varphi_{30,\gamma}$ and $\psi_{30}$ are equivalent.
\end{proof}

\begin{lemma}\label{l:31}
Let $\varphi_{31,a,b}$ be the map 31 in Table \ref{table1} with parameters $a,b$ and let $\psi_{31}$ be a map with enriched weighted proximity graph 31 in Table  \ref{table2}. Then, $\psi_{31}$ is equivalent to $\varphi_{31,\gamma}$ for some $a,b\ne0,1$, $a\ne b$.
\end{lemma}

\begin{proof}
The base points of $\varphi_{31,\gamma}$ are $p_0=[0:0:1]$ of multiplicity $2$, $p_1=[0:1:0]$, $p_2=[1:0:0]$, $p_3=[1:1:1]$ and $p_4=[a:b:1]$.
The base points of $\psi_{31}$ are $q_0$ of multiplicity 2 and $q_1,q_2,q_3,q_4\in\PP^2$.
Clearly, there exists an automorphism $\alpha_1$ of $\PP^2$ such that $\alpha_1(p_i)=q_i$ for $i=0,1,2,3$.

The base points of $\psi_{31}\circ\alpha_1$ are then $p_0, p_1, p_2, p_3, q'_4$ where $q'_4=[t_4:u_4:v_4]$ with $t_4,u_4,v_4\in\C^*$: indeed,
\begin{itemize}
\item $v_4\ne0$ because otherwise $q'_4$ would be aligned with $p_1$ and $p_2$;
\item $u_4\ne0$ because otherwise $q'_4$ would be aligned with $p_0$ and $p_1$;
\item $t_4\ne0$ because otherwise $q'_4$ would be aligned with $p_0$ and $p_2$.
\end{itemize}
 Moreover, $t_4/v_4$ and $u_4/v_4$ satisfy the following conditions:
\begin{itemize}
\item $t_4/v_4\ne1$ because otherwise $q'_4$ would be aligned with $p_1$ and $p_3$;
\item $u_4/v_4\ne1$ because otherwise $q'_4$ would be aligned with $p_2$ and $p_3$;
\item $t_4/v_4\ne u_4/v_4$ because otherwise $q'_4$ would be aligned with $p_0$ and $p_3$.
\end{itemize}
Setting $a=t_4/v_4$ and $b=u_4/v_4$, it follows that $a,b\in\C^{**}$ and $a\ne b$, the maps $\varphi_{31,a,b}$ and $\psi_{31}\circ\alpha_1$ are defined by the same homaloidal net, therefore $\varphi_{31,a,b}$ and $\psi_{31}$ are equivalent.
\end{proof}

\begin{lemma}\label{thm : Type26} Set $\varphi_{26,\gamma}$ the map of type $26$ in Table \ref{table1} with parameter $\gamma$ where $\gamma \neq 0,1$. Then, $\varphi_{26,\gamma}$ is equivalent to $\varphi_{26,\gamma'}$ if and only if either $\gamma' = \gamma$ or $\gamma' = {\gamma}/({\gamma -1})$.
\end{lemma}
\begin{proof}
Let $p_0,p_1,\ldots,p_4$ be the base points of $\varphi_{26,\gamma}$ as in the proof of Lemma \ref{l:26}.

An automorphism $\alpha$ of $\PP^2$ that fixes the homaloidal net defining $\varphi_{26,\gamma}$, and that is different from the identity, is such that $\alpha(p_i)=p_i$, $i=0,1$, $\alpha(p_2)=p_3$ and $\alpha(p_3)=p_2$.
Therefore, $\alpha$ is unique and it is defined by
$$\alpha([x:y:z]) = [y-x: y: y-z].$$
so $\alpha(p_4)$ has standard coordinates $(p_1,(\gamma-1)/\gamma)$, hence $\varphi_{26,\gamma/(\gamma-1)}$ is equivalent to $\varphi_{26,\gamma}$.
\end{proof}

\begin{lemma}\label{thm : Type27} Set $\varphi_{27,\gamma}$ the map of type $27$ in Table \ref{table1} with parameter $\gamma$ where $\gamma \neq 0,1$. Then, $\varphi_{27,\gamma}$ is equivalent to $\varphi_{27,\gamma'}$ if and only if either $\gamma' = \gamma$ or $\gamma' = {1}/{\gamma}$.
\end{lemma}
\begin{proof}
Let $p_0,p_1,p_2,p_3,p_4$ be the base points of $\varphi_{27,\gamma}$ as in the proof of Lemma \ref{l:27}.

The base points of $\varphi_{27,\gamma'}$ are $q_i=p_i$, $i=0,1,2,3$, and $q_4=(q_0,-1/\gamma')$.

Suppose that $\varphi_{27,\gamma'}$ is equivalent to $\varphi_{27,\gamma}$.
This implies that there exist automorphisms $\alpha_1,\ldots,\alpha_4$ of $\PP^2$ with the following properties:
\begin{enumerate}[(1)]
\item $\alpha_1$ is such that  $\alpha_1(p_i)=q_i$, $i=0,1,2,3,4$;
\item $\alpha_2$ is such that  $\alpha_2(p_i)=q_i$, $i=0,1,2$, $\alpha_2(p_3)=q_4$ and $\alpha_2(p_4)=q_3$;
\item $\alpha_3$ is such that  $\alpha_3(p_i)=q_i$, $i=0,3,4$, $\alpha_3(p_1)=q_2$ and $\alpha_3(p_2)=q_1$;
\item $\alpha_4$ is such that  $\alpha_4(p_0)=q_0$, $\alpha_4(p_1)=q_2$, $\alpha_4(p_2)=q_1$, $\alpha_4(p_3)=q_4$ and $\alpha_4(p_4)=q_3$.
\end{enumerate}
Then, Case (1) occurs only if $\gamma'=\gamma$ and $\alpha_1$ is the identity.
Case (2) occurs only if $\gamma'=1/\gamma$ and $\alpha_2([x:y:z]) = [x:\gamma y:-\gamma z]$.
Case (3) occurs only if $\gamma'=1/\gamma$ and $\alpha_3([x:y:z]) = [y:x:-z]$.
Case (4) occurs only if $\gamma'=\gamma$ and $\alpha_4([x:y:z]) = [\gamma y:x:z]$.
\end{proof}

Let us now recall some definitions of permutations with cycle notation.

\begin{definition}\label{S3,S4}
Let $\mathfrak{S}_n$ denote the group of permutations of $\{ 1,2,\ldots ,n \}$. Every permutation can be written as a cycle or a product of disjoint cycles.
For $n=3$, the group $\mathfrak{S}_3$ has six elements:
\begin{equation*}\label{S3}
\begin{aligned}
& \mathfrak{s}_1 = \text{id},
&& \mathfrak{s}_2 = (23),
&& \mathfrak{s}_3 = (12),
&& \mathfrak{s}_4 = (123),
&& \mathfrak{s}_5 = (13),
&& \mathfrak{s}_6 = (132).
\end{aligned}
\end{equation*}

For $n=4$, the group $\mathfrak{S}_4$ has 24 elements:
\begin{equation*}\label{S4}
\begin{aligned}
& \mathfrak{s}_1 = \text{id},
&& \mathfrak{s}_2 = (12),
&& \mathfrak{s}_3 = (34),
&& \mathfrak{s}_4 = (12)(34),
&& \mathfrak{s}_5 = (23),
&& \mathfrak{s}_6 = (123), \\
  & \mathfrak{s}_7 = (243),
&& \mathfrak{s}_8 = (1243),
&& \mathfrak{s}_9 = (132),
&& \mathfrak{s}_{10} = (13),
&& \mathfrak{s}_{11} = (1432),
&& \mathfrak{s}_{12} = (143), \\
  & \mathfrak{s}_{13} = (1234),
&& \mathfrak{s}_{14} = (234),
&& \mathfrak{s}_{15} = (124),
&& \mathfrak{s}_{16} = (24),
&& \mathfrak{s}_{17} = (134),
&& \mathfrak{s}_{18} = (1342), \\
  & \mathfrak{s}_{19} = (14),
&& \mathfrak{s}_{20} = (142),
&& \mathfrak{s}_{21} = (13)(24),
&& \mathfrak{s}_{22} = (1324),
&& \mathfrak{s}_{23} = (1423),
&& \mathfrak{s}_{24} = (14)(23).
\end{aligned}
\end{equation*}
\end{definition}

\begin{lemma}\label{thm : Type28-29-30}
For $n \in \lbrace 28,29,30 \rbrace$, set $\varphi_{n,\gamma}$ the map of type $n$ in Table \ref{table1} with parameter $\gamma$ where $\gamma \neq 0,1$. Then, $\varphi_{n,\gamma'}$ is equivalent to $\varphi_{n,\gamma}$ if and only if 
$$\gamma' \in \bigg\lbrace \gamma,\dfrac{1}{\gamma},1-\gamma,\dfrac{1}{1-\gamma},\dfrac{\gamma}{\gamma-1},\dfrac{\gamma-1}{\gamma} \bigg\rbrace.$$
\end{lemma}
\begin{proof}
We first consider the case $n=28$.

The map  $\varphi_{28,\gamma}$ has base points $p_0=[0:0:1]$ of multiplicity 2,  $p_1=[0:1:0]$, $p_2=[1:0:0]$, $p_3=[1:1:0]$ and $p_4$ where $p_4\succ_{1}p_0$ with standard coordinates $p_4=(p_0,\gamma)$.

The base points of $\varphi_{28,\gamma'}$ are $q_0,\ldots,q_4$ where $q_i=p_i$, $i=0,1,2,3$ and $q_4=(p_0,\gamma')$.

Suppose that $\varphi_{28,\gamma'}$ is equivalent to $\varphi_{28,\gamma}$.
This implies that there exist automorphisms $\alpha_1,\ldots,\alpha_6$ of $\PP^2$ such that, for $i=1,\ldots,6$, one has $\alpha_i(p_j)=q_j$, $j=0,4$, and
\[
\alpha_i(p_j)=q_{\mathfrak{s}_i(j)} \quad \text{ for } j = 1,2,3,
\]
where $\mathfrak{s}_1,\ldots,\mathfrak{s}_6$ are the six elements of $\mathfrak{S}_3$ given in Definition \ref{S3,S4}.

\begin{itemize}
\item Case $i=1$ occurs only if $\gamma'=\gamma$ and $\alpha_1$ is the identity.
\item Case $i=2$ occurs only if $\gamma'=1-\gamma$ and $\alpha_2=[x: x-y: z]$.
\item Case $i=3$ occurs only if $\gamma'=1/\gamma$ and $\alpha_3=[y: x: z]$.
\item Case $i=4$ occurs only if $\gamma'=1/(1-\gamma)$ and $\alpha_4=[x-y: x: z]$.
\item Case $i=5$ occurs only if $\gamma'=\gamma/(\gamma-1)$ and $\alpha_5=[x-y: -y: z]$.
\item Case $i=6$ occurs only if $\gamma'=\gamma/(\gamma-1)$ and $\alpha_6=[y: y-x: z]$.
\end{itemize}

We proceed similarly for $n=29$.
The map  $\varphi_{29,\gamma}$ has the same base points $p_i$, $i=0,1,2,4$, of $\varphi_{28,\gamma}$ but $p_3=[1:1:1]$.
The base points of $\varphi_{29,\gamma'}$ are $q_0,\ldots,q_4$ where $q_i=p_i$, $i=0,1,2,3$ and $q_4=(q_0,\gamma')$.

If $\varphi_{28,\gamma'}$ is equivalent to $\varphi_{28,\gamma}$, then there exist automorphisms $\alpha_1,\ldots,\alpha_6$ of $\PP^2$ with the same above properties that occur exactly when $\gamma'$ is as above and $\alpha_1$ is the identity,
\begin{align*}
\alpha_2 &=[x: x-y: x-z],
&\alpha_3&=[y: x: z],
&\alpha_4&=[x-y: x: x-z],
\\
\alpha_5&=[y-x: y: y-z],
&\alpha_6&=[y: y-x: y-z].
\end{align*}

Finally, for $n=30$, the map  $\varphi_{30,\gamma}$ has the same base points $p_i$, $i=0,1,2$, of $\varphi_{28,\gamma}$ but $p_3=[\gamma:1:0]$ and $p_4=[1:1:1]$. The base points of $\varphi_{30,\gamma'}$ are $q_0,\ldots,q_4$ where $q_i=p_i$, $i=0,1,2,4$ and $q_3=[\gamma':1:0]$.

If $\varphi_{30,\gamma'}$ is equivalent to $\varphi_{30,\gamma}$, then there exist automorphisms $\alpha_1,\ldots,\alpha_6$ of $\PP^2$ with the same above properties that occur exactly when $\gamma'$ is as above and $\alpha_1$ is the identity,
\begin{align*}
\alpha_2 &=[(\gamma -1)x:\gamma y-x:(\gamma -1)z],
&\alpha_3&=[y: x: z],
\\
\alpha_4&=[\gamma y-x:(\gamma -1)x:(\gamma -1)z],
&\alpha_5&=[\gamma y-x:(\gamma -1)y:(\gamma -1)z],
\\
\alpha_6&=[(\gamma -1)y:\gamma y-x:(\gamma -1)z].
\end{align*}
\end{proof}

\begin{remark}
One may check that the numbers in the set of Lemma \ref{thm : Type28-29-30} are all different if and only if
$$\gamma\notin\left\{ -1,2,\dfrac{1}{2},\dfrac{1}{2}-i\dfrac{\sqrt{3}}{2},\dfrac{1}{2}+i\dfrac{\sqrt{3}}{2} \right\}.$$
\end{remark}

\begin{lemma}\label{thm : Type31}
Set $\varphi_{31,a,b}$ the map of type $31$ in Table \ref{table1} with two parameters $a, b$ where $a \neq b$ and $a,b \neq 0,1$. Then, $\varphi_{31,a',b'}$ is equivalent to $\varphi_{31,a,b}$ if and only if $(a',b') \in S$, where $S$ is defined in \eqref{set S}.
\end{lemma}
\begin{proof}
The base points of $\varphi_{31,a,b}$ are $p_0=[0:0:1]$ of multiplicity $2$ and four simple base points $p_1=[0:1:0],p_2=[1:0:0],p_3=[1:1:1],p_4=[a:b:1]$. Similarly, the base points of $\varphi_{31,a',b'}$ are $q_0,\ldots,q_4$ where $q_i=p_i$, $i=0,1,2,3$ and $q_4=[a':b':1]$.

Suppose that $\varphi_{31,a',b'}$ is equivalent to $\varphi_{31,a,b}$.
Then, there exists an automorphism, says $\gamma$, of $\mathbb{P}^2$ such that $\gamma(p_0)=q_0$ and $\gamma$ maps $p_1,\ldots,p_4$ to a permutation of $q_1,q_2,q_3,q_4$.
Therefore, for each element $\mathfrak{s}_i$, $i=1,\ldots,24$, of $\mathfrak{S}_4$ there is an automorphism $\gamma_i$, $i=1,\ldots,24$, of $\PP^2$ such that
$$\gamma_i(p_j)=q_{\mathfrak{s}_i(j)} \quad \text{ for } j = 1, \ldots, 4,$$
and, accordingly, we find the values of $(a',b')$ for each one of the 24 cases.
In Table \ref{table 24aut}, we list the automorphisms $\gamma_i$, $i=1,\ldots,24$ and their corresponding values of $(a',b')$.

\begin{longtable}[c]{|c||c|c|}
\hline
$i$ & $\gamma_i([x:y:z])$ & $(a',b')$  \\ \hline \hline
\rule{0ex}{2.0ex} 1   & $[x: y: z]$  & $(a,b)$   \\[0.2ex] \hline
\rule{0ex}{2.0ex} 2   & $[y: x: z]$  & $(b,a)$   \\[0.2ex] \hline
\rule{0ex}{4.0ex} 3   & $[bx: ay: abz]$   & $\bigg(\dfrac{1}{a},\dfrac{1}{b}\bigg)$      \\[1.8ex] \hline
\rule{0ex}{4.0ex} 4   & $[ay: bx: abz]$   & $\bigg(\dfrac{1}{b},\dfrac{1}{a}\bigg)$       \\[1.8ex] \hline
\rule{0ex}{4.0ex} 5   & $[x: x-y: x-z]$   & $\bigg(\dfrac{a}{a-1},\dfrac{a-b}{a-1}\bigg)$      \\[1.8ex] \hline
\rule{0ex}{4.0ex} 6   & $[x-y: x: x-z]$   & $\bigg(\dfrac{a-b}{a-1},\dfrac{a}{a-1}\bigg)$      \\[1.8ex] \hline
\rule{0ex}{4.0ex} 7   & $\bigg[\dfrac{x}{a}: \dfrac{x-y}{a-b}: \dfrac{x-z}{a-1}\bigg]$ & $\bigg(\dfrac{a-1}{a},\dfrac{a-1}{a-b}\bigg)$  \\[1.8ex] \hline
\rule{0ex}{4.0ex} 8  & $\bigg[\dfrac{x-y}{a-b}: \dfrac{x}{a}: \dfrac{x-z}{a-1}\bigg]$ & $\bigg(\dfrac{a-1}{a-b},\dfrac{a-1}{a}\bigg)$  \\[1.8ex] \hline
\rule{0ex}{4.0ex} 9   & $[y: y-x: y-z]$  & $\bigg(\dfrac{b}{b-1},\dfrac{b-a}{b-1}\bigg)$      \\[1.8ex] \hline
\rule{0ex}{4.0ex} 10  & $[y-x: y: y-z]$  & $\bigg(\dfrac{b-a}{b-1},\dfrac{b}{b-1}\bigg)$       \\[1.8ex] \hline
\rule{0ex}{4.0ex} 11  & $\bigg[\dfrac{y}{b}:\dfrac{x-y}{a-b}:\dfrac{y-z}{b-1}\bigg]$ & $\bigg(\dfrac{b-1}{b},\dfrac{b-1}{b-a}\bigg)$ \\[1.8ex] \hline
\rule{0ex}{4.0ex} 12  & $\bigg[\dfrac{x-y}{a-b}:\dfrac{y}{b}:\dfrac{y-z}{b-1}\bigg]$ & $\bigg(\dfrac{b-1}{b-a},\dfrac{b-1}{b}\bigg)$ \\[1.8ex] \hline
\rule{0ex}{4.0ex} 13  & $[bx-ay: bx: b(x-az)]$  & $\bigg(\dfrac{b-a}{b(1-a)},\dfrac{1}{1-a}\bigg)$ \\[1.8ex] \hline
\rule{0ex}{4.0ex} 14   & $[bx:bx-ay:b(x-az)]$   & $\bigg(\dfrac{1}{1-a},\dfrac{b-a}{b(1-a)}\bigg)$ \\[1.8ex] \hline
\rule{0ex}{4.0ex} 15  & $\bigg[\dfrac{ay-bx}{a-b}:x:\dfrac{az-x}{a-1}  \bigg]$ & $ \bigg(\dfrac{b(a-1)}{a-b},1-a\bigg)$ \\[1.8ex] \hline
\rule{0ex}{4.0ex} 16  & $\bigg[x:\dfrac{ay-bx}{a-b}: \dfrac{az-x}{a-1}\bigg]$ & $ \bigg(1-a,\dfrac{b(a-1)}{a-b}\bigg)$\\ [1.8ex] \hline
\rule{0ex}{4.0ex} 17  & $[ay-bx: ay: a(y-bz)]$ & $ \bigg(\dfrac{a-b}{a(1-b)},\dfrac{1}{1-b}\bigg)$ \\[1.8ex] \hline
\rule{0ex}{4.0ex} 18  & $[ay: ay-bx: a(y-bz)]$ & $ \bigg(\dfrac{1}{1-b},\dfrac{a-b}{a(1-b)}\bigg)$ \\[1.8ex] \hline
\rule{0ex}{4.0ex} 19  & $\bigg[\dfrac{ay-bx}{a-b}:y:\dfrac{bz-y}{b-1} \bigg]$ & $\bigg(\dfrac{a(1-b)}{a-b},1-b\bigg)$ \\[1.8ex] \hline
\rule{0ex}{4.0ex} 20  & $\bigg[y:\dfrac{ay-bx}{a-b}:\dfrac{bz-y}{b-1} \bigg]$ & $\bigg(1-b,\dfrac{a(1-b)}{a-b}\bigg)$ \\[1.8ex] \hline
\rule{0ex}{4.0ex} 21  & $\bigg[y-x:\dfrac{ay-bx}{a}:\dfrac{(1-b)x}{a-1}+\dfrac{(b-a)z}{a-1}+y\bigg]$ & $\bigg(\dfrac{a-1}{b-1},\dfrac{b(a-1)}{a(b-1)}\bigg)$ \\[1.8ex] \hline
\rule{0ex}{4.0ex} 22  & $\bigg[\dfrac{ay-bx}{a}:y-x:\dfrac{(1-b)x}{a-1}+\dfrac{(b-a)z}{a-1}+y\bigg]$ & $\bigg(\dfrac{b(a-1)}{a(b-1)},\dfrac{a-1}{b-1}\bigg)$ \\[1.8ex] \hline
\rule{0ex}{4.0ex} 23  & $\bigg[y-x:\dfrac{ay-bx}{b}:\dfrac{(a-1)y}{b-1}+\dfrac{(b-a)z}{b-1}-x\bigg]$ & $\bigg(\dfrac{b-1}{a-1},\dfrac{a(b-1)}{b(a-1)}\bigg)$ \\[1.8ex] \hline 
\rule{0ex}{4.0ex} 24  & $\bigg[\dfrac{ay-bx}{b}:y-x:\dfrac{(a-1)y}{b-1}+\dfrac{(b-a)z}{b-1}-x\bigg]$ & $\bigg(\dfrac{a(b-1)}{b(a-1)},\dfrac{b-1}{a-1}\bigg)$ \\[1.8ex] \hline 
\addlinespace 
\caption{Automorphisms $\gamma_1,\ldots,\gamma_{24}$ of $\PP^2$ and their corresponding values of $(a',b')$}\label{table 24aut}
\end{longtable}



\end{proof}

\begin{remark}
One may check that the pairs in $S$ are all different if and only if $(a,b)$ does not belong to the following set:
\begin{eqnarray*}
\bigg\lbrace \bigg(a,\dfrac{1}{a} \bigg) \bigg\vert a \neq -1 \bigg\rbrace \cup \bigg\lbrace \bigg(a,\dfrac{2a-1}{a} \bigg) \bigg\vert a \neq \dfrac{1}{2} \bigg\rbrace \cup \bigg\lbrace \bigg(\dfrac{2b-1}{b},b \bigg) \bigg\vert b \neq \dfrac{1}{2} \bigg\rbrace
\\
\cup \bigg\lbrace
(a,b) \bigg\vert a = \dfrac{3}{2}\pm\dfrac{i\sqrt{3}}{6}, b =  -\dfrac{1}{2}\pm\dfrac{i\sqrt{3}}{6} \bigg\rbrace \cup \bigg\lbrace
(a,\overline{a}) \bigg\vert a \in \bigg\lbrace \dfrac{1}{2}\pm\dfrac{i\sqrt{3}}{6}, \dfrac{3}{2}\pm\dfrac{i\sqrt{3}}{2}  \bigg\rbrace\bigg\rbrace
\\
\cup \bigg\lbrace (a,-\overline{a}),\big(a,\overline{-\overline{a}}\big) \bigg\vert a \in \bigg\lbrace -\dfrac{1}{2}\pm\dfrac{i\sqrt{3}}{2}, \dfrac{1}{2}\pm\dfrac{i\sqrt{3}}{2}, -\dfrac{1}{2}\pm\dfrac{i\sqrt{3}}{6}   \bigg\rbrace\bigg\rbrace.
\end{eqnarray*}
\end{remark}

\section{Ordinary quadratic length of cubic plane Cremona maps} \label{proof main theorem2}

In this section we prove Theorem \ref{thm main2}.
Theorem \ref{thm main1} implies that it suffices to compute the lengths of the cubic plane Cremona maps listed in Table \ref{table1} at page \pageref{table1}.

Recall that the quadratic length, and hence the ordinary quadratic length, of cubic plane Cremona maps is at least 2 (Corollary \ref{cor>1}).
On the other hand, in Table \ref{table3} at page \pageref{table3} and Table \ref{table4} at page \pageref{table4} there are decompositions of all types of plane cubic maps, but type 1, in exactly two quadratic maps.
So, in order to complete the proof of the first assertion of  Theorem \ref{thm main2}, it remains to prove the following lemma.

\begin{lemma}
Let $\varphi_1 \in \Cr(\PP^2)$ be the map 1 in Table \ref{table1}. Then, $\varphi_1$ has quadratic length $3$. 
\end{lemma}
\begin{proof}
Let $p_1$ be the double base point of $\varphi_1$ and let $p_2,\ldots,p_5$ be its simple base points, that are all infinitely near $p_1$, namely $p_5\infnear_1 p_4 \infnear_1 p_3 \infnear_1 p_2 \infnear_1 p_1$ where $p_3\satel p_1$.
Hence, a quadratic map can be based at $p_1$ and at $p_2$, but not at $p_3$, cf.\ Remark \ref{lm no_pro_graph}.

The decomposition in Table \ref{table4} at page \pageref{table4} implies that $q(\varphi_1)\leq3$.
By contradiction, suppose that $\q(\varphi_1) =2$. Then, there should exist a quadratic map $\rho$ such that $\q(\varphi_1\circ\rho^{-1}) =1$, so $\varphi_1\circ\rho^{-1}$ should be a quadratic map by Lemma \ref{lm q_0_1}. However,
\begin{itemize}
\item if $\rho$ is not based at $p_1$, then $\varphi_1\circ\rho^{-1}$ has degree $6$, a contradiction;
\item if $\rho$ is based at $p_1$, but not at $p_2$, then $\varphi_1\circ\rho^{-1}$ has degree $4$, again a contradiction;
\item finally, if $\rho$ is based at $p_1$ and $p_2$, then $\varphi_1\circ\rho^{-1}$ has degree $3$, a contradiction.
\end{itemize}
Hence, we conclude that $\q(\varphi_1) =3$.
\end{proof}

We now prove the second assertion of Theorem \ref{thm main2}, that is that the cubic plane Cremona map of type $n$, $1\leq n \leq 31$,  in Table \ref{table1} at page \pageref{table1} has the respective ordinary quadratic length listed in the third column in Table \ref{table2} at page \pageref{table2}.

The decompositions in Table \ref{table3} at page \pageref{table3} show that the maps of types 21, 23, 25, 26, 27, 29, 30, 31 have the ordinary quadratic length exactly 2.

Recall that Proposition \ref{oq>height} says the ordinary quadratic length of a plane Cremona map is at least the maximum height of its base points.
In particular, the maps $\varphi_n$, $n=10$, 11, 12, 13, 15, 16, 18, 19, have $\oq(\varphi_n)\geq3$ and the decompositions in Table \ref{table3} at page \pageref{table3} show that indeed $\oq(\varphi_n)=3$.
Similarly, the maps $\varphi_n$, $n=2$, 7, 9, have $\oq(\varphi_n)\geq4$ and the decompositions in Table \ref{table3} show that $\oq(\varphi_n)=4$.

We now consider the maps of the remaining types, going backwards from the last types to the first ones.

\begin{lemma}\label{oq sharp_28}
Let $\varphi_{28}$ be the map 28 in Table \ref{table1}. Then, $\oq(\varphi_{28}) = 3.$
\end{lemma}
\begin{proof}
Let $p_1$ be the double base point of $\varphi_{28}$ and $p_3, p_4, p_5$ the proper simple base points of $\varphi_{28}$, which are collinear.
The decomposition of $\varphi_{28}$ in Table \ref{table3} shows that $\oq(\varphi_{28}) \leq 3$.
Suppose by contradiction that $\oq(\varphi_{28}) =2$. Therefore, there should exist an ordinary quadratic map $\rho$ such that $\oq(\varphi_{28}\circ\rho^{-1}) =1$, i.e.\ the map $\varphi_{28}\circ\rho^{-1}$ should be an ordinary quadratic map. Since $\varphi_{28}\circ\rho^{-1}$ should have degree 2, the map $\rho$ must be based at $p_1$ and two proper simple base points of $\varphi_{28}$, say $p_3, p_4$. However, in that case, the quadratic map $\varphi_{28}\circ\rho^{-1}$ is not ordinary, because $p_5$ would correspond to an infinitely near base point of $\varphi_{28}\circ\rho^{-1}$, a contradiction. 
\end{proof}

\begin{remark}\label{rem:20,22,24}
The same argument used in the proof of Lemma \ref{oq sharp_28} shows that the maps 20, 22, 24 in Table \ref{table1} have ordinary quadratic length exactly $3$.
\end{remark}

\begin{lemma}\label{oq sharp_17}
Let $\varphi_{17}$ be the map 17 in Table \ref{table1}. Then, $\oq(\varphi_{17}) = 4.$
\end{lemma}
\begin{proof}
The enriched weighted proximity graph of $\varphi_{17}$ is listed in Table \ref{table2} at page \pageref{table2}.
Let $p_1$ be the double base point, $p_2, p_3$ the two proper simple base points and $p_4, p_5$ such that $p_5\infnear_1 p_4 \infnear_1 p_3$ where $p_2, p_3, p_4$ are aligned.
Then,
$$3 \leqslant \oq(\varphi_{17}) \leqslant 4$$
because of the decomposition of $\varphi_{17}$ in Table \ref{table3} and the fact that the height of $p_5$ with respect to $\varphi_{17}$ is 3, cf.\ Proposition  \ref{oq>height}.

Suppose by contradiction that $\oq(\varphi_{17}) =3$. Then, there should exist an ordinary quadratic map $\rho$ such that $\oq(\varphi_{17}\circ\rho^{-1}) = 2$.
In particular, $\rho$ must be based at $p_3$, otherwise, the maximum height of the base points of the map $\varphi_{17}\circ\rho^{-1}$ would be still $3$ and Proposition \ref{oq>height} would give a contradiction.

If $\rho$ is based also at $p_2$ (or at another point on the line passing through $p_3$ and $p_2$), then $p_4$ would correspond to an infinitely near base point of $\varphi_{17}\circ\rho^{-1}$ and the maximum height of the base points of $\varphi_{17}\circ\rho^{-1}$ would be again $3$, a contradiction.

There are now two cases: either $p_1$ is a base point of $\rho$ or $p_1$ is not a base point of $\rho$.

In the former case, the map $\varphi_{17}\circ\rho^{-1}$ would have the enriched weighted proximity graph 24 in Table \ref{table2}, and therefore would have ordinary quadratic length 3, as we noted in Remark \ref{rem:20,22,24}, a contradiction.

In the latter case, the map $\varphi_{17}\circ\rho^{-1}$ would have degree $5$, and therefore its ordinary quadratic length cannot be $2$ by Corollary \ref{cor>2}, a contradiction.

Hence, we conclude that $\oq(\varphi_{17}) =4$.
\end{proof}

\begin{lemma}\label{oq sharp_14}
Let $\varphi_{14}$ be the map 14 in Table \ref{table1}. Then, $\oq(\varphi_{14}) = 3.$
\end{lemma}
\begin{proof}
The decomposition of $\varphi_{14}$ in Table \ref{table3} shows that $\oq(\varphi_{14}) \leq 3$.
Suppose by contradiction that $\oq(\varphi_{14}) =2$. Therefore, there should exist an ordinary quadratic map $\rho$ such that $\oq(\varphi_{14}\circ\rho^{-1}) =1$, i.e.\ the map $\varphi_{14}\circ\rho^{-1}$ should be an ordinary quadratic map. In other words, $\rho$ should be based at the double base point of $\varphi_{14}$ and other two proper simple base points of $\varphi_{14}$, theat however do not exist.
\end{proof}

\begin{lemma}\label{oq sharp_8}
Let $\varphi_8$ be the map 8 in Table \ref{table1}. Then, $\oq(\varphi_8) = 5.$
\end{lemma}

\begin{proof}
The enriched weighted proximity graph of $\varphi_{8}$ is listed in Table \ref{table2}.
Let $p_1$ be the double base point, $p_2$ the proper simple base point and $p_3, p_4, p_5$ the other infinitely near base points such that $p_5\infnear_1 p_4 \infnear_1 p_3 \infnear_1 p_2$ where $p_2, p_3, p_4$ are aligned.
Then,
$$4 \leqslant \oq(\varphi_8) \leqslant 5$$
because of the decomposition of $\varphi_{8}$ in Table \ref{table3} and the fact that the height of $p_5$ with respect to $\varphi_{8}$ is 4, cf.\ Proposition  \ref{oq>height}.

Suppose by contradiction that $\oq(\varphi_8) =4$. Then, there should exist an ordinary quadratic map $\rho_1$ such that $\oq(\varphi_8\circ\rho_1^{-1}) = 3$. In particular, $\rho_1$ must be based at $p_2$, otherwise, the maximum height of the base points of the map $\varphi_{8}\circ\rho_1^{-1}$ would be still $4$ and Proposition \ref{oq>height} would give a contradiction.
For the same reason, $\rho_1$ cannot be based at $p_2$ and also at a point on the line passing through $p_2$ and $p_3$.

There are now two cases: either $p_1$ is a base point of $\rho_1$ or $p_1$ is not a base point of $\rho_1$.

In the former case, the map $\varphi_{8}\circ\rho^{-1}$ would have the enriched weighted proximity graph 17 in Table \ref{table2}, and therefore it would have ordinary quadratic length 4, as we proved in Lemma \ref{oq sharp_17}, a contradiction.

In the latter case, the map $\varphi_{8}\circ\rho^{-1}$ would have degree 5 and the following weighted proximity graph:

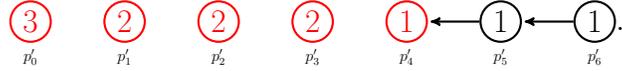
\begin{center}
\scalebox{0.5}{%
\begin{tikzpicture}[->,>=stealth',shorten >=2pt,auto,node distance=2.5cm,ultra thick,baseline=-1.25ex]
  \tikzstyle{every state}=[fill=white,draw=black,text=black]

  \node[state,draw=red,text=red,label=below: $p'_0$] (A)                    {\huge 3};
  \node[state,draw=red,text=red,label=below: $p'_1$] (B) [right of=A] {\huge 2};
  \node[state,draw=red,text=red,label=below: $p'_2$] (C) [right of=B] {\huge 2};
  \node[state,draw=red,text=red,label=below: $p'_3$] (D) [right of=C] {\huge 2};
  \node[state,draw=red,text=red,label=below: $p'_4$] (E) [right of=D] {\huge 1};
  \node[state,label=below: $p'_5$] (F) [right of=E] {\huge 1};
  \node[state,label=below: $p'_6$] (G) [right of=F] {\huge 1};

\path (G) edge (F);
\path (F) edge (E);
\end{tikzpicture}%
}.
\end{center}
where $p'_0, p'_4, p'_5$ are aligned.
Furthermore, there should exist an ordinary quadratic map $\rho_2$ such that $\oq(\varphi_8\circ\rho_1^{-1}\circ\rho_2^{-1}) =2$.
In particular, $\rho_2$ must be based at $p'_4$, otherwise the maximum height of the base points of the map $\varphi_8\circ\rho_1^{-1}\circ\rho_2^{-1}$ would be still $3$ and Proposition \ref{oq>height} would give a contradiction.
For the same reason, $\rho_2$ cannot be based at $p'_4$ and also at $p'_0$ or at another point on the line passing through $p'_4$ and $p'_5$.
Therefore, $\rho_2$ is based at $p'_4$ and other two points where $\varphi_{8}\circ\rho^{-1}$ has multiplicity $\leq 2$, hence the map $\varphi_8\circ\rho_1^{-1}\circ\rho_2^{-1}$ would have degree $\geq5$ and we get a contraction with Corollary \ref{cor>2}.

We conclude that $\oq(\varphi_8) = 5$.
\end{proof}

\begin{lemma}\label{oq sharp_6}
Let $\varphi_6$ be the map 6 in Table \ref{table1}. Then, $\oq(\varphi_6) = 4.$
\end{lemma}

\begin{proof}
The enriched weighted proximity graph of $\varphi_{6}$ is listed in Table \ref{table2}.
Let $p_1$ be the double base point, $p_5$ the proper simple base point and $p_2, p_3, p_4$ the other infinitely near base points such that $p_2\infnear_1 p_1$ and $p_4 \infnear_1 p_3 \infnear_1 p_1$.
Then,
$$3 \leqslant \oq(\varphi_6) \leqslant 4$$
because of the decomposition of $\varphi_{6}$ in Table \ref{table3} and the fact that the height of $p_4$ with respect to $\varphi_{6}$ is 3, cf.\ Proposition  \ref{oq>height}.

Suppose by contradiction that $\oq(\varphi_6) =3$. Then, there should exist an ordinary quadratic map $\rho$ such that $\oq(\varphi_6\circ\rho^{-1}) = 2$. In particular, $\rho$ must be based at $p_1$, otherwise the maximum height of the base points of the map $\varphi_{6}\circ\rho^{-1}$ would be still $3$ and Proposition \ref{oq>height} would give a contradiction.
For the same reason, $\rho_1$ cannot be based at $p_1$ and also at a point on the line passing through $p_1$ and $p_3$.

There are now two cases: either $\rho$ is based at $p_5$ or $\rho$ is not based at $p_5$.

In the former case, the map $\varphi_{6}\circ\rho^{-1}$ would have the enriched weighted proximity graph 24 in Table \ref{table2}, and therefore it would have ordinary quadratic length 3 (cf.\ Remark \ref{rem:20,22,24}), a contradiction.

In the latter case, the map $\varphi_{6}\circ\rho^{-1}$ would be a de Jonqui\`eres map of degree 4, a contradiction with Lemma \ref{cordeJonq>d-1}.

Therefore, we conclude that  $\oq(\varphi_6) =4$.
\end{proof}

\begin{lemma}\label{oq sharp_5}
Let $\varphi_5$ be the map 5 in Table \ref{table1}. Then, $\oq(\varphi_5) =5.$
\end{lemma}

\begin{proof}
The enriched weighted proximity graph of $\varphi_{5}$ is listed in Table \ref{table2}.
Let $p_1$ be the double base point, $p_5$ the proper simple base point and $p_2, p_3, p_4$ the other infinitely near base points such that $p_4 \infnear_1 p_3 \infnear_1 p_2 \infnear_1 p_1$ with $p_3\satel p_1$.
Then,
$$4 \leqslant \oq(\varphi_5) \leqslant 5$$
because of the decomposition of $\varphi_{5}$ in Table \ref{table3} and the fact that the height of $p_4$ with respect to $\varphi_{5}$ is 4, cf.\ Proposition  \ref{oq>height}.

Suppose by contradiction that $\oq(\varphi_5) =4$. Then, there should exist an ordinary quadratic map $\rho_1$ such that $\oq(\varphi\circ\rho^{-1}_1) = 3$. This implies that $\rho_1$ must be based at $p_1$, otherwise the maximum height of the base points of the map $\varphi_{5}\circ\rho_1^{-1}$ would be still $4$ and Proposition \ref{oq>height} would give a contradiction.
For the same reason, $\rho_1$ cannot be based at $p_1$ and at a point on the line passing through $p_1$ and $p_2$.

There are now two cases: either $p_5$ is a base point of $\rho_1$ or $p_5$ is not a base point of $\rho_1$.

In the former case, the map $\varphi_{5}\circ\rho_1^{-1}$ would have enriched weighted proximity graph of type 17 in Table \ref{table2} and, therefore, it would have ordinary quadratic length 4 by Lemma \ref{oq sharp_17}, a contradiction.

In the latter case, the map $\varphi_5\circ\rho^{-1}_1$ would be a de Jonqui\`eres map of degree $4$ and its weighted proximity graph would be

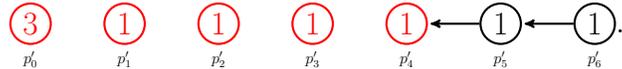
\begin{center}
\scalebox{0.5}{%
\begin{tikzpicture}[->,>=stealth',shorten >=2pt,auto,node distance=2.5cm,ultra thick,baseline=-1.25ex]
  \tikzstyle{every state}=[fill=white,draw=black,text=black]

  \node[state,draw=red,text=red,label=below: $p'_0$] (A)              {\huge 3};
  \node[state,draw=red,text=red,label=below: $p'_1$] (B) [right of=A] {\huge 1};
  \node[state,draw=red,text=red,label=below: $p'_2$] (C) [right of=B] {\huge 1};
  \node[state,draw=red,text=red,label=below: $p'_3$] (D) [right of=C] {\huge 1};
  \node[state,draw=red,text=red,label=below: $p'_4$] (E) [right of=D] {\huge 1};
  \node[state,label=below: $p'_5$] (F) [right of=E] {\huge 1};
  \node[state,label=below: $p'_6$] (G) [right of=F] {\huge 1};

\path (F) edge (E);
\path (G) edge (F);
\end{tikzpicture}%
}.
\end{center}
where $p'_2, p'_3, p'_4, p'_5$ are aligned.

Then, there should exist an ordinary quadratic map $\rho_2$ such that $\oq(\varphi_5\circ\rho_1^{-1}\circ\rho_2^{-1}) =2$.
The map $\rho_2$ must be based at $p'_4$, and not at $p'_2,p'_3$, otherwise the maximum height of the base points of the map $\varphi_5\circ\rho_1^{-1}\circ\rho_2^{-1}$ would be at least $3$, a contradiction with Proposition \ref{oq>height}.
If $\rho_2$ is not based at $p'_0$, then $\deg(\varphi_5\circ\rho_1^{-1}\circ\rho_2^{-1})\geq6$ and we get a contradiction with Corollary \ref{cor>2}. Otherwise $\rho_2$ is based at $p'_0$ and, furthermore, either $p'_1$ is a base point of $\rho_2$ or $p'_1$ is not a base point of $\rho_2$.

In the latter case, the map $\varphi_5\circ\rho_1^{-1}\circ\rho_2^{-1}$ would be a de Jonqui\`eres map of degree 4 and we get a contradiction with Lemma \ref{cordeJonq>d-1}.

In the former case, the map $\varphi_5\circ\rho_1^{-1}\circ\rho_2^{-1}$ would have the enriched weighted proximity graph of type 24 in Table \ref{table2} and its ordinary quadratic length would be 3, a contradiction.

Hence, we conclude that $\oq(\varphi_5) =5$.
\end{proof}

\begin{lemma}\label{oq sharp_4}
Let $\varphi_4$ be the map 4 in Table \ref{table1}. Then, $\oq(\varphi_4) =4.$
\end{lemma}

\begin{proof}
The enriched weighted proximity graph of $\varphi_{4}$ is listed in Table \ref{table2}.
Let $p_1$ be the double base point, $p_2, p_3, p_4,p_5$ the infinitely near simple base points such that $p_3\infnear_1 p_2\infnear_1 p_1$ and $p_5 \infnear_1 p_4 \infnear_1 p_1$.
Then,
$$3 \leqslant \oq(\varphi_4) \leqslant 4$$
because of the decomposition of $\varphi_{4}$ in Table \ref{table3} and the fact that the heights of $p_3$ and of $p_5$ with respect to $\varphi_{4}$ are 3, cf.\ Proposition  \ref{oq>height}.

Suppose by contradiction that $\oq(\varphi_4) =3$.
Then, there should exist an ordinary quadratic map $\rho$ such that $\oq(\varphi\circ\rho^{-1}) = 2$. In particular, $\rho$ must be based at $p_1$. Then, the map $\varphi\circ\rho^{-1}$ is a de Jonqui\`eres map of degree $4$ and we get a contradiction with Lemma \ref{cordeJonq>d-1}.
\end{proof}

\begin{lemma}\label{oq sharp_3}
Let $\varphi_3$ be the map 3 in Table \ref{table1}. Then, $\oq(\varphi_3) =5.$
\end{lemma}

\begin{proof}
The enriched weighted proximity graph of $\varphi_{3}$ is listed in Table \ref{table2}.
Let $p_1$ be the double base point, $p_2, p_3, p_4, p_5$ the infinitely near simple base points such that $p_2 \infnear_1 p_1$ and $p_5 \infnear_1 p_4 \infnear_1 p_3 \infnear_1 p_1$.
Then,
$$4 \leqslant \oq(\varphi_3) \leqslant 5$$
because of the decomposition of $\varphi_{3}$ in Table \ref{table3} and the fact that the height of $p_4$ with respect to $\varphi_{3}$ is 4, cf.\ Proposition  \ref{oq>height}.

Suppose by contradiction that $\oq(\varphi_3) =4$. Then, there should exist an ordinary quadratic map $\rho_1$ such that $\oq(\varphi_3\circ\rho^{-1}_1) = 3$. In particular, $\rho_1$ must be based at $p_1$ and not at a point lying on the line passing through $p_1$ and $p_3$, otherwise the maximum height of the base points with respect to $\varphi_3\circ\rho^{-1}_1$ would be still 4. Then, $\varphi_3\circ\rho_1^{-1}$ is a de Jonqui\`eres map of degree $4$ and its weighted proximity graph is:
\begin{equation}\label{prooftype3}
\scalebox{0.5}{%
\begin{tikzpicture}[->,>=stealth',shorten >=2pt,auto,node distance=2.5cm,ultra thick,baseline=-1.25ex]
  \tikzstyle{every state}=[fill=white,draw=black,text=black]

  \node[state,draw=red,text=red,label=below: $p'_0$] (A)              {\huge 3};
  \node[state,draw=red,text=red,label=below: $p'_1$] (B) [right of=A] {\huge 1};
  \node[state,draw=red,text=red,label=below: $p'_2$] (C) [right of=B] {\huge 1};
  \node[state,draw=red,text=red,label=below: $p'_3$] (D) [right of=C] {\huge 1};
  \node[state,draw=red,text=red,label=below: $p'_4$] (E) [right of=D] {\huge 1};
  \node[state,label=below: $p'_5$] (F) [right of=E] {\huge 1};
  \node[state,label=below: $p'_6$] (G) [right of=F] {\huge 1};

\path (G) edge (F);
\path (F) edge (E);
\end{tikzpicture}%
}
\end{equation}
where $p'_1, p'_2, p'_3, p'_4$ are aligned.

Then, there should exist an ordinary quadratic map $\rho_2$ such that $\oq(\varphi_3\circ\rho_1^{-1}\circ\rho_2^{-1}) =2$. The map $\rho_2$ must be based at $p'_4$, otherwise the maximum height of the base points of the map $\varphi_3\circ\rho_1^{-1}\circ\rho_2^{-1}$ would be at least $3$, a contradiction with Proposition \ref{oq>height}.
Furthermore, the map $\rho_2$ must be based also at $p'_0$, otherwise the degree of $\varphi_3\circ\rho_1^{-1}\circ\rho_2^{-1}$ would be larger than 4, a contradiction with Corollary \ref{cor>2}.

There are now two cases: either $\rho_2$ is based at $p'_i$, for some $i\in\{1,2,3\}$, or $\rho_2$ is not based at $p'_1.p'_2,p'_3$.

In the former case, the map $\varphi_3\circ\rho_1^{-1}\circ\rho_2^{-1}$ would have the enriched weighted proximity graph of type $14$ in Table \ref{table2}, a contradiction with Lemma \ref{oq sharp_14}.

In the latter case, the map $\varphi_3\circ\rho_1^{-1}\circ\rho_2^{-1}$ is a de Jonqui\`eres map of degree $4$, a contradiction with Lemma \ref{cordeJonq>d-1}.
 
Hence, we conclude that $\oq(\varphi_3) =5$.
\end{proof}

\begin{lemma}\label{oq sharp_1}
Let $\varphi_1$ be the map 1 in Table \ref{table1}. Then, $\oq(\varphi_1) =6.$
\end{lemma}

\begin{proof}
The enriched weighted proximity graph of $\varphi_{1}$ is listed in Table \ref{table2}.
Let $p_1$ be the double base point, $p_2, p_3, p_4, p_5$ the infinitely near simple base points such that $p_5 \infnear_1 p_4 \infnear_1 p_3 \infnear_1 p_2 \infnear_1 p_1$ with $p_3 \satel p_1$.
Then,
$$5 \leqslant \oq(\varphi_1) \leqslant 6$$
because of the decomposition of $\varphi_{1}$ in Table \ref{table3} and the fact that the height of $p_5$ with respect to $\varphi_{1}$ is 5, cf.\ Proposition  \ref{oq>height}.

Suppose by contradiction that $\oq(\varphi_1) =5$. Then, there should exist an ordinary quadratic map $\rho_1$ such that $\oq(\varphi_1\circ\rho^{-1}_1) = 4$. In particular, $\rho_1$ must be based at $p_1$ and not at a point lying on the line passing through $p_1$ and $p_2$, otherwise the maximum height of the base points with respect to $\varphi_1\circ\rho^{-1}_1$ would be still 5.
So the map $\varphi_1\circ\rho_1^{-1}$ is a de Jonqui\`eres map of degree $4$ and its weighted proximity graph is:
\begin{center}
\scalebox{0.5}{%
\begin{tikzpicture}[->,>=stealth',shorten >=2pt,auto,node distance=2.5cm,ultra thick,baseline=-1.25ex]
  \tikzstyle{every state}=[fill=white,draw=black,text=black]

  \node[state,draw=red,text=red,label=below: $p'_0$] (A)              {\huge 3};
  \node[state,draw=red,text=red,label=below: $p'_1$] (B) [right of=A] {\huge 1};
  \node[state,draw=red,text=red,label=below: $p'_2$] (C) [right of=B] {\huge 1};
  \node[state,draw=red,text=red,label=below: $p'_3$] (D) [right of=C] {\huge 1};
  \node[state,label=below: $p'_4$] (E) [right of=D] {\huge 1};
  \node[state,label=below: $p'_5$] (F) [right of=E] {\huge 1};
  \node[state,label=below: $p'_6$] (G) [right of=F] {\huge 1};

\path (G) edge (F);
\path (F) edge (E);
\path (E) edge (D);
\end{tikzpicture}%
}
\end{center}
where $p'_1, p'_2, p'_3, p'_4$ are aligned.

Then, there should exist an ordinary quadratic map $\rho_2$ such that $\oq(\varphi_1\circ\rho_1^{-1}\circ\rho_2^{-1}) =3$. In particular, the map $\rho_2$ must be based at $p'_3$ and not at $p'_1,p'_2$ (or at another point lying on the line passing through $p'_3$ and $p'_4$), otherwise the maximum height of the base points of the map $\varphi_1\circ\rho_1^{-1}\circ\rho_2^{-1}$ is 4, a contradiction with Proposition \ref{oq>height}.

There are now two cases: either $\rho_2$ is based at $p'_0$ or $\rho_2$ is not based at $p'_0$.

In the former case, the map $\varphi_1\circ\rho_1^{-1}\circ\rho_2^{-1}$ is a de Jonqui\`eres map of degree $4$ and its enriched weighted proximity graph is \eqref{prooftype3} and we reach a contradiction as in the proof of Lemma \ref{oq sharp_3}.

In the latter case, the map $\varphi_1\circ\rho_1^{-1}\circ\rho_2^{-1}$ has degree 7 and its weighted proximity graph is:
\begin{center}
\scalebox{0.5}{%
\begin{tikzpicture}[->,>=stealth',shorten >=2pt,auto,node distance=2.5cm,ultra thick,baseline=-1.25ex]
  \tikzstyle{every state}=[fill=white,draw=black,text=black]

  \node[state,draw=red,text=red,label=below: $p''_0$] (A)              {\huge 4};
  \node[state,draw=red,text=red,label=below: $p''_1$] (B) [right of=A] {\huge 3};
  \node[state,draw=red,text=red,label=below: $p''_2$] (C) [right of=B] {\huge 3};
  \node[state,draw=red,text=red,label=below: $p''_3$] (D) [right of=C] {\huge 3};
  \node[state,draw=red,text=red,label=below: $p''_4$] (E) [right of=D] {\huge 1};
  \node[state,draw=red,text=red,label=below: $p''_5$] (F) [right of=E] {\huge 1};
  \node[state,draw=red,text=red,label=below: $p''_6$] (G) [right of=F] {\huge 1};
  \node[state,label=below: $p''_7$] (H) [right of=G] {\huge 1};
  \node[state,label=below: $p''_8$] (I) [right of=H] {\huge 1};

\path (I) edge (H);
\path (H) edge (G);
\end{tikzpicture}%
}
\end{center}
where $p''_2, p''_3, p''_6$ are aligned and also $p''_0, p''_4, p''_5, p''_6$ are collinear.

Then, there should exist an ordinary quadratic map $\rho_3$ such that $\oq(\varphi_1\circ\rho_1^{-1}\circ\rho_2^{-1}\circ\rho_3^{-1}) =2$. Thus, $\rho_3$ must be based at $p''_6$, otherwise the maximum height of the base points of $\varphi_1\circ\rho_1^{-1}\circ\rho_2^{-1}\circ\rho_3^{-1}$ is 3, a contradiction with Proposition \ref{oq>height}.
This implies that $\varphi_1\circ\rho_1^{-1}\circ\rho_2^{-1}\circ\rho_3^{-1}$ would have degree $\geq6$, a contradiction with Corollary \ref{cor>2}.

Hence, we conclude that $\oq(\varphi) =6$.
\end{proof}

\begin{corollary}\label{oq(tau)}
Let $\tau$ be the quadratic map defined in \eqref{rhotau}.
Then, $\oq(\tau)=4$ and the decomposition \eqref{taudecomp} of $\tau$ is minimal.
\end{corollary}

\begin{proof}
Let $p_1, p_2, p_3$ be the base points of $\tau$, where $p_3\infnear_1 p_2\infnear_1 p_1\in\PP^2$ and let $\ell$ be the line through $p_1$ and $p_2$.
Proposition \ref{oq>height} implies that $\oq(\tau)\geq3$ and the decomposition \eqref{taudecomp} says that $\oq(\tau)\leq4$. Suppose by contradiction that $\oq(\tau)=3$.
Then, there exists an involutory ordinary quadratic map $\psi$ such that $\oq(\tau\circ\psi)=2$.

There are now two cases: either $p_1$ is a base point of $\psi$ or $p_1$ is not a base point of $\psi$.

In the latter case, $\tau\circ\psi$ has a base point of height 3, hence Proposition \ref{oq>height} implies $\oq(\tau\circ\psi)\geq3$, a contradiction.

In the former case, if one of the other two base points of $\psi$ lies on the line $\ell$, the map $\tau\circ\psi$ has still a base point of height 3 and we get again the same contradiction.
Otherwise, the map $\tau\circ\psi$ has the proximity graph of type 24 in Table \ref{table2}, which has ordinary quadratic length 3, according to Remark \ref{rem:20,22,24}, a contradiction.
\end{proof}

\end{document}